\documentclass[reqno,final]{amsart}
\usepackage{caption}
\usepackage{subcaption}
\usepackage[left=1in, right=1in, top=1in, bottom=1in]{geometry}

\usepackage{mathabx}
\usepackage[utf8]{inputenc} 
\usepackage[english]{babel}
\usepackage{amstext}
\usepackage{euscript}
\usepackage{bbm}
\usepackage{mathrsfs}
\usepackage{enumerate}
\usepackage{graphicx}
\usepackage{caption}
\usepackage{amscd}
\usepackage{verbatim}
\usepackage{amssymb}
\usepackage{amsthm}
\usepackage{amsmath}
\usepackage{amstext}
\usepackage{amsfonts}
\usepackage{latexsym}
\usepackage{mathtools} 
\usepackage{enumitem} 
\usepackage{accents} 
\usepackage[foot]{amsaddr} 

\usepackage[usenames,dvipsnames]{xcolor} 
\definecolor{dkblue}{RGB}{1,31,91} 
\usepackage[colorlinks=true, pdfstartview=FitV, linkcolor=dkblue, citecolor=dkblue, urlcolor=dkblue]{hyperref}    \usepackage{cite}

\usepackage{todonotes}

\usepackage[color]{showkeys}

\newcommand{\Kradj}{\mathcal{R}^{(j)}_{p}}

\newcommand{\Fj}{F^{(j)}}
\newcommand{\Fon}{F^{(1)}}

\newcommand{\Ftw}{F^{(2)}}

\newcommand{\Knj}{\mathcal{K}_{non.el}^{(j)}}
\newcommand{\Knon}{\mathcal{K}_{non.el}^{(1)}}
\newcommand{\Kntw}{\mathcal{K}_{non.el}^{(2)}}

\newcommand{\eqdef }{\overset{\mbox{\tiny{def}}}{=}}

\newcommand{\rth}{{\mathbb{R}^3}}

\newcommand{\stw}{{\mathbb{S}^2}}

\theoremstyle{definition}
\newtheorem{theorem}{Theorem}

\newtheorem{definition}{Definition}

\newtheorem{proposition}[theorem]{Proposition}
\newtheorem{remark}[theorem]{Remark}

\numberwithin{equation}{section}
\numberwithin{theorem}{section}
\numberwithin{definition}{section}

\begin{document}


\title[LTE and Non-LTE Solutions in Gases Interacting with Radiation]{LTE and Non-LTE Solutions in Gases Interacting with Radiation}

\author[J. W. Jang]{Jin Woo Jang}
\address{Department of Mathematics, Pohang University of Science and Technology (POSTECH), Pohang, South Korea (37673). \href{mailto:jangjw@postech.ac.kr}{jangjw@postech.ac.kr} }

\author[Juan J. L. Vel\'azquez]{Juan J. L. Vel\'azquez}
\address{Institute for Applied Mathematics, University of Bonn, 53115 Bonn, Germany. \href{mailto:velazquez@iam.uni-bonn.de}{velazquez@iam.uni-bonn.de} }

\begin{abstract}In this paper, we study a class of kinetic equations describing radiative transfer in gases which include also the interaction of the molecules of the gas with themselves. We discuss several scaling limits and introduce some Euler-like systems coupled with radiation as an aftermath of specific scaling limits. 
We consider scaling limits in which local thermodynamic equilibrium (LTE) holds, as well as situations in which this assumption fails (non-LTE). The structure of the equations describing the gas-radiation system is very different in the LTE and non-LTE cases.  
We prove the existence of stationary solutions with zero velocities to the resulting limit models in the LTE case. We also prove the non-existence of a stationary state with zero velocities in a non-LTE case.
\end{abstract}

\thispagestyle{empty}

\maketitle
\tableofcontents

\section{Introduction}
The goal of this paper is to formulate and to study several classes of kinetic equations which describe the interactions between radiations (photons) and the gas molecules. The main assumption of this model is that the gas molecules and the photons interact by means of binary collisions in which they exchange their momenta and kinetic energies. 

The main motivation for this study is to make precise the scaling limits in which the distributions of molecules and photons are in the state of non-local thermodynamic equilibrium (non-LTE). It is well-known that in many physical systems in every small macroscopic region the distributions of molecules and photons on it can be assumed to be at a thermodynamic equilibrium. However, in systems where radiative transfer takes place the local thermodynamic equilibrium (LTE) assumption can be lost due to the presence of incoming radiation (for instance at different temperature) or due to the fact that some of the radiation is escaping from the system.

\subsection{Underlying model}
The model that we will consider in this paper is a very simplified model of interactions between molecules and radiation. However, in this model we will be able to describe some of the properties of some non-LTE systems in precise mathematical terms.  In this model, we consider a physical situation that heat can transferred by means of convection and radiation.

The model that we consider is similar to the one considered in \cite{italians}. Specifically, we will assume that the molecules of the gas can be in two different states, the ground state and the excited state, which we will denote as $A$ and $\bar{A}$, respectively. We will assume also that the radiation is monochromatic and it consists of collection of photons with frequency $\nu_0>0.$ Therefore, all the photons of the system have the same energy $\epsilon_0 = h\nu_0$ where $h$ is the Planck constant. A more complicated model with three molecular states will be considered in Section \ref{sec.nonLTE.3level}.

We will assume that the molecules of the system can interact in three different ways:
\begin{enumerate}
    \item Elastic collisions between molecules: \begin{equation}\label{eq.elastic reactions}
	\begin{split}
	A+A &\rightleftarrows A+A,\\
	A+\bar{A} &\rightleftarrows A+\bar{A},\\
	\bar{A}+\bar{A} &\rightleftarrows \bar{A}+\bar{A}.
	\end{split}
\end{equation}
These are collisions between molecules in which the total kinetic energy and the total momentum are conserved. 
\item Nonelastic collisions: 
\begin{equation}\label{eq.nonelastic reactions}
	A+A\rightleftarrows A+\bar{A}.
\end{equation}These collisions are the collisions between two ground-state molecules in which one ground-state molecule is being excited. The reverse reaction can also take place. 
\item Collisions between a molecule and a photon:
\begin{equation}\label{eq.reactions2}
	A+\phi\rightleftarrows \bar{A}.
\end{equation}
Here a molecule in the ground state absorbs a photon via this reaction and becomes a molecule in the excited state. The reverse reaction can also take place.
\end{enumerate}
In \eqref{eq.reactions2} we assume that the momentum of $\phi$ is negligible and therefore the momentum of the excited molecule $\bar{A}$ is the same as that of the ground state $A$. On the other hand, since the energy of a photon is $\epsilon_0$ we have that the energy of the ground state $A$ with the velocity $v$ is $\frac{1}{2}|v|^2$ and the energy of the molecule in the excited state is equal to $\frac{1}{2 }|v|^2+\epsilon_0.$ We assume here and the rest of the paper that the mass $m$ of each molecule is equal to $1$. In principle, we could also include reactions with the form,
$$\bar{A}+A\rightleftarrows \bar{A}+\bar{A},$$
which, however, will be ignored in this paper.

Considering the nonelastic collisions \eqref{eq.nonelastic reactions}, we will assume that the total energy, including the energy $\epsilon_0 $ that is required to form an excited state, is conserved. Therefore, the incoming velocities of pre-collisional molecules, which we denote as $v_1$ and $v_2$, and the outgoing velocities of post-collisional molecules, which we denote as $v_3$ and $v_4$, satisfy

\begin{equation}\label{EnergyInel} \frac{1}{2}|v_1|^2+\frac{1}{2}|v_2|^2=\frac{1}{2}|v_3|^2+\frac{1}{2}|v_4|^2+\epsilon_0.\end{equation}
The total momentum of the molecules is conserved in this type of collision. Notice that the name ``nonelastic" collision here is a bit misleading, since the total energy is conserved. However, the total kinetic energy is not conserved here and therefore we will use this terminology throughout the paper.

An important simplification that we make in this paper is to assume that the radiation is fully monochromatic. This means that we assume that the width of the spectral lines is negligible. In particular, we neglect the increase of the width of the spectral lines due to Doppler effect. The relative increase of the thickness of the spectral lines is proportional to $v/c$ where $v$ is a characteristic velocity of the gas molecules and $c$ is the speed of light (cf. \cite{oxenius}). Therefore, the validity of the models considered in this paper is restricted to situations in which the gas molecules move at non-relativistic speeds.  

We are interested in studying the solutions of the previously discussed in situations in which the gas-radiation system is open to the exchange of radiation with the exterior of the system, but it cannot exchange gas molecules with the exterior of the system. Physical examples of systems that can be approximated in this manner are planetary and stellar atmospheres.

\subsection{LTE and Non-LTE situations}

Denote $\rho_1=\rho_1(t,x)$ and $\rho_2=\rho_2(t,x)$ the number densities of the ground and the excited states, respectively.  We say that the densities are in the Boltzmann ratio if at any point $x$ at any time $t$ they satisfy 
\begin{equation}\label{Boltzmanratio}\frac{\rho_2}{\rho_1}=e^{-\frac{\epsilon_0}{k_B T}},\end{equation} where $T$ is the local temperature. In \eqref{Boltzmanratio} we assume for simplicity that the degeneracies of the ground and excited states (denoted as $q_1$ and $q_2$ respectively) are equal to one. If this assumption is not made there would be an additional factor $\frac{q_2}{q_1}$ on the right hand side of \eqref{Boltzmanratio} (cf. \cite{oxenius}).

 We will say that the gas-radiation system is in Local Thermodynamic Equilibrium (LTE) if the number densities $\rho_1, \rho_2$ satisfy approximately the Boltzmann ratio \eqref{Boltzmanratio} and the distribution of velocities at each point can be approximated by a Maxwellian distribution (cf. \eqref{LocalMaxwellian}), with total molecule density $\rho=\rho_1+\rho_2$ with a local temperature $T=T(t,x)$ and a local velocity $u=u(t,x).$ More precisely, we define the local Maxwellian distribution $M$ for the molecule density $\rho, $ macroscopic fluid velocity $u$ and the local temperature $T$ as \begin{equation}\label{LocalMaxwellian}
     M(v;\rho,u,T)\eqdef \frac{c_0 \rho}{T^{3/2}}\exp\left( -\frac{1}{T}|v-u|^2\right),
 \end{equation}where $c_0=\pi^{-3/2}$. Then for a system where LTE holds, we have the following approximation: $$\Fon (t,x,v)=M(v;\rho_1,u,T),\ \Ftw (t,x,v)=e^{-\frac{2\epsilon_0}{T}} \Fon (v),$$ where $\Fon$ and $\Ftw$ stand for the molecule distributions of the ground and the excited states, and $\rho_1=\rho_1(t,x)$,$u=u(t,x)$, and $T=T(t,x)$.  Here and in the rest of the paper, we will fix the Boltzmann constant to be $k_B=\frac{1}{2}$ in order to get simpler formulas.
 
 Alternatively, we will say that the gas-radiation system is in a non-Local Thermodynamic Equilibrium (non-LTE) if the number densities are not in the Boltzmann ratio or their molecule velocity distributions are not the Maxwellian equilibria; i.e. if at least one of the identities (\ref{Boltzmanratio}) or (\ref{LocalMaxwellian}) fail, then we say that the system is in a non-LTE. However, in realistic physical situations, the failure of the local Maxwellian approximation is rare. Therefore, in this paper, we will restrict only to situations, in which the distributions of velocities are the Maxwellians, for each of the molecules $A$ and $\bar{A}$, but perhaps with different temperature $T_1$, $T_2$ for each of the species, and $\rho_1,\rho_2$ not satisfying \eqref{Boltzmanratio}. In this case, each species can have different temperatures; we denote the temperatures for the ground- and the excited-state molecules as $T_1$ and $T_2$, respectively. Under the assumption that the elastic collisions between the ground-state and the excited-state molecules are rare, we will show in Section \ref{sec.nonLTEnonexist} that there is in general no stationary solution to the system with zero macroscopic velocities where all the heat transfer takes place by means of radiation. Also, we will provide in Section \ref{sec.nonLTE.3level} a generalized model of three-level molecules that yields non-LTE stationary states.

In systems open to radiation, as the ones considered in this paper, we cannot expect to have LTE in general, due to the exchange of radiation with the exterior and the system is in general out of equilibrium.  Notice that in the definition of LTE we do not assume anything about the radiation itself, but only on the molecules of the gas. However, the interaction between the gas and radiation which is not in a thermodynamic equilibrium can result in the gas molecules not being in LTE.
The main driving force that tends to bring the system towards LTE are the ``nonelastic" collisions (cf. \eqref{eq.nonelastic reactions}). These compete with the radiation term which in some situations tends to create non-LTE distributions. Depending on the relative size of these terms we can have LTE or not. We discuss in this paper specific scaling limits yielding LTE and non-LTE distributions of molecules.

\subsection{A summary of mathematical results for the radiative transfer equations}

Some of the earliest models describing the interaction between matter and radiation were the ones introduced in \cite{Compton,Milne}. These papers contain the first models for the diffusion of ``imprisoned" radiation which is trapped in a gas due to the fact that the radiation photons are absorbed and re-emitted by the gas molecules.

The equations of radiative transfer as well as many exact and approximated solutions of these equations in settings that are relevant to the study of planetary atmospheres are extensively discussed in \cite{mihalas2013foundations,Chandrasekhar,mihalas1978stellar,rutten1995radiative,oxenius}. In \cite{Golse-Bardos,B2}, the well-posedness theory and some asymptotic limits for the radiative transfer equations are considered. Additional properties of the radiative transfer equation including the validity of so-called the Rosseland approximation have been studied in \cite{B1,B3}.

The models considered in \cite{Compton,Chandrasekhar,Milne} do not take into account in an explicit manner the collisions between the gas molecules themselves. More precisely, collisions between gas molecules which can result in excited gas states without the intervention of the radiation are not taken into account. (See for instance (\ref{eq.nonelastic reactions}) for an example of this type of collisions). In the analysis of this paper, this type of collisions plays an essential role determining the relative ratios of molecules in the excited state and in the ground state and in particular to determine if the LTE property (\ref{Boltzmanratio}) is approximately satisfied or not.

Kinetic models which describe the interaction between a gas and radiation can be found in \cite{oxenius}. Some of these models have been used to study the physical conditions yielding LTE and non-LTE. On the other hand, the model characterized by the collisions (\ref{eq.elastic reactions}), (\ref{eq.nonelastic reactions}), (\ref{eq.reactions2}) has been introduced in \cite{RSM}. In \cite{italians} this model has been reduced to a generalized system of Euler equations, which include the effects of the radiation using a zeroth order Chapman-Enskog approximation. The well posedness of the resulting system has been also established in \cite{italians} assuming that the boundary conditions are those corresponding to the so-called Milne-Chandrasekhar problem, i.e. the diffusion of imprisoned radiation through the walls of the system. Additional mathematical properties of the model introduced in \cite{RSM} have been studied in \cite{Nouri1, Nouri2} (albeit with a different type of boundary conditions as the one considered in \cite{italians}).

An appealing property of the model with collisions (\ref{eq.elastic reactions}), (\ref{eq.nonelastic reactions}), (\ref{eq.reactions2}) introduced in \cite{RSM} is that it yields a system of equations simple enough to allow to obtain precise mathematical statements. In particular, we will be able to introduce precise scaling limits yielding either LTE or non-LTE. We will see that different scaling limits can result in models that we call an Euler-like system coupled with the radiation equation different from the one obtained in \cite{italians}.

\subsection{Kinetic equations for the gas-radiation system.}\label{sec.kinetic system}
In this paper, we consider several formal scaling limits and obtain several limiting models for the kinetic model above via the Chapman-Enskog expansions. A similar approach has been introduced in \cite{italians}. The goal of this paper is to introduce several scaling limits yielding LTE and non-LTE.

We will formulate the set of equations which describe a gas-radiation system characterized by the set of collisions (\ref{eq.elastic reactions}), (\ref{eq.nonelastic reactions}), (\ref{eq.reactions2}). We will assume that the gas is sufficiently dilute and the radiation weak enough to neglect the ternary or higher order collisions.

We first determine the velocities of the outgoing molecules in terms of the incoming molecules in collisions of the form (\ref{eq.elastic reactions}), (\ref{eq.nonelastic reactions}). In both types of collisions the total momentum is conserved.   We will denote the velocities of the incoming molecules as $v_1,v_2$ and the velocity of the outgoing molecules as $v_3, v_4$ for the elastic collisions and $\bar{v}_3,\bar{v}_4$ for the nonelastic collisions where in the case of the collisions (\ref{eq.nonelastic reactions}) we will assume that $\bar{v}_3$ is the velocity of the molecule in the excited state $\bar{A}$. The conservation of momentum then yields

\begin{equation}\label{ConsMoment}
    v_1+v_2=v_3+v_4.
 \end{equation}
In the case of the collisions (\ref{eq.elastic reactions}), the molecule velocities satisfy in addition the conservation of kinetic energy
\begin{equation}\label{ConsEnergy}
    \frac{|v_1|^2}{2}+\frac{|v_2|^2}{2}=\frac{|v_3|^2}{2}+\frac{|v_4|^2}{2}.
 \end{equation}
On the contrary, in the case of collisions with the form (\ref{eq.nonelastic reactions}) the energy formula (\ref{EnergyInel}) holds. 
Combining (\ref{ConsMoment}) and (\ref{ConsEnergy}) we obtain the following formula for the velocities of the outgoing molecules in terms of the incoming molecules

\begin{equation}\label{VelocConsCase}
   v_3=\frac{v_1+v_2}{2}+\frac{|v_1-v_2|}{2}\omega\text{  and  } v_4=\frac{v_1+v_2}{2}-\frac{|v_1-v_2|}{2}\omega
 \end{equation}
 where $\omega\in \stw$.

In the case of the nonelastic collisions, it follows from (\ref{EnergyInel}) and (\ref{ConsMoment}) that the post-collisional velocities $\bar{v}_3, \bar{v}_4$ given the pre-collisional velocities $\bar{v}_1$ and $\bar{v}_2$ have the form:

\begin{equation}\label{VelocInelCase}
  \bar{v}_3=\frac{\bar{v}_1+\bar{v}_2}{2}+\omega\sqrt{\frac{|\bar{v}_1-\bar{v}_2|^2}{4}-\epsilon_0}\text{  , } \bar{v}_4=\frac{\bar{v}_1+\bar{v}_2}{2}-\omega\sqrt{\frac{|\bar{v}_1-\bar{v}_2|^2}{4}-\epsilon_0} \text{,    } \omega\in \stw.
 \end{equation}

\subsubsection{Equation for the radiation energy}\label{sec.eq.rad}
Concerning the distribution of photons, since we are neglecting the Doppler effect, we assume that the radiation is monochromatic, and the only degree of freedom of the photons is their direction of motion. 
We will denote the distributions for the ground and the excited states as $\Fon=\Fon(t,x,v)$ and $\Ftw=\Ftw(t,x,v)$, respectively. On the other hand, we denote the intensity of the radiation at the frequency $\nu$ as $I_\nu=I_\nu(t,x,n)$ where $n\in\stw$. Then the radiative transfer equation for the radiation intensity $I_\nu$ was introduced in  \cite[Section 6.3]{oxenius} and is given by
\begin{equation}\label{eq.radiation.proto}
	\frac{1}{c}\frac{\partial I_\nu}{\partial t} + n\cdot \nabla_x I_\nu \\=\frac{\epsilon_0}{4\pi}\int_\rth dv\ \left[ A_{21}\Ftw (v)+B_{21}\Ftw (v)I_\nu(n)-B_{12}\Fon(v) I_\nu(n)\right],
\end{equation} 
where $A_{21}$ is the Einstein coefficient for spontaneous emission,  $B_{21}$ is the Einstein coefficient for stimulated emission, and $B_{12}$ is the Einstein coefficient for absorption (c.f. \cite{oxenius}). These coefficients are related by means of the Einstein relations namely
$$\frac{A_{21}}{B_{21}}=\frac{2h\nu^3}{c^2},\ \frac{B_{12}}{B_{21}}=1,$$ which are a consequence of the fact that at equilibrium,  detailed balance holds as well as the fact that the equilibrium distribution is the Planck distribution. Then \eqref{eq.radiation.proto} reduces to  \begin{equation}\label{eq.radiation}
	\frac{1}{c}\frac{\partial I_{\nu_0}}{\partial t} + n\cdot \nabla_x I_{\nu_0} =\frac{\epsilon_0 B_{12}}{4\pi}\int_\rth dv\ \left[ \frac{2h\nu_0^3}{c^2}\Ftw (v) \left(1+\frac{c^2}{2h\nu_0^3}I_{\nu_0}\right)-\Fon(v) I_{\nu_0}(n)\right].
\end{equation}
    Here we assume that $B_{12}$ is constant. Notice that in \eqref{eq.radiation} the emission and the absorption of photons are isotropic. 
Further detailed physical descriptions on the equation including more complicated nonisotropic cases can be found in \cite{oxenius}.

The equation for the photon number density $Q=Q(t,x,n;\nu_0)$ is then
\begin{equation}\notag
	\frac{\partial Q}{\partial t} + cn\cdot \nabla_x Q =\frac{B_{12}}{4\pi}\int_\rth dv\ \left[ \frac{2h\nu_0^3}{c^2}\Ftw (v)\left(1+\frac{c^3}{2\nu_0^2}Q\right)-\Fon (v) c\epsilon_0 Q\right],
\end{equation} since the photon energy density $I_{\nu_0}$ can be written as 
$I_{\nu_0}=c\epsilon_0 Q$  (cf. \cite{oxenius,Chandrasekhar}).


\subsubsection{Kinetic equations for the gas molecules}\label{sec.Kinetic eqs}
We now introduce the kinetic equations for the molecules.   We first note that the gain of the excited molecules by the photon absorption, which is the same as the loss term for molecules in the ground state, is given by
$$\frac{ B_{12}}{4\pi}\int_{\mathbb{S}^2} dn\ \Fon  I_{\nu_0}
.$$
Also, the gain of the ground-state molecules or the loss of the excited state due to the emission of photons is given by
$$\frac{ B_{12}}{4\pi}\int_{\mathbb{S}^2} dn\  \frac{2h\nu_0^3}{c^2}\Ftw \left(1+\frac{c^2}{2h\nu_0^3}I_{\nu_0}\right).$$ See (6.3.8) of \cite{oxenius}. 
Then the evolution equations for $F=(\Fon,\Ftw)^\top $ which take all the interactions \eqref{eq.elastic reactions}-\eqref{eq.reactions2} into account are then given by
\begin{equation}\label{eq.kinetic1}
	\frac{D\Fon }{Dt}=\mathcal{K}_{el}^{(1,1)}[\Fon ,\Fon ]+\mathcal{K}_{el}^{(1,2)}[\Fon ,\Ftw ]+\Knon[F,F]+\int_{\mathbb{S}^2} dn\ h_{rad}[\Fon ,\Ftw ,Q],
\end{equation}and 
\begin{equation}\label{eq.kinetic2}
	\frac{D\Ftw }{Dt}=\mathcal{K}_{el}^{(2,1)}[\Ftw ,\Fon ]+\mathcal{K}_{el}^{(2,2)}[\Ftw ,\Ftw ]+ \Kntw[F,F]-\int_{\mathbb{S}^2} dn\ h_{rad}[\Fon ,\Ftw ,Q],
\end{equation}
 with the equation for $Q=Q(t,x,n)$ which can also be written as
\begin{equation}\label{eq.kineticphoton}
	\frac{\partial Q}{\partial t} + cn\cdot \nabla_x Q=\int_\rth dv\ h_{rad}[\Fon ,\Ftw ,Q],
\end{equation}
where $$ h_{rad}[\Fon ,\Ftw ,Q]=\frac{B_{12}}{4\pi} \left[\frac{2h\nu_0^3}{c^2}\Ftw (v)\left(1+\frac{c^3}{2\nu_0^2}Q\right)-\Fon (v) c\epsilon_0 Q \right].$$
 Here we denote the collision terms due to the elastic collisions and the nonelastic collisions as $\mathcal{K}^{(i,j)}_{el}$ and $\mathcal{K}^{(j)}_{non.el}$,  respectively. 
Here, the elastic collision terms $\mathcal{K}^{(i,j)}_{el}$ are given as the standard classical Boltzmann collision operators
\begin{equation}\label{eq.kel}\mathcal{K}^{(i,j)}_{el} [F,G](v_1)
\eqdef\int_\rth dv_2 \int_{\mathbb{S}^2} d\omega \ B^{(i,j)}_{el}(|v_1-v_2|,(v_1-v_2)\cdot\omega)(F(v_3)G(v_4)-F(v_1)G(v_2)),\end{equation} for $i,j=1,2$ where the post-collisional velocities $v_3$ and $v_4$ are given by \eqref{VelocConsCase}. 
Also, by Appendix \ref{sec.weakfor derivation}, the nonelastic operators are defined as 
\begin{equation}\label{eq.knonel1}\Knon[F,F]\eqdef 2\mathcal{K}_{1,1}[F,F]+\mathcal{K}^{(1)}_{1,2}[F,F],\end{equation}
\begin{multline}\label{eq.knonel11}\mathcal{K}_{1,1}[F,F](\bar{v}_1)\eqdef \int_\rth d\bar{v}_2 \int_\stw d\omega \frac{\sqrt{|\bar{v}_1-\bar{v}_2|^2-4\epsilon_0}}{2|\bar{v}_1-\bar{v}_2|}\\\times B_{non.el}(|\bar{v}_1-\bar{v}_2|,\omega\cdot (\bar{v}_1-\bar{v}_2)) (\bar{F}_3^{(2)}\bar{F}_4^{(1)}-\bar{F}^{(1)}_1 \bar{F}_2^{(1)}),\end{multline}
\begin{multline}\label{eq.knonel12}\mathcal{K}^{(1)}_{1,2}[F,F](\bar{v}_4)\eqdef \int_\rth d\bar{v}_3 \int_\stw d\omega \frac{\sqrt{|\bar{v}_3-\bar{v}_4|^2+4\epsilon_0}}{2|\bar{v}_3-\bar{v}_4|}\\\times B_{non.el}(|\bar{v}_3-\bar{v}_4|,\omega\cdot (\bar{v}_3-\bar{v}_4))  (\bar{F}^{(1)}_1\bar{F}^{(1)}_2-\bar{F}^{(2)}_3\bar{F}^{(1)}_4),\end{multline}and
\begin{multline}\label{eq.knonel2}\Kntw[F,F](\bar{v}_3)\eqdef \int_\rth d\bar{v}_4 \int_\stw d\omega \frac{\sqrt{|\bar{v}_3-\bar{v}_4|^2+4\epsilon_0}}{2|\bar{v}_3-\bar{v}_4|}\\\times B_{non.el}(|\bar{v}_3-\bar{v}_4|,\omega\cdot (\bar{v}_3-\bar{v}_4))   (\bar{F}^{(1)}_1\bar{F}^{(1)}_2-\bar{F}^{(2)}_3\bar{F}^{(1)}_4).\end{multline} The Jacobian is due to the fact that the mapping $(\bar{v}_1,\bar{v}_2)\to (\bar{v}_3,\bar{v}_4)$ is non-sympletic for the non-elastic collisions.  Throughout the paper, we assume that all the collisional cross-sections $B_{el}$ and $B_{non.el}$ have an angular cutoff.
In the case of hard-sphere interactions, we have (c.f. \eqref{non.el.cross.section})
\begin{equation}\notag
B_{non.el}(|\bar{v}_3-\bar{v}_4|,\omega\cdot (\bar{v}_3-\bar{v}_4))= C_0|\omega\cdot (\bar{v}_3-\bar{v}_4)| \text{ (or simply} =C_0| \bar{v}_3-\bar{v}_4|),
\end{equation} for some $C_0>0.$  We remark that this explicit hard-sphere assumption is not necessary for the derivations in the case of LTE as long as it is with an angular cutoff.   Here  $F_j^{(k)}=F^{(k)}(v_j)$ and the operator $\mathcal{K}_{i,j}$ refers to the nonelastic operator that has a collision between molecules $i$ and $j$ as its loss term and $\mathcal{K}_{i,j}^{(k)}$ affects the change $\partial_t F^{(k)}.$ 
We also used the shorthand notations $F^{(i)}_k=F^{(i)}(v_k)$ and $\bar{F}^{(i)}_k=F^{(i)}(\bar{v}_k)$ for $i=1,2$ and $k=1,2,3,4.$

The derivation of the nonelastic collision operators \eqref{eq.knonel1}, \eqref{eq.knonel11}, \eqref{eq.knonel12} and \eqref{eq.knonel2} via the weak formulation of the following section below will be given in Appendix \ref{sec.weakfor derivation}.

\subsection{Outline of the paper}
The goal of this paper is to provide different scaling limits that yield different limiting system of hydrodynamic and radiation equations in both  LTE and non-LTE situations. We derive in different asymptotic limit models which consist in a set of Euler-like equations for the gas, coupled with the equation for the radiation both in LTE and non-LTE situations in Section \ref{sec.LTEEulerlimit} and Section \ref{sec.Euler nonLTE}.  Different regimes arise from the relevance of the non-elastic collisions.
Then we will also study the well-posedness of the boundary value problems for the stationary Euler equations under several physical boundary conditions for the gas molecules and the radiation in Section \ref{sec.LTEexist} and Section \ref{sec.nonLTEexist}. We show in Section \ref{sec.LTEexist} and Section \ref{sec.nonLTEexist} that in stationary situations the radiation equation determines uniquely the distribution of temperatures in the particular scaling limits considered in this paper. In Section \ref{sec.nonLTEnonexist}, we will assume that the elastic collisions between the ground-state and the excited-state take place less often so that there is not enough mixture of temperature and velocities. In this case we prove that there is no stationary solutions with zero velocities to the Euler-like system. Then in section \ref{sec.nonLTE.3level} we discuss more complicated Euler-like system with molecules of three different levels and obtain non-LTE steady states without any additional assumptions on the scales of elastic collisions among molecules.   In the Appendix, we introduce how to derive the explicit forms of the nonelastic Boltzmann operator of \eqref{eq.knonel1}-\eqref{eq.knonel12} for the reaction \eqref{eq.nonelastic reactions} via the weak formulation \eqref{weak.formulation}. Also, we introduce the detailed balance via the Planck distribution as the equilibrium.

\begin{remark}
  We remark that it is possible to derive the Navier-Stokes-Fourier limit using well-established methods of kinetic theory. However, in most of the applications in fields like astrophysics, the Euler approximation is more relevant. 
\end{remark}

\section{Euler limit yielding LTE for the gas molecules}\label{sec.LTEEulerlimit}
In this section we derive a system of equations describing the evolution the radiation-gas systems under consideration. The resulting system, which will be derived applying the Chapman-Enskog method to a suitably rescaled version of the kinetic system \eqref{eq.radiation}-\eqref{eq.kinetic2} for gas molecules and radiation, is a Euler system of equations for a compressible gas coupled with the radiation equation.

\subsection{Rescaled kinetic equations for molecules and photons}
\label{sec.rescaled.kinetic}
We now compare the size of the different collision terms in \eqref{eq.radiation}-\eqref{eq.kinetic2}. This will allow us to formulate precise scaling limits for which the solutions of the equations exhibit different behaviours. We introduce also a suitable system of units in order to reformulate the problem in a simple form. 

Suppose that we denote the order of magnitude of the cross section for the non-elastic collisions as $\|B_{non.el}\|$ and the order of magnitude of the cross section associated to the elastic collisions as $\|B_{el}\|$.
In this subsection, we will rescale the variables and rewrite the system of equations \eqref{eq.radiation}-\eqref{eq.kinetic2} in a simpler way. We denote $\alpha$ as the mean-free path. Then we rescale the variables $t\mapsto t'$ and $x\mapsto y$ and the kernels $\|B_{el}\|\mapsto \alpha\|B_{el}\|$ and $\|B_{non.el}\|\mapsto \frac{\alpha}{\eta}\|B_{non.el}\|$ with some constant ratio $\eta$ between elastic and nonelastic collisions such that $\frac{B_{12}}{4\pi }\approx 1$ and $\frac{2h\nu_0^3}{c^2}\approx1$, and the system of equations \eqref{eq.radiation}-\eqref{eq.kinetic2} now becomes
\begin{equation}\label{eq.rescaled kinetic equations}
    \begin{split}
        &\frac{\partial}{\partial t'}F+ v\cdot \nabla_y F=\frac{1}{\alpha} \left(\mathcal{K}_{el}[F,F]+\eta \mathcal{K}_{non.el}[F,F]\right)+\mathcal{R}_p[F,G],\\&\frac{1}{c}\frac{\partial}{\partial t'}G+n\cdot \nabla_y G=\mathcal{R}_r[F,G],
    \end{split}
\end{equation}where 
$F=(\Fon,\Ftw)^\top$ is the rescaled vector solution for the distributions of molecules, $G\eqdef\frac{c^2}{2h\nu_0^3}I$  is the rescaled distribution for the photon intensity $I$,
   \begin{equation}\label{eq.Rp}\mathcal{R}_p[F,G]\eqdef \begin{pmatrix}\int_{\mathbb{S}^2}[\Ftw(1+G)-\Fon G]dn\\
   -\int_{\mathbb{S}^2}[\Ftw(1+G)-\Fon G]dn\end{pmatrix},\end{equation}\begin{equation}\label{eq.Rr}\mathcal{R}_r[F,G]\eqdef \epsilon_0\int_{\rth}[\Ftw(1+G)-\Fon G]dv,\end{equation} and
\begin{multline*}\mathcal{K}[F,F]\eqdef \mathcal{K}_{el}[F,F]+\eta \mathcal{K}_{non.el}[F,F]\\=\begin{pmatrix}\mathcal{K}^{(1,1)}_{el}[\Fon,\Fon]+\mathcal{K}^{(1,2)}_{el}[\Fon,\Ftw]+2\eta \mathcal{K}_{1,1}[F,F]+\eta \mathcal{K}_{1,2}^{(1)}[F,F]\\\mathcal{K}^{(2,1)}_{el}[\Ftw,\Fon]+
\mathcal{K}^{(2,2)}_{el}[\Ftw,\Ftw]+\eta \Kntw[F,F]\end{pmatrix},\end{multline*}where the operators $\mathcal{K}_{el}$ and $\mathcal{K}_{non.el}$ are defined in \eqref{eq.kel}-\eqref{eq.knonel2}.
By defining 
\begin{equation}\label{W-def} W_- (v_1,v_2;v_3,v_4)\eqdef \frac{\sqrt{|v_1-v_2|^2-4\epsilon_0}}{2|v_1-v_2|}B_{non.el}(|v_1-v_2|,\omega\cdot (v_1-v_2)),\end{equation} and 
\begin{equation}\label{W+def} W_+ (v_3,v_4;v_1,v_2)\eqdef \frac{\sqrt{|v_3-v_4|^2+4\epsilon_0}}{2|v_3-v_4|}B_{non.el}(|v_3-v_4|,\omega\cdot (v_3-v_4)),\end{equation}we can rewrite the operators as\begin{equation}\label{eq.K11}\mathcal{K}_{1,1}[F,F](v)=\int_\rth dv_2 \int_\stw d\omega\ W_- (v,v_2;v_3,v_4) (F_3^{(2)}F_4^{(1)}-\Fon F_2^{(1)}),\end{equation}
\begin{equation}\label{eq.K112}\mathcal{K}^{(1)}_{1,2}[F,F](v)= \int_\rth dv_3 \int_\stw d\omega\ W_+ (v_3,v;v_1,v_2) (F^{(1)}_1F^{(1)}_2-F^{(2)}_3F^{(1)}),\end{equation}and
\begin{equation}
    \label{Kntw}\Kntw[F,F](v)= \int_\rth dv_4 \int_\stw d\omega\ W_+ (v,v_4;v_1,v_2)(F^{(1)}_1F^{(1)}_2-F^{(2)}F^{(1)}_4).
\end{equation}At equilibrium there is a cancellation of the terms associated to individual collisions and the detailed balance holds. 
\begin{remark}A reasonable approximation in most physical applications is to let $c\to\infty$ (under the assumption that the velocities of the molecules are of order one). In our scaling limit, we have $\frac{h\nu_0^3}{c^2}\approx 1$ with given $h\nu_0=\epsilon_0\approx 1$. Thus, we have $\nu_0\approx c \to \infty$ and $h\approx \frac{1}{\nu_0}\to 0^+$ in the limit. Then the system \eqref{eq.rescaled kinetic equations} becomes the following system of a time-dependent evolution equations for the gas molecules and a stationary equation for the radiation:
\begin{equation}\label{stationary equations}
    \begin{split}
        &\partial_t F+ v\cdot \nabla_y F=\frac{1}{\alpha} \left(\mathcal{K}_{el}[F,F]+\eta \mathcal{K}_{non.el}[F,F]\right)+\mathcal{R}_p[F,G],\\&n\cdot \nabla_y G=\mathcal{R}_r[F,G].
    \end{split}
\end{equation}Also the limit $c\to \infty$ implies that the broadening of the spectral lines due to the Doppler effect in the collisions are negligible.\end{remark}

 \subsection{Chapman-Enskog Expansions} 
 We consider \eqref{stationary equations} in the limit $\alpha\to 0^+$ with $\eta$ of order one. Then we expect LTE because the non-elastic collisions bring the distribution of gas molecule velocities to LTE.  We consider the linearization of the system \eqref{stationary equations} around the equilibrium solution $F_{eq}=(F_{eq}^{(1)},F_{eq}^{(2)})^\top$ which is defined via the local Maxwellians
 \begin{equation}\label{Feq}F_{eq}^{(j)}=F_{eq}^{(j)}(\rho,u,T)\eqdef \frac{c_0 \rho}{T^{3/2}}\exp \left(-\frac{1}{T}\left(|v-u|^2+2\epsilon_0\delta_{j,2}\right)\right),\ j=1,2,\end{equation} where 
the Boltzmann constant in the rest of the paper is chosen as $\frac{1}{2}$, $$\delta_{j,2}\eqdef 1 \text{ if }j=2,\text{ and }=0,\text{ otherwise}.$$ Also, $\rho=\rho(t,x)$ is not the total molecule density but only the density of the ground-state molecule $(j=1)$ for the sake of simpler computations. 
 Here $c_0$ is the normalization constant $c_0=\left(\int_\rth e^{-|\xi|^2}d\xi\right)^{-1}=\pi^{-3/2}>0$, $u=u(t,x)$ is the macroscopic mean velocity, and $T=T(t,x)$ is the temperature, which are defined as
\begin{equation}
    \begin{split}
        &\rho\eqdef \int_\rth F^{(1)}dv,\\
         &u_{i}\eqdef \frac{1}{\rho}\int_\rth v_i F^{(1)}dv,\text{ for each }i=1,2,3,\\
         &T\eqdef \frac{2}{3\rho}\int_\rth |v-u|^2F^{(1)}dv.
    \end{split}
\end{equation}Note that this equilibrium solution $F_{eq}$ satisfies that 
 $$\mathcal{K}[F_{eq},F_{eq}]=0.$$

Then we consider the Chapman-Enskog expansion with the mean-free path $\alpha$ such that
\begin{equation}
    \label{Chapman}\Fj=F_{eq}^{(j)}(1+f^{(j)})=F_{eq}^{(j)}(1+\alpha f_1^{(j)}+\alpha^2 f_2^{(j)}+\cdots).
\end{equation} This choice of the expansion will result in a self-adjoint operator acting on $f$.
Then we linearize the operator $\mathcal{K}$ as
$$\mathcal{K}[F,F]=\mathcal{K}[F]=
\mathcal{K}[F_{eq}(1+\alpha f_1+\alpha^2 f_2+\cdots)]=\alpha L[F_{eq};f_1]+O(\alpha^2),$$
where $$L[F_{eq};f_1]\eqdef \begin{pmatrix}L^{(1,1)}_{el}[f_1]+L^{(1,2)}_{el}[f_1]+2\eta L_{1,1}[f_1]+\eta L_{1,2}^{(1)}[f_1]\\L^{(2,1)}_{el}[f_1]+
L^{(2,2)}_{el}[f_1]+\eta L^{(2)}_{non.el}[f_1]\end{pmatrix}.$$
Here each linear operator are defined as
\begin{equation}\label{def.linear operator}
\begin{split}
L^{(i,j)}_{el}[f_1]&\eqdef\mathcal{K}^{(i,j)}_{el}[F_{eq}^{(i)},F_{eq}^{(j)}f_1^{(j)}]+\mathcal{K}^{(i,j)}_{el}[F_{eq}^{(i)}f_1^{(i)},F_{eq}^{(j)}],\\
L_{(1,1)}[f_1]&\eqdef\mathcal{K}_{(1,1)}[F_{eq},F_{eq}f_1]+\mathcal{K}_{(1,1)}[F_{eq}f_1,F_{eq}],\\
L^{(1)}_{1,2}[f_1]&\eqdef\mathcal{K}^{(1)}_{1,2}[F_{eq},F_{eq}f_1]+\mathcal{K}^{(1)}_{1,2}[F_{eq}f_1,F_{eq}],\text{ and }\\
L^{(2)}_{non.el}[f_1]&\eqdef\mathcal{K}^{(2)}_{non.el}[F_{eq},F_{eq}f_1]+\mathcal{K}^{(2)}_{non.el}[F_{eq}f_1,F_{eq}].
\end{split}
\end{equation}
Then we expand each operator as a power series of $\alpha$ as 
$$\mathcal{K}^{(i,j)}_{el}[F^{(i)},F^{(j)}]=\mathcal{K}^{(i,j)}_{el}[F^{(i)}_{eq},F^{(j)}_{eq}]+\alpha L^{(i,j)}_{el}[f_1]+\cdots,$$
$$\mathcal{K}_{1,1}[F,F]=\mathcal{K}_{1,1}[F]=\mathcal{K}_{1,1}[F_{eq}]+\alpha L_{1,1}[f_1]+\cdots,$$
$$\mathcal{K}^{(1)}_{1,2}[F,F]=\mathcal{K}^{(1)}_{1,2}[F]=\mathcal{K}_{1,2}^{(1)}[F_{eq}]+\alpha L_{1,2}^{(1)}[f_1]+\cdots,$$
and
$$\Kntw[F,F]=\Kntw[F]=\Kntw[F_{eq}]+\alpha L_{1,2}^{(2)}[f_1]+\cdots.$$
Note that the detailed balance conditions imply that
\begin{equation}\label{K.null.eq2}\mathcal{K}^{(i,j)}_{el}[F^{(i)}_{eq},F^{(j)}_{eq}]=\mathcal{K}_{1,1}[F_{eq}]=\mathcal{K}^{(1)}_{1,2}[F_{eq}]=\Kntw[F_{eq}]=0.
\end{equation}
In the rest of the section, we will use the standard $L^2 $ inner product with all the two species taken into account: 
\begin{equation}\label{eq.L2 inner product} \left\langle F,G\right\rangle \eqdef\sum_{j=1}^2 \int_\rth dv \ \Fj\cdot G^{(j)}.\end{equation}
This type of the inner product will leads us to the hydrodynamic limit in the LTE situation.
\subsection{Kernel of $L$}We now consider the kernel of $L$. To this end,
 we fix $( \rho_0, u_0, T_0)$ and consider the Taylor expansion of $F_{eq}$ around the point $( \rho_0, u_0, T_0)$. Define 
$ F_{eq,0}=F_{eq}( \rho_0, u_0, T_0).$
By \eqref{K.null.eq2}, we first have
$\mathcal{K}[F_{eq}]=0.$ Thus,
\begin{multline*}
    0=\mathcal{K}[F_{eq}]
    =\mathcal{K}\left[ F_{eq,0}+\frac{\partial  F_{eq,0}}{\partial \rho}(\rho- \rho_0) +\frac{\partial  F_{eq,0}}{\partial u}(u- u_0)+\frac{\partial  F_{eq,0}}{\partial T}(T- T_0)+\cdots\right]\\
    =\mathcal{K}\left[ F_{eq,0}\right]+L\left[\frac{\partial  F_{eq,0}}{\partial \rho}(\rho- \rho_0) +\frac{\partial  F_{eq,0}}{\partial u}(u- u_0)+\frac{\partial  F_{eq,0}}{\partial T}(T- T_0)\right]+\cdots\\
    =L\left[\frac{\partial  F_{eq,0}}{\partial \rho}(\rho- \rho_0) +\frac{\partial  F_{eq,0}}{\partial u}(u- u_0)+\frac{\partial  F_{eq,0}}{\partial T}(T- T_0)\right]+\cdots.
\end{multline*}
Thus, we conclude
$$L\left[\frac{\partial  F_{eq,0}}{\partial \rho}\right]=L\left[\frac{\partial  F_{eq,0}}{\partial u}\right]=L\left[\frac{\partial  F_{eq,0}}{\partial T}\right]=0.$$
Note that 
$$\frac{\partial  F_{eq,0}}{\partial \rho}(\rho_0,u_0,T_0)=\frac{ F_{eq,0}}{ \rho_0},\ \frac{\partial  F_{eq,0}}{\partial u}(\rho_0,u_0,T_0)=\frac{2}{ T_0} F_{eq,0} (v- u_0),
$$
and
$$\frac{\partial  F^{(j)}_{eq,0}}{\partial T}(\rho_0,u_0,T_0)= F^{(j)}_{eq,0}\left[-\frac{3}{2}\frac{1}{ T_0}+\frac{|v- u_0|^2+2\epsilon_0\delta_{j,2}}{ T_0^2}\right],\ j=1,2.$$
Since the multiplicative constants can be removed, we thus obtain the kernel of $L$ as
\begin{equation}\label{kernel.L}\ker(L)=\text{span} \{ F_{eq,0},  F_{eq,0}(v- u_0),  F_{eq,0}|v- u_0|^2\}.\end{equation}

\subsection{Derivation of a family of Euler equations coupled with radiation}\label{sec.Euler limit}
By postulating the Chapman-Enskog expansion in \eqref{Chapman} we observe that the system \eqref{eq.rescaled kinetic equations} yields that in the leading order in $\alpha $ we have
\begin{multline*}
   \partial_t F_{eq}+ v\cdot \nabla_y F_{eq} 
   = L[F_{eq};f_1]+\mathcal{R}_p[F_{eq},G]\\=\begin{pmatrix}L^{(1,1)}_{el}[f_1]+L^{(1,2)}_{el}[f_1]+2\eta L_{1,1}[f_1]+\eta L_{1,2}^{(1)}[f_1]\\L^{(2,1)}_{el}[f_1]+
L^{(2,2)}_{el}[f_1]+\eta L^{(2)}_{non.el}[f_1]\end{pmatrix}+\mathcal{R}_p[F_{eq},G],
\end{multline*}where $\mathcal{K}[F_{eq}(1+\alpha f)]=\alpha L[F_{eq};f_1]+O(\alpha^2)$ and the linear operators are explicitly defined in \eqref{def.linear operator}.  
Then we take the inner product of the leading-order equation above against $1, (v-u),$ and $|v-u|^2+2\epsilon_0\delta_{j,2}$ to obtain that 
\begin{multline}\label{macro eq 1}
    \left\langle 1,\frac{\partial F_{eq}}{\partial \rho} [\partial_t\rho+v\cdot \nabla\rho]\right\rangle+\left\langle 1,\sum_{i=1}^3\frac{\partial F_{eq}}{\partial u_{i}} [\partial_tu_{i}+v\cdot \nabla u_{i}]\right\rangle+\left\langle 1,\frac{\partial F_{eq}}{\partial T} [\partial_tT+v\cdot \nabla T]\right\rangle\\= \left\langle 1,L[F_{eq};f_1]\right\rangle+\left\langle 1,\mathcal{R}_p[F_{eq},G]\right\rangle=\left\langle 1,\mathcal{R}_p[F_{eq},G]\right\rangle,
\end{multline}
\begin{multline}\label{macro eq 2}
    \left\langle (v-u),\frac{\partial F_{eq}}{\partial \rho} [\partial_t\rho+v\cdot \nabla\rho]\right\rangle+\left\langle (v-u),\sum_{i=1}^3\frac{\partial F_{eq}}{\partial u_{i}} [\partial_tu_{i}+v\cdot \nabla u_{i}]\right\rangle+\left\langle (v-u),\frac{\partial F_{eq}}{\partial T} [\partial_tT+ v\cdot \nabla T]\right\rangle\\= \left\langle (v-u),L[F_{eq};f_1]\right\rangle+\left\langle (v-u),\mathcal{R}_p[F_{eq},G]\right\rangle=\left\langle (v-u),\mathcal{R}_p[F_{eq},G]\right\rangle,
\end{multline}
and
\begin{multline}\label{macro eq 3}
    \bigg\langle |v-u|^2+2\epsilon_0\delta_{j,2},\frac{\partial F_{eq}}{\partial \rho} [\partial_t\rho+v\cdot \nabla\rho]
    +\sum_{i=1}^3\frac{\partial F_{eq}}{\partial u_{i}} [\partial_tu_{i}+v\cdot \nabla u_{i}]
    +\frac{\partial F_{eq}}{\partial T} [\partial_tT+v\cdot \nabla T]\bigg\rangle\\= \left\langle |v-u|^2+2\epsilon_0\delta_{j,2},L[F_{eq};f_1]\right\rangle+\left\langle |v-u|^2+2\epsilon_0\delta_{j,2},\mathcal{R}_p[F_{eq},G]\right\rangle\\=\left\langle |v-u|^2+2\epsilon_0\delta_{j,2},\mathcal{R}_p[F_{eq},G]\right\rangle,
\end{multline} by the kernel condition of $L$ in \eqref{kernel.L} and the symmetry of $L$.

We now compute the scalar products in detail. We first consider the scalar products with 1. By symmetry we obtain \begin{multline*}
\left\langle 1,\frac{\partial F_{eq}}{\partial \rho} [\partial_t\rho+v\cdot \nabla\rho]\right\rangle
=\left\langle 1,\frac{\partial F_{eq}}{\partial \rho} [\partial_t\rho+u\cdot \nabla\rho]\right\rangle+\left\langle 1,\frac{\partial F_{eq}}{\partial \rho} [(v-u)\cdot \nabla\rho]\right\rangle\\
=\frac{ 1}{ \rho}[\partial_t\rho+u\cdot \nabla\rho]\left[1+e^{-\frac{2\epsilon_0}{T}}\right]\int_\rth M(v)dv=[\partial_t\rho+u\cdot \nabla\rho]\left[1+e^{-\frac{2\epsilon_0}{T}}\right],
\end{multline*}where the local Maxwellian $M(v)=M(v;\rho,u,T)$ is defined in \eqref{LocalMaxwellian}.
and also
\begin{multline*}
\left\langle 1,\frac{\partial F_{eq}}{\partial u_{i}} [\partial_tu_{i}+v\cdot \nabla u_{i}]\right\rangle
=\left\langle 1,\frac{2}{T}(v-u)_iF_{eq} [\partial_tu_{i}+v\cdot \nabla u_{i}]\right\rangle\\
=\frac{2}{T}\left(\left\langle 1,F_{eq}(v-u)_i  [(v-u)\cdot \nabla u_{i}]\right\rangle+\left\langle 1,F_{eq}(v-u)_i  [\partial_tu_{i}+u\cdot \nabla u_{i}]\right\rangle\right)\\
=2\left[1+e^{-\frac{2\epsilon_0}{T}}\right]\partial_{y_i}u_{i} \frac{\pi^{-3/2}\rho}{3} \frac{3\pi^{3/2}}{2} =\left[1+e^{-\frac{2\epsilon_0}{T}}\right](\partial_{y_i}u_{i}) \rho  .
\end{multline*}
In addition, we have\begin{multline*}
\left\langle 1,\frac{\partial F_{eq}}{\partial T} [\partial_tT+v\cdot \nabla T]\right\rangle
=\left\langle 1,\frac{\partial F_{eq}}{\partial T} [\partial_tT+u\cdot \nabla T]\right\rangle+\left\langle 1,\frac{\partial F_{eq}}{\partial T} [(v-u)\cdot \nabla T]\right\rangle\\
=[\partial_tT+u\cdot \nabla T]\bigg[ -\frac{3}{2T}(1+e^{-\frac{2\epsilon_0}{T}})\rho +\frac{3}{2T}\rho+\frac{e^{-\frac{2\epsilon_0}{T}}}{T^2}\left(\frac{3\rho T}{2}+2\epsilon_0\rho\right)\bigg]
=2\epsilon_0\rho[\partial_tT+u\cdot \nabla T]\frac{e^{-\frac{2\epsilon_0}{T}}}{T^2}.
\end{multline*}Lastly, 
\begin{multline*}
    \langle 1, \mathcal{R}_p(F_{eq},G)\rangle = \int_\rth dv\int_{\mathbb{S}^2}dn\bigg[e^{-\frac{2\epsilon_0}{T}}M(v)(1+G(n))-M(v)G(n)\bigg]\\
    -\int_\rth dv\int_{\mathbb{S}^2}dn\bigg[e^{-\frac{2\epsilon_0}{T}}M(v)(1+G(n))-M(v)G(n)\bigg]=0.
\end{multline*}

Now we move onto the next step and compute the scalar product with $(v-u)$. We first obtain
\begin{multline*}
\left\langle v-u,\frac{\partial F_{eq}}{\partial \rho} [\partial_t\rho+v\cdot \nabla\rho]\right\rangle
=\left\langle v-u,\frac{\partial F_{eq}}{\partial \rho} [\partial_t\rho+u\cdot \nabla\rho]\right\rangle+\left\langle v-u,\frac{\partial F_{eq}}{\partial \rho} [(v-u)\cdot \nabla\rho]\right\rangle\\
=\frac{ 1}{ 3\rho} \nabla\rho\left[1+e^{-\frac{2\epsilon_0}{T}}\right]\int_\rth |v-u|^2M(v)dv=\frac{ T}{ 2} \nabla\rho\left[1+e^{-\frac{2\epsilon_0}{T}}\right].
\end{multline*} Furthermore, we have for $i,j=1,2,3,$
\begin{multline*}
\left\langle (v-u)_j,\frac{\partial F_{eq}}{\partial u_{i}} [\partial_tu_{i}+v\cdot \nabla u_{i}]\right\rangle
=\left\langle (v-u)_j,\frac{2}{T}(v-u)_iF_{eq} [\partial_tu_{i}+v\cdot \nabla u_{i}]\right\rangle\\
=\frac{2}{3T}[\partial_tu_{i}+u\cdot \nabla u_{i}]\left[1+e^{-\frac{2\epsilon_0}{T}}\right]\int_\rth  |v-u|^2  M(v)dv\text{ if }j=i\text{ and }=0,\text{ otherwise}\\
=\rho[\partial_tu_{i}+u\cdot \nabla u_{i}]\left[1+e^{-\frac{2\epsilon_0}{T}}\right]\text{ if }j=i\text{ and }=0,\text{ otherwise}.
\end{multline*}
In addition, we deduce that\begin{multline*}
\left\langle v-u,\frac{\partial F_{eq}}{\partial T} [\partial_tT+v\cdot \nabla T]\right\rangle
=\bigg\langle v-u,[(v-u)\cdot \nabla T]\bigg(-\frac{3}{2}\frac{c_0\rho}{T^{\frac{5}{2}}}\exp\left(-\frac{1}{T}(|v-u|^2+2\epsilon_0\delta_{j,2})\right)\\+\frac{c_0\rho}{T^{\frac{7}{2}}}[|v-u|^2+2\epsilon_0\delta_{j,2}]\exp\left(-\frac{1}{T}(|v-u|^2+2\epsilon_0\delta_{j,2})\right)\bigg)\bigg\rangle
= \nabla T\left[(1+e^{-\frac{2\epsilon_0}{T}})\frac{\rho}{2}+\frac{\epsilon_0\rho}{T}e^{-\frac{2\epsilon_0}{T}}\right].
\end{multline*}
Also, we obtain
$$\langle v-u, \mathcal{R}_p(F_{eq},G)\rangle=0.$$

Lastly, we compute the energy terms. We first have
\begin{multline*}
\left\langle |v-u|^2+2\epsilon_0\delta_{j,2},\frac{\partial F_{eq}}{\partial \rho} [\partial_t\rho+v\cdot \nabla\rho]\right\rangle\\
=\left\langle |v-u|^2+2\epsilon_0\delta_{j,2},\frac{\partial F_{eq}}{\partial \rho} [\partial_t\rho+u\cdot \nabla\rho]\right\rangle=\frac{ 1}{ \rho} [\partial_t\rho+u\cdot\nabla\rho]\left\langle |v-u|^2+2\epsilon_0\delta_{j,2},F_{eq} \right\rangle\\
=\frac{ 3T}{ 2} [\partial_t\rho+u\cdot\nabla\rho]\left[1+e^{-\frac{2\epsilon_0}{T}}\right]
+2\epsilon_0[\partial_t\rho+u\cdot\nabla\rho]e^{-\frac{2\epsilon_0}{T}}.
\end{multline*} Furthermore, we obtain
\begin{multline*}
\left\langle |v-u|^2+2\epsilon_0\delta_{j,2},\frac{\partial F_{eq}}{\partial u_{i}} [\partial_tu_{i}+v\cdot \nabla u_{i}]\right\rangle
=\left\langle |v-u|^2+2\epsilon_0\delta_{j,2},\frac{2}{T}(v-u)_iF_{eq} [\partial_tu_{i}+v\cdot \nabla u_{i}]\right\rangle\\
=\left[1+e^{-\frac{2\epsilon_0}{T}}\right]\partial_{y_i} u_{i} \frac{5\rho T}{2}
+2\epsilon_0\rho e^{-\frac{2\epsilon_0}{T}}\partial_{y_i} u_{i} .
\end{multline*}
In addition, we have\begin{multline*}
\left\langle |v-u|^2+2\epsilon_0\delta_{j,2},\frac{\partial F_{eq}}{\partial T} [\partial_tT+v\cdot \nabla T]\right\rangle
=\left\langle |v-u|^2+2\epsilon_0\delta_{j,2},\frac{\partial F_{eq}}{\partial T} [\partial_tT+u\cdot \nabla T]\right\rangle\\
=\rho[\partial_tT+u\cdot\nabla T]\bigg[ \frac{3}{2}(1+e^{-\frac{2\epsilon_0}{T}})
+\frac{3\epsilon_0}{T}e^{-\frac{2\epsilon_0}{T}}+\frac{4\epsilon^2_0}{T^2}e^{-\frac{2\epsilon_0}{T}}\bigg].
\end{multline*}
Also, we deduce that
\begin{multline*}\langle |v-u|^2+2\epsilon_0 \delta_{j,2}, \mathcal{R}_p(F_{eq},G)\rangle\\=\langle 2\epsilon_0 \delta_{j,2}, \mathcal{R}_p(F_{eq},G)\rangle
=-2\epsilon_0 \int_\rth dv\int_{\mathbb{S}^2}dn\ [\Ftw (1+G)-\Fon G]\\
=2\epsilon_0\rho\int_{\mathbb{S}^2}dn\ [G(1-e^{-\frac{2\epsilon_0}{T}})-e^{-\frac{2\epsilon_0}{T}}].\end{multline*}

Summarizing, we obtain the following table of products:
\begin{equation}
    \begin{split}
    \left\langle 1,\frac{\partial F_{eq}}{\partial \rho} D_v\rho\right\rangle&=D_u\rho\left[1+e^{-\frac{2\epsilon_0}{T}}\right],\\ \sum_{i=1}^3\left\langle 1,\frac{\partial F_{eq}}{\partial u_{i}} D_vu_i\right\rangle&=\left[1+e^{-\frac{2\epsilon_0}{T}}\right](\nabla\cdot u)\rho ,\\
    \left\langle 1,\frac{\partial F_{eq}}{\partial T} D_vT\right\rangle&=2\epsilon_0\rho (D_u T)\frac{e^{-\frac{2\epsilon_0}{T}}}{T^2},\\\left\langle 1,\mathcal{R}_p[F_{eq},G]\right\rangle&=0,\\
      \left\langle (v-u),\frac{\partial F_{eq}}{\partial \rho} D_v \rho\right\rangle&=\frac{ T}{ 2} \nabla\rho\left[1+e^{-\frac{2\epsilon_0}{T}}\right],\\
      \left\langle (v-u),\sum_{i=1}^3\frac{\partial F_{eq}}{\partial u_{i}} D_v u_i\right\rangle&=\rho D_u u \left[1+e^{-\frac{2\epsilon_0}{T}}\right],\\
      \left\langle (v-u),\frac{\partial F_{eq}}{\partial T} D_v T\right\rangle&=\nabla T\left[(1+e^{-\frac{2\epsilon_0}{T}})\frac{\rho}{2}+\frac{\epsilon_0\rho}{T}e^{-\frac{2\epsilon_0}{T}}\right],\\\left\langle (v-u),\mathcal{R}_p[F_{eq},G]\right\rangle&=0,\\
    \left\langle |v-u|^2+2\epsilon_0\delta_{j,2},\frac{\partial F_{eq}}{\partial \rho} D_v \rho\right\rangle&=\frac{ 3T}{ 2} (D_u \rho)\left[1+e^{-\frac{2\epsilon_0}{T}}\right]
+2\epsilon_0(D_u \rho) e^{-\frac{2\epsilon_0}{T}},\\
  \sum_{i=1}^3  \left\langle |v-u|^2+2\epsilon_0\delta_{j,2},\frac{\partial F_{eq}}{\partial u_{i}} D_v u_i\right\rangle&=\frac{5\rho T}{2}\left[1+e^{-\frac{2\epsilon_0}{T}}\right]\nabla \cdot u 
+2\epsilon_0\rho e^{-\frac{2\epsilon_0}{T}}\nabla  \cdot u, \\
   \left\langle |v-u|^2+2\epsilon_0\delta_{j,2},\frac{\partial F_{eq}}{\partial T} D_v T\right\rangle& =\rho D_u T\bigg[ \frac{3}{2}(1+e^{-\frac{2\epsilon_0}{T}})
+\frac{3\epsilon_0}{T}e^{-\frac{2\epsilon_0}{T}}+\frac{4\epsilon^2_0}{T^2}e^{-\frac{2\epsilon_0}{T}}\bigg],\\
    \langle |v-u|^2+2\epsilon_0 \delta_{j,2}, \mathcal{R}_p(F_{eq},G)\rangle&=2\epsilon_0\rho\int_{\mathbb{S}^2}dn\   [G(1-e^{-\frac{2\epsilon_0}{T}})-e^{-\frac{2\epsilon_0}{T}}],
    \end{split}
\end{equation}where we use the convective notation $D_v\eqdef \partial_t + v\cdot \nabla$ and $D_u\eqdef \partial_t + u\cdot \nabla$ here. 
Then by \eqref{macro eq 1}-\eqref{macro eq 3}, we obtain the following system of macroscopic equations:
\begin{multline}\label{macroscopic proto}
    D_u \rho\left[1+e^{-\frac{2\epsilon_0}{T}}\right]+\left[1+e^{-\frac{2\epsilon_0}{T}}\right]\nabla\cdot  u\rho+2\epsilon_0\rho (D_u T)\frac{e^{-\frac{2\epsilon_0}{T}}}{T^2}=0,\\
  \frac{ T}{ 2} \nabla\rho\left[1+e^{-\frac{2\epsilon_0}{T}}\right]+\rho D_u u\left[1+e^{-\frac{2\epsilon_0}{T}}\right]+  \nabla T\left[(1+e^{-\frac{2\epsilon_0}{T}})\frac{\rho}{2}+\frac{\epsilon_0\rho}{T}e^{-\frac{2\epsilon_0}{T}}\right]=0,\\
 \frac{ 3T}{ 2} D_u \rho\left[1+e^{-\frac{2\epsilon_0}{T}}\right]
+2\epsilon_0(D_u\rho) e^{-\frac{2\epsilon_0}{T}}+  \frac{5\rho T}{2} \left[1+e^{-\frac{2\epsilon_0}{T}}\right]\nabla\cdot u 
\\+2\epsilon_0\rho e^{-\frac{2\epsilon_0}{T}}\nabla\cdot u+   \rho D_u T\bigg[ \frac{3}{2}(1+e^{-\frac{2\epsilon_0}{T}})
+\frac{3\epsilon_0}{T}e^{-\frac{2\epsilon_0}{T}}+\frac{4\epsilon^2_0}{T^2}e^{-\frac{2\epsilon_0}{T}}\bigg]\\=     2\epsilon_0\rho\int_{\mathbb{S}^2}dn\ [G(1-e^{-\frac{2\epsilon_0}{T}})-e^{-\frac{2\epsilon_0}{T}}].
    \end{multline}
    
    The first equation \eqref{macroscopic proto}$_1$ can be rewritten as a continuity equation for the total molecule density:
    $$\partial_t\tilde{\rho}+\nabla\cdot \left(\tilde{\rho} u\right)=0,$$where $\tilde{\rho}\eqdef\rho\left[1+e^{-\frac{2\epsilon_0}{T}}\right]$ is the total molecule density. 
    
    The second equation \eqref{macroscopic proto}$_2$ is the equation for the total momentum: i.e., the Euler equation. Note that the total momentum is $\tilde{\rho} u.$ Then we reduce the second equation to get
    $$\tilde{\rho}D_u u = -\frac{1}{2}\nabla (\tilde{\rho}T)=-\nabla p,$$where $p\eqdef \frac{1}{2}\tilde{\rho}T$ is the pressure of the gas following the usual ideal gas law with the Boltzmann constant $k_B=\frac{1}{2}.$ Then the second equation is just the standard Euler equation for the fluid. 
    
    The presence of the radiation modifies the temperature equation as follows. We expect that the equation should be the equation for the transport of the energy with a source term induced by the radiation. Here we simplify the equation \eqref{macroscopic proto}$_3$ for the energy. Firstly, we note that
    $$\frac{3T}{2}D_u \tilde{\rho}=\frac{ 3T}{ 2} D_u\rho\left[1+e^{-\frac{2\epsilon_0}{T}}\right]+\frac{3\epsilon_0}{T}e^{-\frac{2\epsilon_0}{T}}\rho D_u T.$$ Therefore, the left-hand side of \eqref{macroscopic proto}$_3$ is equal to 
    \begin{equation*} \frac{3}{2}\partial_t(T\tilde{\rho})+\frac{3}{2}\nabla\cdot ( uT\tilde{\rho})+(\nabla\cdot u)\tilde{\rho}T+2\epsilon_0e^{-\frac{2\epsilon_0}{T}}D_u \rho +2\epsilon_0\rho e^{-\frac{2\epsilon_0}{T}}\nabla\cdot u
    +\frac{4\epsilon^2_0}{T^2}e^{-\frac{2\epsilon_0}{T}}\rho D_u T.\end{equation*}
 Here further note that
    \begin{equation*}2\epsilon_0  e^{-\frac{2\epsilon_0}{T}}D_u\rho+2\epsilon_0\rho e^{-\frac{2\epsilon_0}{T}}\nabla\cdot u
    +\rho \frac{4\epsilon^2_0}{T^2}e^{-\frac{2\epsilon_0}{T}}D_u T
    = 2\epsilon_0\left(\partial_t(\rho e^{-\frac{2\epsilon_0}{T}})+\nabla \cdot\left(u\rho e^{-\frac{2\epsilon_0}{T}}\right)\right).\end{equation*}   Therefore, we finally obtain that
    \begin{multline*} \frac{3}{2}\partial_t(T\tilde{\rho})+\frac{3}{2}\nabla\cdot ( uT\tilde{\rho})+2p(\nabla\cdot u)+2\epsilon_0\left(\partial_t(\rho e^{-\frac{2\epsilon_0}{T}})+\nabla \cdot\left(u\rho e^{-\frac{2\epsilon_0}{T}}\right)\right)\\
    =2\epsilon_0\rho\int_{\mathbb{S}^2}dn\ [G(1-e^{-\frac{2\epsilon_0}{T}})-e^{-\frac{2\epsilon_0}{T}}],\end{multline*}
    where we recall that the pressure $p$ is defined as $p=\frac{1}{2}\tilde{\rho}T$. Recall that we assume the Boltzmann constant to be $\frac{1}{2}$ throughout the paper. The energy density $e(T)$ is given by 
    \begin{equation}\label{total energy density}e=e(T)\eqdef \frac{3}{4}T+\epsilon_0\frac{ e^{-\frac{2\epsilon_0}{T}}}{1+e^{-\frac{2\epsilon_0}{T}}}.\end{equation} Then the energy equation becomes
    \begin{equation}
       \partial_t(\tilde{\rho}e)+ \nabla\cdot (\tilde{\rho}u e)+p\nabla\cdot u = \epsilon_0\rho\int_{\mathbb{S}^2}dn\ [G(1-e^{-\frac{2\epsilon_0}{T}})-e^{-\frac{2\epsilon_0}{T}}].
    \end{equation}
    
   Thus, we have obtained the following Euler system:
    \begin{equation}\label{timedep Euler}
        \begin{split}
      \frac{\partial \tilde{\rho}}{\partial t}+  \nabla\cdot \left(\tilde{\rho} u\right)&=0,    \\
    \frac{\partial u}{\partial t}+ (u\cdot \nabla)u +\frac{\nabla p}{\tilde{\rho}}&=0,\\
            \frac{\partial (\tilde{\rho}e)}{\partial t}+\nabla\cdot (\tilde{\rho}u e)+p\nabla\cdot u &= \epsilon_0\rho\int_{\mathbb{S}^2}dn\ [G(1-e^{-\frac{2\epsilon_0}{T}})-e^{-\frac{2\epsilon_0}{T}}],
        \end{split}
    \end{equation}coupled with the radiation equation 
    \begin{equation}
        \label{X2}
        n\cdot \nabla_y G=\epsilon_0 \rho \left[e^{-\frac{2\epsilon_0}{T}}(1+G)-G\right] ,
    \end{equation}
   which is obtained by integrating the Maxwellians in \eqref{stationary equations}$_2$. We can easily see that the third equation combined with the first is equivalent to
    $$ \frac{\partial e}{\partial t}+u\cdot \nabla e +\frac{p}{\tilde{\rho}}\nabla\cdot u =\epsilon_0 \int_{\mathbb{S}^2}dn\ \left[\frac{1-e^{-\frac{2\epsilon_0}{T}}}{1+e^{-\frac{2\epsilon_0}{T}}}G-\frac{e^{-\frac{2\epsilon_0}{T}}}{1+e^{-\frac{2\epsilon_0}{T}}}\right].$$The right-hand side of this equation is the heating (or cooling) term due to the radiation. 
  
  In \cite{WANG1,LZ1, LZ2}, the authors obtain well-posedness results for Euler-like and Navier-Stokes-Fourier-like systems for a compressible gas coupled with radiation. In the models that they consider, the equations for energy and momentum are reformulated including in them the momentum and the energy due to the radiation. Analogous systems have also been introduced in \cite{BUET}, where the thermodynamic consistency of the resulting equation has been studied. In addition, Chapman-Enskog approximations under the assumption that the radiation distribution is close to equilibrium at each point and it changes slowly have also been obtained.

\subsection{Equation for the entropy}We define the entropy density $s$ by means of
\begin{multline}
    \label{entropy}
   \tilde{\rho}s=-\sum_{j=1}^2 \int_\rth dv\ \log(F_{eq}^{(j)} )F_{eq}^{(j)}\\
   =-\tilde{\rho}\left(\log\left(\frac{\tilde{\rho}}{T^{3/2}(1+e^{\frac{-2\epsilon_0}{T}})}\right)+\frac{2\epsilon_0}{T}\frac{e^{\frac{-2\epsilon_0}{T}}}{T^{3/2}(1+e^{\frac{-2\epsilon_0}{T}})}-\log(c_0)+\frac{3}{2}\right).
\end{multline}
Define $$\lambda (T)\eqdef-\left(-\log\left(T^{3/2}(1+e^{\frac{-2\epsilon_0}{T}})\right)+\frac{2\epsilon_0}{T}\frac{e^{\frac{-2\epsilon_0}{T}}}{T^{3/2}(1+e^{\frac{-2\epsilon_0}{T}})}-\log(c_0)+\frac{3}{2}\right) ,$$ such that 
$$s=-\log(\tilde{\rho}) +\lambda (T).$$ 
Then we observe that
\begin{equation*}
    \partial_t s +(u\cdot \nabla)s= -\frac{(\partial_t+u\cdot\nabla) \tilde{\rho}}{\tilde{\rho}}+\lambda'(T)(\partial_t+u\cdot\nabla) T =\nabla\cdot u +\frac{\lambda'(T)}{e'(T)}\left[-\frac{p}{\tilde{\rho}}\nabla\cdot u+\frac{1}{\tilde{\rho}}Q[G]\right],
\end{equation*}using \eqref{timedep Euler}. Then an elementary, but tedious computation, yields
$$T\lambda'(T)=2e'(T).$$
Therefore, we have
\begin{equation}\label{entropy eq}
    \partial_t s +(u\cdot \nabla)s =\nabla\cdot u +\frac{2}{T}\left[-\frac{p}{\tilde{\rho}}\nabla\cdot u+\frac{1}{\tilde{\rho}}Q[G]\right]=\frac{2}{T\tilde{\rho}}Q[G]=\frac{1}{p}Q[G].
\end{equation} This is the equation for the entropy. We observe that the entropy is not conserved due to the radiation. Using \eqref{entropy eq} and \eqref{timedep Euler} we can also obtain
$$\partial_t(\tilde{\rho}s)+\nabla \cdot (\tilde{\rho}us)=\frac{2}{T}Q[G].$$ 

\begin{remark}
For ideal gases without the radiation taken into account, the entropy density $s$ satisfies 
$e^s=p\rho^{-\gamma},\ T=\frac{p}{\rho},\ \frac{Ds}{Dt}=\partial_ts+v\cdot\nabla s=0$ for some $\gamma$.
This is from $\partial_t(\rho s)=-\nabla\cdot (\rho u s).$ In our case, the entropy equation contains a contribution of the heat transfer due to the radiation.  
\end{remark}

This completes the formal derivation of the system of Euler equations from the reference kinetic equations for gas molecules and radiation in the LTE situation. In the next section, we also introduce the derivation of Euler equations in a non-LTE situation.

\section{Stationary solutions to the Euler-like system in the LTE situation}\label{sec.LTEexist}In this section, we study the behavior of solutions to the Euler system \eqref{timedep Euler} coupled with the equation for radiation from the hydrodynamic limit. Our goal is to prove the existence of stationary solutions to \eqref{timedep Euler} under several different scaling situations. We recall that the rescaled radiation intensity $G$ satisfies
\begin{equation}
    \notag
    n\cdot \nabla_y G =  \epsilon_0 \int_\rth dv\ [\Ftw (1+G)-\Fon G].
\end{equation}
Here, we use the Chapman-Enskog expansion \eqref{Chapman} and let $\alpha \to 0$ to obtain the leading order expansion. Then we have (c.f. \cite{rutten1995radiative}) 
\begin{equation}
    \label{rad eq}
    n\cdot \nabla_y G=\epsilon_0 \rho\left[e^{-\frac{2\epsilon_0}{T}} (1+G)-G\right]. 
\end{equation}Thus, an equilibrium solution is the pseudo-Planck distribution (Planck for a single frequency):
\begin{equation}\label{pseudo planck}G_{eq}=\frac{e^{-\frac{2\epsilon_0}{T}}}{1-e^{-\frac{2\epsilon_0}{T}}}.\end{equation}

\subsection{Steady states}We consider in this paper steady states in which the fluid is at rest; i.e., $u=0.$ Note that assuming specular reflection boundary conditions for the Boltzmann equations $F(t,x,v)=F(t,x,v-2(n_x\cdot v)n_x)$ for an outward normal vector $n_x$ at $x\in \partial\Omega$, we would obtain the boundary condition $u\cdot n=0$ for the Euler-like system.  Then the Euler-like system for the gas reduces to the equation 
\begin{equation}
    \label{X1}
    p=p_0=(\text{constant}),
\end{equation}
where $p=\frac{1}{2}\tilde{\rho}T$. Then \eqref{X1} yields a relation between $\tilde{\rho}$ and $T$, namely
\begin{equation}
    \label{C1}
    \tilde{\rho}T=c_0
\end{equation} for some free parameter $c_0>0.$ The equation for the energy \eqref{timedep Euler}$_3$ in the steady case with $u=0$ reduces to 
$$\int_{\mathbb{S}^2}dn\ [G(1-e^{-\frac{2\epsilon_0}{T}})-e^{-\frac{2\epsilon_0}{T}}]
=0.$$ Then, integrating in $\mathbb{S}^2$ in the equation for the radiation \eqref{X2}, we obtain \begin{equation}
    \label{R1}
    \nabla\cdot J_R = 0,
\end{equation}where $J_R=\int_{\mathbb{S}^2}nGdn.$ We will discuss in this paper several situations in which we can determine the temperature $T$ at each point using \eqref{X2} and \eqref{R1}. Using then \eqref{C1}, it is possible to obtain the gas density at each point. There is a free parameter $c_0$ that can be determined using the total amount of gas contained in the system, namely $\int_{\Omega}\tilde{\rho}dx.$  In order to determine the temperature $T$, we need to impose suitable boundary conditions for \eqref{X2}. In this paper, we will consider only the boundary condition in which the incoming radiation towards the container where the gas is included is given. 
Moreover, we will consider only situations in which the domain $\Omega$ containing the gas is convex and the boundary $\partial\Omega $ is $C^1$.

We begin analyzing a problem that can be reduced to a one-dimensional situation; namely a gas is contained in a slab between two parallel planes and the radiation intensity depends only on the azimuthal angle with respect to the vector normal to the planes.

\subsubsection{In a slab geometry}\label{sec.slab1}
We can first look into a physically interesting situation where the gas is contained between the two parallel planes $y_1=0$ and $y_1=L$, and $G$ does not depend on $y_2,y_3$ and it depends only on the angle $\theta$ between $n$ and the vector $(1,0,0)$ such that $G=G(y_1,\theta).$ In this case, we have 
$$\cos\theta \frac{\partial G}{\partial y_1} =\epsilon_0 \rho [e^{-\frac{2\epsilon_0}{T}}(1+G)-G].$$ We impose the following boundary conditions in which the incoming radiation is given by
\begin{equation}\label{eq.rad.boundarycondition}
    G(0,\theta,t)=I_0(\theta)\text{ if }\theta <\frac{\pi}{2},\text { and }G(L,\theta,t)=0\text{ if }\theta> \frac{\pi}{2},
\end{equation}and we have no incoming radiation on $y_1=L$.  Then together with the boundary condition \eqref{eq.rad.boundarycondition}, we can explicitly solve the ODE and obtain that
\begin{multline*}
    G(y_1,\theta)=I_0(\theta)\exp\left(-\frac{\epsilon_0}{\cos\theta}\int_0^{y_1} \rho(\xi)(1-e^{-\frac{2\epsilon_0}{T(\xi)}})d\xi\right) \\
    +\int_0^{y_1}\frac{\epsilon_0\rho(\xi)}{\cos\theta}e^{-\frac{2\epsilon_0}{T(\xi)}} \exp\left(-\frac{\epsilon_0}{\cos\theta}\int_\xi^{y_1} \rho(\eta)(1-e^{-\frac{2\epsilon_0}{T(\eta)}})d\eta\right)d\xi,
\end{multline*}if $\theta<\frac{\pi}{2}$ and
$$
    G({y_1},\theta)=\int_{y_1}^L\frac{\epsilon_0\rho(\xi)}{|\cos\theta|}e^{-\frac{2\epsilon_0}{T(\xi)}} \exp\left(-\frac{\epsilon_0}{\cos\theta}\int_\xi^{y_1} \rho(\eta)(1-e^{-\frac{2\epsilon_0}{T(\eta)}})d\eta\right)d\xi,
$$if $\theta>\frac{\pi}{2}$.
Then for the stationary solutions, we can consider the linearization around a constant $\rho_0$ and $T_0$ such that 
$p=\frac{1}{2}\tilde{\rho}T$ is constant, $u=0$, and $J_R$ is constant. Here free constants that we can control are $\rho_0,T_0,$ and $I_0$.

\subsection{Constant stationary solutions}
\label{sec.constant solutions} We will obtain solvability condition for \ref{timedep Euler}, \eqref{rad eq}, and \eqref{generic.rad.boundarycondition} in cases in which the absorption ratio is near constant. To this end we first consider particular incoming radiation distributions yielding constant temperature at the gas. 
In particular, we are interested in the constant stationary solutions to the system \eqref{timedep Euler} and \eqref{rad eq}. Then later we will also consider the linearization of the model nearby these constant steady-states. In this section, for the sake of simplicity, we will only consider the 1-dimensional slab case where the functions are independent of $y_2$ and $y_3$ variables as in Section \ref{sec.slab1}. We remark that more general 3-dimensional case can be treated similarly.  Here we impose generic incoming boundary conditions for the radiation we take the following boundary conditions in which the radiation enters and it escapes:
\begin{equation}\label{generic.rad.boundarycondition}
    G(0,\theta,t)=a_+(\theta)\text{ if }\theta <\frac{\pi}{2},\text { and }G(L,\theta,t)=a_-(\theta)\text{ if }\theta> \frac{\pi}{2},
\end{equation}  where $a_\pm$ are given.
In this section, we will examine conditions to have solutions with $u=0,$ constant $T$, and constant $\rho$ (or equivalently $\tilde{\rho}$).  We will show that we can have constant solutions under a particular choice of boundary conditions for the radiation $G$; namely, if we assume the incoming boundary condition for the radiation with the incoming profile equal to the Planck distribution, then we can have constant solutions. More precisely, we obtain the following proposition.
\begin{proposition}\label{prop.constant.sol}
The system \eqref{timedep Euler}, \eqref{rad eq}, and \eqref{generic.rad.boundarycondition} can have a constant stationary solution $(\tilde{\rho}_0,0,T_0,G_0)$ if the boundary condition \eqref{generic.rad.boundarycondition} is given by the constant Planck distribution ratio $$a_+(\theta)=a_-(\pi-\theta)=\frac{e^{-\frac{2\epsilon_0}{T_0}}}{1-e^{-\frac{2\epsilon_0}{T_0}}}.$$\end{proposition}
\begin{proof}
Let $\rho=\rho_0$ and $T=T_0$ be constant. For stationary solutions with $u=0$, the energy equation \eqref{timedep Euler}$_3$ becomes \begin{equation}
    \label{energy.eq.constant section}\epsilon_0\rho\int_{\mathbb{S}^2}dn\ [G(1-e^{-\frac{2\epsilon_0}{T}})-e^{-\frac{2\epsilon_0}{T}}]=0.
\end{equation} 
Also, the radiation equation \eqref{rad eq} becomes \begin{equation}
    \label{radiation.eq.constant section}n\cdot \nabla_y G=\cos\theta \frac{\partial G}{\partial y_1}=\epsilon_0 \rho\left[e^{-\frac{2\epsilon_0}{T}} (1+G)-G\right].\end{equation}
Solving \eqref{radiation.eq.constant section} with the boundary conditions \eqref{generic.rad.boundarycondition}, we have 
\begin{multline}\label{constantcase1}
    G(y_1,\theta)=a_+(\theta)\exp\left(-\frac{\epsilon_0\rho}{\cos\theta} (1-e^{-\frac{2\epsilon_0}{T}})y_1\right) 
   +\int_0^{y_1}\frac{\epsilon_0\rho}{\cos\theta}e^{-\frac{2\epsilon_0}{T}} \exp\left(-\frac{\epsilon_0\rho}{\cos\theta}(y_1-\xi) (1-e^{-\frac{2\epsilon_0}{T}})\right)d\xi,
\end{multline}if $\theta<\frac{\pi}{2}$ and
\begin{multline}\label{constantcase2}
    G({y_1},\theta)=a_-(\theta)\exp\left(-\frac{\epsilon_0\rho}{\cos\theta} (1-e^{-\frac{2\epsilon_0}{T}})(y_1-L)\right) \\-\int_{y_1}^L\frac{\epsilon_0\rho}{\cos\theta}e^{-\frac{2\epsilon_0}{T}} \exp\left(-\frac{\epsilon_0\rho}{\cos\theta} (y_1-\xi)(1-e^{-\frac{2\epsilon_0}{T}})\right)d\xi,
\end{multline}if $\theta>\frac{\pi}{2}$.
On the other hand, \eqref{energy.eq.constant section}-\eqref{radiation.eq.constant section} together imply
$$\partial_{y_1} \int_{\mathbb{S}^2}dn \ G\cos\theta = 0,$$ and hence $J(y_1)\eqdef \int_{\mathbb{S}^2}dn  G\cos\theta $ must be constant. Using \eqref{constantcase1}-\eqref{constantcase2}, we have
\begin{multline*}
    J(y_1)
   =2\pi \int _0^{\frac{\pi}{2}}d\theta\ \cos\theta\sin\theta \bigg(a_+(\theta)\exp\left(-\frac{\epsilon_0\rho}{\cos\theta} (1-e^{-\frac{2\epsilon_0}{T}})y_1\right)\bigg) \\
   +2\pi \int _0^{\frac{\pi}{2}}d\theta\ \cos\theta\sin\theta \bigg(\int_0^{y_1}\frac{\epsilon_0\rho}{\cos\theta}e^{-\frac{2\epsilon_0}{T}} \exp\left(-\frac{\epsilon_0\rho}{\cos\theta}(y_1-\xi) (1-e^{-\frac{2\epsilon_0}{T}})\right)d\xi\bigg)\\
   -2\pi \int _0^{\frac{\pi}{2}}d\theta\ \cos\theta\sin\theta \bigg(a_-(\pi-\theta)\exp\left(\frac{\epsilon_0\rho}{\cos\theta} (1-e^{-\frac{2\epsilon_0}{T}})(y_1-L)\right)\bigg) \\+2\pi \int _0^{\frac{\pi}{2}}d\theta\ \cos\theta\sin\theta \bigg(\int_{y_1}^L\frac{\epsilon_0\rho}{|\cos\theta|}e^{-\frac{2\epsilon_0}{T}} \exp\left(\frac{\epsilon_0\rho}{\cos\theta} (y_1-\xi)(1-e^{-\frac{2\epsilon_0}{T}})\right)d\xi\bigg),
\end{multline*}
by the change of variables $\theta\mapsto \pi-\theta$. Then we can write $J(y_1)=J_1(y_1)+J_2(y_1)$ where we define $J_1$ and $J_2$ as
\begin{multline}
    \label{J1}
    J_1(y_1)\eqdef 2\pi \int _0^{\frac{\pi}{2}}d\theta\ \cos\theta\sin\theta \bigg(a_+(\theta)\exp\left(-\frac{\epsilon_0\rho}{\cos\theta} (1-e^{-\frac{2\epsilon_0}{T}})y_1\right)\\
    -a_-(\pi-\theta)\exp\left(\frac{\epsilon_0\rho}{\cos\theta} (1-e^{-\frac{2\epsilon_0}{T}})(y_1-L)\right)\bigg),
\end{multline}
and
\begin{multline}
    \label{J2}
    J_2(y_1)\eqdef 2\pi \int _0^{\frac{\pi}{2}}d\theta\ \cos\theta\sin\theta \bigg(\int_0^{y_1}\frac{\epsilon_0\rho}{\cos\theta}e^{-\frac{2\epsilon_0}{T}} \exp\left(-\frac{\epsilon_0\rho}{\cos\theta}(y_1-\xi) (1-e^{-\frac{2\epsilon_0}{T}})\right)d\xi\bigg)\\+2\pi \int _0^{\frac{\pi}{2}}d\theta\ \cos\theta\sin\theta \bigg(\int_{y_1}^L\frac{\epsilon_0\rho}{|\cos\theta|}e^{-\frac{2\epsilon_0}{T}} \exp\left(\frac{\epsilon_0\rho}{\cos\theta} (y_1-\xi)(1-e^{-\frac{2\epsilon_0}{T}})\right)d\xi\bigg).
    \end{multline}
    Then we further note that 
    \begin{multline}\label{J2new}
        J_2(y_1)
        =2\pi \int _0^{\frac{\pi}{2}}d\theta\ \cos\theta\sin\theta\frac{e^{-\frac{2\epsilon_0}{T}}}{1-e^{-\frac{2\epsilon_0}{T}}}\\\times  \bigg[\exp\left(\frac{\epsilon_0\rho}{\cos\theta}(1-e^{-\frac{2\epsilon_0}{T}})(y_1-L) \right)-\exp\left(-\frac{\epsilon_0\rho}{\cos\theta}(1-e^{-\frac{2\epsilon_0}{T}})y_1 \right)\bigg].
    \end{multline}
    In order to examine the appropriate boundary conditions $a_\pm$, we want them to make $J$ be constant. In other words, $J_1+J_2$ must be constant for a given set of boundary conditions $a_+$, $a_-$. The simplest choice of $a_\pm$ yielding $J$ constant is 
$$a_+(\theta)=a_-(\pi-\theta)=\frac{e^{-\frac{2\epsilon_0}{T}}}{1-e^{-\frac{2\epsilon_0}{T}}},$$ which means that the incoming radiation is given by the pseudo-Planck equilibrium distributions.  In this way, we can easily see from \eqref{J1} and \eqref{J2new} that $$J_1+J_2=0.$$ Also, we can see that $G=G_0=\frac{e^{-\frac{2\epsilon_0}{T}}}{1-e^{-\frac{2\epsilon_0}{T}}}$ becomes a particular solution for the system under the boundary conditions. This completes the proof.
\end{proof}
\begin{remark}
One open question is to prove that this solution above is the only solution for $T$ constant.
\end{remark}
\subsection{Linearization near constant steady states}
\label{sec.linearized}
We can now consider the linearization of the system near the constant steady-state $(\tilde{\rho}, u, T)=(\tilde{\rho}_0,0,T_0)$ that we introduced in Section \ref{sec.constant solutions}. We first let $\tilde{\rho}=\tilde{\rho}_0(1+\zeta),$ $T=T_0(1+\theta)$ for $|\zeta|+|\theta|+|u|\ll1 $ such that  $\frac{2\epsilon_0}{T_0}|\theta|\ll 1.$   We also linearize the equation \eqref{rad eq} as $G=G_0(1+h)$ for the radiation $G$ near $G_0\eqdef \frac{e^{-\frac{2\epsilon_0}{T_0}}}{1-e^{-\frac{2\epsilon_0}{T_0}}}$ such that $e^{-\frac{2\epsilon_0}{T_0}}(1+G_0)-G_0=0.$
Then the system \eqref{timedep Euler} coupled with \eqref{rad eq} with $c= \infty$ yields that the perturbations $\zeta$, $\theta$, and $h$ satisfy the following system of equations:
\begin{equation}\label{linearized equations}
    \begin{split}
        \frac{\partial \zeta}{\partial t} +\nabla_y \cdot u&=0,\\
         \frac{\partial u}{\partial t} +\frac{T_0}{2}\nabla_y (\zeta+\theta)&=0,\\
         \lambda_0\epsilon_0\frac{\partial \theta}{\partial t}+\frac{p_0}{\tilde{\rho}_0}\nabla_y \cdot u &=\frac{\epsilon_0 G_0}{1+e^{-\frac{2\epsilon_0}{T_0}}}\int_{\mathbb{S}^2}dn\ \left[\frac{h}{1+G_0}-\frac{2\epsilon_0}{T_0}\theta \right]\\
         n\cdot \nabla _y h &= \frac{\epsilon_0 \tilde{\rho}_0 G_0}{1+e^{-\frac{2\epsilon_0}{T_0}}}\left[\frac{2\epsilon_0}{T_0}\theta -\frac{h}{1+G_0}\right],
    \end{split}
\end{equation}
where $\lambda_0 \eqdef \frac{2\epsilon_0}{T_0}\frac{e^{-\frac{2\epsilon_0}{T_0}}}{\left(1+e^{-\frac{2\epsilon_0}{T_0}}\right)^2}$ and $p_0=\tilde{\rho}_0T_0.$
 Using \eqref{linearized equations} we can also obtain that
\begin{equation}\begin{split}
    \label{linearized system2}
    \frac{\partial^2\zeta}{\partial t^2}&=\frac{T_0}{2}\Delta_y (\zeta+\theta),\\
    \frac{\partial}{\partial t}(\lambda_0\epsilon_0\theta -\frac{p_0}{\tilde{\rho}_0}\zeta)&=\frac{1}{\tilde{\rho}_0}\nabla_y\cdot \int_{\mathbb{S}^2}nh dn.
\end{split}
\end{equation}

One of the physical boundary conditions for the radiation $h$ is the absorbing boundary condition: $h(0,n)=0, \text{ for }n_1>0,$ and $h(L,n)=0,$ for $n_1<0$. Then the energy is lost only through the escape of radiation through the boundaries.  If we consider the specular reflection boundary for the velocity distribution $F$ for the molecules, then this corresponds to $u\cdot n=0$ at the boundaries. Then using \eqref{linearized equations}$_2$ we can have
\begin{equation}\label{boundary cond for lin pro}u\cdot n=0\text{  and  } \frac{\partial}{\partial n}(\zeta+\theta)=0,\end{equation}at the boundaries $y_1=0$ or $L.$


\subsection{Existence of steady-states to the linearized problem}\label{sec.linearized existence}
In this section, we study the steady-states to the linearized problem. We study the stationary solutions of the system \eqref{linearized equations}-\eqref{linearized system2} with the boundary conditions \eqref{boundary cond for lin pro}. We impose the ``incoming" boundary condition for $h$ as 
\begin{equation}
    \label{boundary cond for h}
    h(0,n)=j_0(n)\text{ if }n_1>0 \text {  and  }h(L,n)=0\text{ if }n_1<0.
\end{equation}
 We also assume the symmetry in $y_2$ and $y_3$ such that each perturbation depends only on $y_1$ variable. Notice that the function characterizing the gas (i.e., $T$ and $\rho$) are also part of the unknowns of the problems. \begin{remark}The existence of solutions assuming that the absorption or emission are given functions (i.e., $T$ and $\rho$ given in our setting), in the half-plane have been considered in \cite{Gupta,Chandrasekhar,Holstein}.
 \end{remark} Then our theorem states the following
 \begin{theorem}\label{existence theorem 3}
 The linearized stationary system \eqref{linearized equations}-\eqref{linearized system2} with the boundary conditions \eqref{boundary cond for lin pro} and \eqref{boundary cond for h} has a unique solution $(\zeta,\theta)$ with $u=0$. 
 \end{theorem}\begin{proof}
 From \eqref{linearized system2} we have
$$\frac{T_0}{2}\partial_{y_1}^2(\zeta+\theta)=\frac{1}{\tilde{\rho}_0}\partial_{y_1} \int_{\mathbb{S}^2}n_1h dn=0.$$  By \eqref{linearized equations}$_2$, we can also see that \begin{equation}\label{eq.zetatheta}\zeta+\theta = C_0,\end{equation} for some constant $C_0$. For $h$ we use \eqref{linearized equations}$_4$ and obtain 
$$ n_1\partial_{y_1} h = \frac{\epsilon_0  \tilde{\rho}_0 e^{\frac{-2\epsilon_0}{T_0}}}{1+e^{-\frac{2\epsilon_0}{T_0}}}\left[\frac{2\epsilon_0}{T_0}\theta(1+G_0) -h\right]=\frac{1}{l_0}\left[\alpha_0\theta -h\right],$$ where $\frac{1}{l_0}\eqdef \frac{\epsilon_0  \tilde{\rho}_0 e^{\frac{-2\epsilon_0}{T_0}}}{1+e^{-\frac{2\epsilon_0}{T_0}}}$ and $\alpha_0 \eqdef \frac{2\epsilon_0}{T_0}(1+G_0).$
Then by solving this equation with the boundary condition \eqref{boundary cond for h}, we have
\begin{equation}\notag
    \begin{split}
        h(y_1,\psi)&=j_0(\psi)e^{-\frac{1}{l_0}\frac{y_1}{\cos\psi}} +\alpha_0 \int_0^{y_1}\frac{d\xi}{l_0\cos\psi}\theta(\xi)e^{-\frac{1}{l_0}\frac{y_1-\xi}{\cos\psi}},\text{ if }\psi\in (0,\pi/2),\\
       h(y_1,\psi)&=\alpha_0 \int_{y_1}^L\frac{d\xi}{l_0|\cos\psi|}\theta(\xi)e^{-\frac{1}{l_0}\frac{y_1-\xi}{\cos\psi}},\text{ if }\psi\in (\pi/2,\pi),
    \end{split}
\end{equation}where we write $n_1=\cos\psi.$ Then this $h$ must solve $\partial_{y_1}\int_{\mathbb{S}^2}n_1 h dn = 0.$
By plugging the solution $h$ into $\partial_{y_1}\int_{\mathbb{S}^2}n_1 h dn = 0,$
 we obtain that for any $y_1\in (0,L),$
 \begin{multline*}
     \int_0^{\pi/2}d\psi\sin\psi\cos\psi j_0(\psi)e^{-\frac{1}{l_0}\frac{y_1}{\cos\psi}} +\alpha_0  \int_0^{\pi/2}d\psi\sin\psi\int_0^{y_1}\frac{d\xi}{l_0}\theta(\xi)e^{-\frac{1}{l_0}\frac{y_1-\xi}{\cos\psi}}
 -    \alpha_0  \int_{\pi/2}^\pi d\psi\sin\psi\int_{y_1}^L\frac{d\xi}{l_0}\theta(\xi)e^{-\frac{1}{l_0}\frac{y_1-\xi}{\cos\psi}}\\= \int_0^{\pi/2}d\psi\sin\psi\cos\psi j_0(\psi)e^{-\frac{1}{l_0}\frac{y_1}{\cos\psi}} -\frac{\alpha_0}{l_0}  \int_0^{\pi/2}d\psi\sin\psi\int_0^{L}d\xi\theta(\xi)G(y_1-\xi,\psi)
 =i_0,
 \end{multline*}
for some constant $i_0$ where we define 
$$G(x,\psi)\eqdef -\text{sgn}(x)\exp\left(-\frac{1}{l_0}\frac{|x|}{\cos\psi}\right),\ \psi\in [0,\pi/2).$$ Note that $G$ satisfies 
$$-\frac{\partial^2}{\partial x^2}G+\frac{1}{l^2_0\cos^2\psi}G=2\delta'(x).$$
By defining $$\Phi(y_1)\eqdef \frac{l_0}{\alpha_0}\left(\int_0^{\pi/2}d\psi\sin\psi\cos\psi j_0(\psi)e^{-\frac{1}{l_0}\frac{y_1}{\cos\psi}}-i_0\right) ,$$ we can rewrite the problem as  
\begin{equation}
     \int_0^{\pi/2}d\psi\sin\psi\int_0^{L}d\xi\ \theta(\xi)G(y_1-\xi,\psi)=\Phi(y_1).
\end{equation}Here we rescale the variables and let $l_0=1 $ without loss of generality. 
We can also rewrite the problem for $Q(x,\psi)\eqdef e^{-\frac{|x|}{\cos\psi}}$ which satisfies $\cos\psi \frac{\partial Q}{\partial x}(x,\psi)= G(x,\psi)$ and \begin{equation}
     \int_0^{\pi/2}d\psi\sin\psi\int_0^{L}d\xi\ \theta(\xi)\cos\psi \frac{\partial Q}{\partial x}(y_1-\xi,\psi)=\Phi(y_1).
\end{equation}
Since $Q$ solves  $\cos\psi(-\frac{\partial^2}{\partial x^2}Q+\frac{1}{\cos^2\psi}Q)=2\delta(x)$, we have 
$$
   -2\theta (y_1)\int_0^{\pi/2}d\psi\sin\psi+ \int_0^{\pi/2}d\psi\tan\psi\int_0^{L}d\xi\ \theta(\xi)Q(y_1-\xi,\psi)=\Phi'(y_1).
$$
Therefore, we have
\begin{equation}\notag
      \theta (y_1)=\frac{1}{2} \int_0^{\pi/2}d\psi\tan\psi\int_0^{L}d\xi\ \theta(\xi)Q(y_1-\xi,\psi)-\frac{1}{2}\Phi'(y_1).
\end{equation}
 Define the kernel
 $$K(x)\eqdef \frac{1}{2} \int_0^{\pi/2}d\psi\tan\psi\ Q(x,\psi).$$  Then we have
 $$\int_{\mathbb{R}}dx\  K(x)
=  \frac{1}{2}\int_{\mathbb{R}}dx\ \int_0^{\pi/2}d\psi\tan\psi e^{-\frac{|x|}{\cos\psi}}= \int_0^{\pi/2}d\psi\sin\psi=1.$$
 Thus, the problem is now equivalent to prove the existence for $\theta $ which satisfies
 \begin{equation}\label{eq.final equation for theta}
     \theta(x)=\int_0^{L} \theta(\xi)K(x-\xi)d\xi-\frac{1}{2}\Phi'(x) \text { with }\int_{\mathbb{R}}K(x)dx=1,
 \end{equation} where $\Phi$ is given.
 Then this equation is the Fredholm integral equation of the second kind.  Since $\Phi'$ and $K$ are continuous and 
 $$\sup_{x\in (0,L)}\int_0^{L} K(x-\xi)d\xi<1,$$ there exists a unique solution $\theta:[0,L]\to \mathbb{R} $ which solves the Fredholm integral equation \eqref{eq.final equation for theta} by the contraction mapping theorem in $C([0,L])$.  
 
Once we have the perturbation of the temperature profile $\theta$, we can also obtain the density perturbation $\zeta$ using \eqref{eq.zetatheta} as $\zeta(y_1)=C_0-\theta(y_1)$.  We can determine the free constant $C_0$ uniquely using the mass conservation. This completes the proof of existence of the steady states to the linearized problem.
\end{proof}

\begin{remark}
  We note that the system \eqref{linearized equations} does not contain any viscosity term (no diffusion). The only dissipative effect is the exchange of energy with the radiation.
\end{remark}

 \subsection{A scaling limit yielding constant absorption rate}
 In Section \ref{sec.linearized}-\ref{sec.linearized existence},
we have considered the linearized problem of the system near the steady-state $(\tilde{\rho}, u, T)=(\tilde{\rho}_0,0,T_0)$ by letting the perturbations satisfy $|\zeta|+|\theta|+|u|\ll1 $ and $\frac{2\epsilon_0}{T_0}|\theta|\ll 1.$  There is another interesting scaling limit in which we obtain constant absorption rate, but an emission rate depending non-linearly in the temperature of the gases. In this case we can prove also well-posedness for the problem determining the temperature given the incoming flux. In this limit we assume that 
the perturbations still satisfy  $|\zeta|+|\theta|+|u|\ll1 $ but the ratio $\frac{2\epsilon_0}{T_0}$ is large such that $\frac{2\epsilon_0}{T_0}|\theta|\approx 1.$
In this case, the linearized system can be written as a nonlinear system with an exponential dependence in the temperature $e^{-\frac{2\epsilon_0}{T}}$, in which the perturbation term is no longer negligible.  Since $\frac{2\epsilon_0}{T_0}$ is large and $G_0=\frac{e^{-\frac{2\epsilon_0}{T_0}}}{1-e^{-\frac{2\epsilon_0}{T_0}}}\simeq e^{-\frac{2\epsilon_0}{T_0}},$ we rescale and let
$$G=e^{-\frac{2\epsilon_0}{T_0}} H,$$ and we will rewrite the system in terms of $H$.
Also, we write $$\frac{1}{T}=\frac{1}{T_0(1+\theta)}=\frac{1}{T_0}(1-\frac{T_0}{2\epsilon_0}\vartheta)\text{ such that }e^{-\frac{2\epsilon_0}{T}}=e^{-\frac{2\epsilon_0}{T_0}}e^\vartheta,$$ and$$\tilde{\rho}=\tilde{\rho}_0(1+\zeta)=\tilde{\rho}_0\left(1+\frac{T_0}{2\epsilon_0}\xi\right).$$ First of all, we recall \eqref{timedep Euler} and \eqref{rad eq} and observe that a stationary solution $(\tilde{\rho},u,T,G)$ satisfies the following system:\begin{equation}\label{stationary system}
        \begin{split}
       \nabla_y\cdot \left(\tilde{\rho} u\right)&=0,    \\
     (u\cdot \nabla_y)u +\frac{\nabla_y p}{\tilde{\rho}}&=0,\\
          u\cdot \nabla_y e +\frac{p}{\tilde{\rho}}\nabla_y\cdot u &=\epsilon_0 \int_{\mathbb{S}^2}dn\ \left[\frac{1-e^{-\frac{2\epsilon_0}{T}}}{1+e^{-\frac{2\epsilon_0}{T}}}G-\frac{e^{-\frac{2\epsilon_0}{T}}}{1+e^{-\frac{2\epsilon_0}{T}}}\right],\\
          n\cdot \nabla_y G&=\epsilon_0 \rho\left[e^{-\frac{2\epsilon_0}{T}} (1+G)-G\right].
        \end{split}
    \end{equation}
Then we now introduce a new scaling limit $|\zeta|+|u|+|\theta|\ll 1$ with $\frac{2\epsilon_0}{T_0}|\theta|\approx 1$, $\zeta=\frac{T_0}{2\epsilon_0}\xi$ and $\theta\approx \frac{T_0}{2\epsilon_0}\vartheta$ such that $(\xi,u,\vartheta,H)$ satisfies the following leading-order system with an exponential dependence in the temperature
 as
\begin{equation}\label{Arrhenius}
        \begin{split}
    &  \tilde{\rho}_0 \nabla_y\cdot u=0,    \\
   & \frac{T_0}{2}\nabla_y(\xi+ \theta)=0,\\
        &  0 =\epsilon_0e^{-\frac{2\epsilon_0}{T_0}} \int_{\mathbb{S}^2}dn\ \left[H-e^{\vartheta}\right],\\
        &  n\cdot \nabla_y H=\epsilon_0 \tilde{\rho}_0\left[ e^{\vartheta}-H\right],
        \end{split}
    \end{equation}as $e^{-\frac{2\epsilon_0}{T}}\ll 1.$
    Using the energy and radiation equations \eqref{Arrhenius}$_3$ and \eqref{Arrhenius}$_4$, we obtain that
    $$\nabla_y\cdot \int_{\mathbb{S}^2}dn\ nH=0.$$
By the mass conservation, we have 
\begin{equation}\label{total mass}
    \int_{\Omega}\tilde{\rho}dy= \int_{\Omega}\tilde{\rho}_0\left(1+\frac{T_0}{2\epsilon_0}\xi\right)dy=M_0\text{ given.}
\end{equation}
Thus, the final model that we will consider is
\begin{equation}\label{eq.Arrhenius.final}
    \begin{split}
& \nabla_y\cdot u=0,    \\
   & \nabla_y(\xi+ \theta)=0,\\
        &  \nabla_y\cdot \int_{\mathbb{S}^2}dn\ nH=0,\\
        &  n\cdot \nabla_y H=b_1\left[ e^{\vartheta}-H\right],\\
     &   \int_\Omega \xi \ dy = b_2,
    \end{split}
\end{equation}where $$b_1=\epsilon_0 \tilde{\rho}_0\ \text{   and   }\ b_2=\frac{2\epsilon_0}{T_0}\left(M_0-\int_\Omega \tilde{\rho}_0\ dy\right).$$
Notice that the emission $b_1e^{\theta}$ depends non-linearly in the rescaled temperature. \subsubsection{In a slab geometry}\label{sec.slab}
As in Section \ref{sec.slab1}, we examine the problem above in a physically interesting slab-geometry situation where we assume that the gas molecules lie in the slab $y_1\in [0,L]$ and that $(\xi,u,\theta,H)$ is symmetric in $y_2$ and $y_3$. We consider the boundary conditions at $y_1=0$ and $=L$ where we have the incident incoming radiation only on one side $y_1=0$ such that
\begin{equation}\label{arrhenius boundary}
    H(0,\psi)=a_+(\psi),\ \text{if }\psi\in(0,\frac{\pi}{2}), \text{ and }H(L,\psi)=a_-(\psi)=0,\ \text{if }\psi\in(\frac{\pi}{2},\pi).
\end{equation} The function $a_+(\psi)$ can be normalized to have $2\pi\int_0^{\pi/2} a_+(\psi)\sin\psi d\psi=1.$

We prove the existence of a unique solution to the boundary value problem in a slab geometry. More precisely:
\begin{theorem}
 The non-linear system in a slab \eqref{eq.Arrhenius.final} with the boundary conditions \eqref{arrhenius boundary} and the mass conservation \eqref{total mass} has a unique solution $(\zeta,\theta)$ with $u=0$.
 \end{theorem}\begin{proof}
Here we claim that the system \eqref{eq.Arrhenius.final} can be written as a linear non-local equation for $e^\vartheta.$ To this end, we define 
$$w=w(y_1)=e^{\vartheta(y_1)}.$$Then by solving \eqref{eq.Arrhenius.final}$_4$, we have 
\begin{equation}\label{eq.H.explicit1}H(y_1,\psi)=a_+(\psi)e^{-\frac{b_1y_1}{\cos\psi}}+\int_0^{y_1}\frac{b_1}{\cos\psi}e^{-\frac{b_1(y_1-z)}{\cos\psi}}w(z)dz\text{ for }0<\psi<\frac{\pi}{2},\end{equation}
and
\begin{equation}\label{eq.H.explicit2}H(y_1,\psi)=\int_{y_1}^L\frac{b_1}{|\cos\psi|}e^{-\frac{b_1(y_1-z)}{\cos\psi}}w(z)dz \text{ for }\frac{\pi}{2}<\psi<\pi,\end{equation}
if $0<y_1<L$.
Now we normalize the constant $b_1=1$ by rescaling $y_1$ and $L$.
By \eqref{eq.Arrhenius.final}$_3$, we have for some given flux $j_0$
$$j_0=\int_{\mathbb{S}^2}nHdn =2\pi \int_0^{\pi} \cos\psi H(y_1,\psi)\sin\psi d\psi,$$ with $n_1=\cos\psi.$
Then by \eqref{eq.H.explicit1}-\eqref{eq.H.explicit2}, we observe that
\begin{equation}\label{eq.j0}
\frac{    j_0}{2\pi}
=S(y_1)+\int_0^{\pi/2}d\psi  \sin\psi\int_0^{y_1}e^{-\frac{y_1-z}{\cos\psi}}w(z)dz-\int_0^{\pi/2}d\psi  \sin\psi\int_{y_1}^Le^{\frac{y_1-z}{\cos\psi}}w(z)dz,
\end{equation}where $S$ is defined as 
$$ S(y_1)\eqdef \int_0^{\pi/2}d\psi \cos\psi \sin\psi a_+(\psi)e^{-\frac{y_1}{\cos\psi}},$$ and we made a change of variables $\psi\mapsto \pi-\psi$ in the last integral. Then we define 
$$Z(y_1)=sgn(y_1)e^{-\frac{|y_1|}{\cos\psi}},\text{ and }\Lambda(y_1)=e^{-\frac{|y_1|}{\cos\psi}}.$$ Then we have 
$$\frac{\partial}{\partial y_1}\Lambda(y_1)=-\frac{Z(y_1)}{\cos\psi}\text{ and }\frac{\partial^2}{\partial y_1^2}\Lambda=-\frac{2\delta(y_1)}{\cos\psi }+\frac{\Lambda}{\cos^2\psi}.$$Then by differentiating \eqref{eq.j0} with respect to $y_1$, we have
\begin{multline}\label{eq.arr1dfinal}
    \partial_{y_1}S=\int_0^{\pi/2}d\psi \cos\psi\sin\psi \int_0^L \partial^2_{y_1}\Lambda (y_1-z)w(z)dz\\
    =-2 w(y_1)+\int_0^{\pi/2}d\psi \tan\psi \int_0^L \Lambda (y_1-z)w(z)dz.
\end{multline}
Note that $$\int_0^{\pi/2}d\psi \ \tan\psi \int_{-\infty}^{\infty}dy\ \Lambda(y_1;\psi)= 2.$$ Therefore, by the contraction mapping principle, we have can prove the existence of a unique solution $w(y_1)$ to the nonlocal equation \eqref{eq.arr1dfinal}, which also uniquely determines $j_0$. Therefore, $\vartheta(y_1)$ is also unique. Finally, we use \eqref{eq.Arrhenius.final}$_2$ and \eqref{eq.Arrhenius.final}$_5$ to conclude that the solution $\xi $ is also uniquely determined. Therefore, the solution is unique for each given $a_+(\psi)$ and $a_-(\psi).$
 \end{proof}
 
 \subsubsection{Higher dimensional case}
 Now, we generalize our previous formulation to a higher dimensional case. More precisely, we consider the system \eqref{eq.Arrhenius.final} where the gas molecules lie in a general 3-dimensional convex domain $\Omega$ with $C^1$ boundary. As before, we rescale $y$ and normalize $b_1=1.$ We first formulate a boundary condition at any given $y_0\in \partial \Omega.$ Define $\nu=\nu_{y_0}$ as the outward normal vector at $y_0$. Then for any $n\in \mathbb{S}^2$ and for $s>0$ the line $y_0-sn$ is contained in $\rth-\Omega$ if $n\cdot \nu_{y_0}<0.$ Then we provide the following incoming boundary condition for each $y_0\in \partial \Omega$: if $n\cdot \nu_{y_0}<0,$ then 
 \begin{equation}
     \label{eq.bc for 3d} 
     H(y_0,n)=f(n),
 \end{equation} for a given profile $f$. This boundary condition describes the physical situation where the radiation is coming from the boundary $|y|=+\infty$ and the incoming radiation at $y_0\in \partial\Omega$ for each incoming direction $n$ is determined by the radiation profile at $|y|=+\infty$ which satisfies 
 $$f(n)=\lim_{s\to \infty} H(-sn,n).$$
 
 \begin{theorem}For each $b_1>0$, $b_2\in \mathbb{R}$, and $\Omega\in\rth $ convex and bounded with $\partial\Omega \in C^1$, there exists a unique solution with $u=0$ to the system with exponential dependence in a higher dimensional case \eqref{eq.Arrhenius.final} with the boundary conditions \eqref{eq.bc for 3d} and the mass conservation \eqref{total mass}.
 \end{theorem}\begin{proof}
    Now, given the incoming boundary condition, we solve \eqref{eq.Arrhenius.final}$_4$ and obtain
    $$H(y_0+sn,n)=f(n)e^{-s}+\int_0^s e^{-(s-\xi)}w(y_0+\xi n)d\xi,\ s>0.$$Alternatively, if we are given $y\in \Omega$ and $n\in\mathbb{S}^2$, there exist unique $y_0=y_0(y,n)\in \partial \Omega$ and $s=s(y,n)$ such that
    $$y=y_0(y,n)+s(y,n)n,$$ and 
    $$H(y,n)=f(n)e^{-s(y,n)}+\int_0^{s(y,n)} e^{-(s(y,n)-\xi)}w(y_0(y,n)+\xi n)d\xi.$$
    Define the flux $\vec{J}=\vec{J}(y)$ as 
    $$\vec{J}\eqdef \int_{\mathbb{S}^2}dn\ nH.$$ Then \eqref{eq.Arrhenius.final}$_3$ implies that $div(\vec{J})=0.$
    Then we further have
    \begin{equation*}
    \vec{J}(y)= \vec{R}(y)+\int_{\mathbb{S}^2}n\left(\int_0^{s(y,n)} e^{-(s(y,n)-\xi)}w(y_0(y,n)+\xi n)d\xi\right)dn,
    \end{equation*}where 
    $$\vec{R}(y)\eqdef \int_{\mathbb{S}^2}nf(n)e^{-s(y,n)}dn.$$
    Thus, we have
    \begin{equation}\label{eq.div.JR}
    0=div(\vec{J})= div(\vec{R})+div\int_{\mathbb{S}^2}n\left(\int_0^{s(y,n)} e^{-(s(y,n)-\xi)}w(y_0(y,n)+\xi n)d\xi\right)dn.
    \end{equation}
    Here we further observe that
    \begin{multline*}
        \int_{\mathbb{S}^2}n\left(\int_0^{s(y,n)} e^{-(s(y,n)-\xi)}w(y_0(y,n)+\xi n)d\xi\right)dn\\
        =\int_{\partial\Omega}dz\int_{\mathbb{S}^2}n\left(\int_0^{s(y,n)} e^{-(s(y,n)-\xi)}w(z+\xi n)\delta(z-y_0(y,n))d\xi\right)dn.
    \end{multline*}
    We make a change of variables $\xi\mapsto \hat{\xi}=s(y,n)-\xi.$ Then we have\begin{multline*}
        \int_{\mathbb{S}^2}n\left(\int_0^{s(y,n)} e^{-(s(y,n)-\xi)}w(y_0(y,n)+\xi n)d\xi\right)dn\\
        =\int_{\partial\Omega}dz\int_{\mathbb{S}^2}n\left(\int_0^{s(y,n)} e^{-\hat{\xi}}w(z+(s-\hat{\xi}) n)\delta(z-y_0(y,n))d\hat{\xi}\right)dn.
    \end{multline*}
    Then we make another change of variables $(
    \hat{\xi},n)\mapsto \eta\eqdef y-\hat{\xi} n \in \Omega$. Since $y$ is independent of $\hat{\xi}$ and $n$, we obtain the Jacobian of the change of variables $J(\eta,n)\eqdef \left|\frac{\partial(\hat{\xi},n)}{\partial \eta}\right|=\frac{1}{|y-\eta|^2}$. Thus, for $n=n(y-\eta)=\frac{y-\eta}{|y-\eta|}$, we have
    \begin{equation*}
     \int_{\partial\Omega}dz\int_{\mathbb{S}^2}n\left(\int_0^{s(y,n)} e^{-\hat{\xi}}w(z+(s-\hat{\xi}) n)\delta(z-y_0(y,n))d\hat{\xi}\right)dn
       =\int_{\Omega}\frac{1}{|y-\eta|^2}\frac{y-\eta}{|y-\eta|} e^{-|y-\eta|}w(\eta)d\eta,
    \end{equation*}since $\hat{\xi}=(y-\eta)\cdot n= |y-\eta|.$
    Then, going back to \eqref{eq.div.JR}, we have
  $$ 0= div(\vec{R})+\int_{\Omega}div\left(\frac{1}{|y-\eta|^2}\frac{y-\eta}{|y-\eta|} e^{-|y-\eta|}\right)w(\eta)d\eta.$$
  Note that 
  \begin{equation*}
      div\left(\frac{1}{|y-\eta|^2}\frac{y-\eta}{|y-\eta|} e^{-|y-\eta|}\right)=div \left(\frac{e^{-r}}{r^2}\hat{r}\right)=-\frac{e^{-r}}{r^2}+4\pi\delta(r).
  \end{equation*}  Therefore, we finally have
\begin{equation}\label{contmap.eq}w(y)=\int_{\Omega}\frac{e^{-|y-\eta|}}{4\pi|y-\eta|^2} w(\eta)d\eta- \frac{1}{4\pi}div(\vec{R}).\end{equation}
This is a Dirichlet problem for an non-local elliptic operator by letting $w(\eta)=0$ if $\eta\notin \Omega$ and $w\in L^\infty(\Omega)$. Note that $\int_{\Omega}\frac{e^{-|y-\eta|}}{4\pi |y-\eta|^2}dy<1$ and we obtain a convergent series. Then a unique solution $w$ exists as it can be readily seen from \eqref{contmap.eq} as a fixed-point problem using the contraction mapping principle.  This completes the proof of existence of the stationary solutions to  \eqref{eq.Arrhenius.final}.
\end{proof}
\begin{remark}
The term $-\frac{1}{4\pi}div(\vec{R})$ in \eqref{contmap.eq} must be positive in order to make sure that $w=e^\theta$ is positive. The positivity of $-\frac{1}{4\pi}div(\vec{R})$ can be shown simply using the definition of $\vec{R}$ $$\vec{R}(y)\eqdef \int_{\mathbb{S}^2}nf(n)e^{-s(y,n)}dn,$$and $\frac{\partial s}{\partial n}>0.$
\end{remark}

\section{An Euler-like system arising in a non-LTE scaling limit}\label{sec.Euler nonLTE} 
    In this section, we consider a general non-LTE situation where the equilibria for the ground- and excited-states molecules are not related by the Boltzmann ratio. In this case, our goal is to derive the system of Euler equations via a Chapman-Enskog expansion. 
    \subsection{Maxwellian equilibria}In the case of a non-LTE situation, we define and consider the local Maxwellian equilibria $M^{(j)}$ separately for each type of molecules $j=1,2$ as \begin{equation}
    \label{nonLTE equilibrium}
    M^{(j)}= M^{(j)}(\rho_j,u,T)\eqdef \frac{c_0\rho_j}{T^{\frac{3}{2}}}\exp\left(-\frac{|v-u|^2}{T}\right),\ j=1,2.
\end{equation} Here the mass densities, macroscopic velocity, and temperature are defined as
\begin{equation}
    \begin{split}
        &\rho_j\eqdef \int_\rth F^{(j)}dv,\\
         &u\eqdef \frac{1}{\rho_j}\int_\rth v F^{(j)}dv,\text{ for }i=1,2,3,\\
         &T \eqdef \frac{2}{3\rho_j}\int_\rth |v-u |^2F^{(j)}dv,
    \end{split}\end{equation}where we recall that the Boltzmann constant is $\frac{1}{2}$ throughout the paper. Compare this equilibria to the LTE equilibrium in \eqref{Feq}. Note that the macroscopic velocities and the temperatures for both types of molecules are the same in the Maxwellian equilibria \eqref{nonLTE equilibrium} because we assume that there is a meaningful exchange of momentum and energy by the elastic collisions $\frac{1}{\alpha}\mathcal{K}_{el}[F^{(1)},F^{(2)}]$  between the two species.
    \subsection{Chapman-Enskog expansion and the Euler-radiation limit yielding non-LTE}\label{sec.chapman.nonlte}
Then we will consider the Chapman-Enskog expansion \eqref{Chapman} on the equation \eqref{eq.rescaled kinetic equations} with $\sigma\eqdef \frac{\eta}{\alpha}$ as follows:
\begin{equation}
    \label{eq.chapman.euler.nonlte}
    \Fj=M^{(j)}(1+\alpha f^{(j)}_1+\cdots ).
\end{equation}
The key difference is also on separate $L^2$ inner products for the density equations for each type of molecules $j=1,2$, since we want to consider the non-LTE situation where the molecule densities are not related by the Boltzmann ratio.
More precisely, we use the following $L^2$ inner product for the density equations for each species $j=1,2$ as 
\begin{equation}\label{new L2 inner product}\langle \Fj, G^{(j)}\rangle \eqdef \int_{\mathbb{R}^3}dv\ \Fj\cdot G^{(j)},\ j=1,2.\end{equation} For the velocity and the energy equations, we use \eqref{eq.L2 inner product}.
Now we postulate the Chapman-Enskog expansion \eqref{eq.chapman.euler.nonlte} in \eqref{eq.rescaled kinetic equations} in the limit $\alpha \to 0^+$ and $\sigma\approx 1$. Then the leading order equation in $\alpha$ is given by
\begin{multline*}
   \partial_t M+ v\cdot \nabla_y M 
   = L_{el}[M;f_1] +\sigma L_{non.el}[M]+\mathcal{R}_p[M,G]\\=\begin{pmatrix}L^{(1,1)}_{el}[M;f_1]+L^{(1,2)}_{el}[M;f_1]\\L^{(2,1)}_{el}[M;f_1]+
L^{(2,2)}_{el}[M;f_1]\end{pmatrix}+\sigma \begin{pmatrix}2 L_{1,1}[M]+ L_{1,2}^{(1)}[M]\\ L^{(2)}_{non.el}[M]\end{pmatrix}+\mathcal{R}_p[M,G],
\end{multline*}
where $M=(M^{(1)},M^{(2)})^\top,$ $\rho=(\rho_1,\rho_2)^\top,$ $u=(u_1,u_2,u_3)$ and $T$ denote the equilibria, the mass densities, the macroscopic velocity, and the temperature, respectively, and the linear operators are defined as
\begin{equation}\notag
\begin{split}
L^{(i,j)}_{el}[M;f_1]&\eqdef\mathcal{K}^{(i,j)}_{el}[M^{(i)},M^{(j)}f_1^{(j)}]+\mathcal{K}^{(i,j)}_{el}[M^{(i)}f_1^{(i)},M^{(j)}],\\
L_{(1,1)}[M]&\eqdef\mathcal{K}_{(1,1)}[M,M],\\
L^{(1)}_{1,2}[M]&\eqdef\mathcal{K}^{(1)}_{1,2}[M,M],\text{ and }\\
L^{(2)}_{non.el}[M]&\eqdef\mathcal{K}^{(2)}_{non.el}[M,M].
\end{split}
\end{equation}
 By taking the inner product \eqref{new L2 inner product} with respect to $v$ variable against $1$ for the densities, we have
 \begin{equation}\label{Euler.nonLTE density}
    \begin{split}
        &\partial_t \rho_1+\nabla_y\cdot  (\rho_1u)=\sigma H^{(1)}+Q^{(1)},\\
         &\partial_t \rho_2+\nabla_y\cdot  (\rho_2u)= \sigma H^{(2)}+Q^{(2)},
    \end{split}
\end{equation}
where \begin{equation}\label{Hj def} H^{(j)}\eqdef \int \Knj [M,M]dv,\end{equation}
\begin{equation} Q^{(j)}\eqdef \int \Kradj [M,G]dv.\end{equation}Here note that the terms $H^{(1)}$, $H^{(2)}$, $Q^{(1)}$, and $Q^{(2)}$ do not vanish via the $L^2$ inner product as in Section \ref{sec.Euler limit}, since we take the $L^2$ inner product separately for each species $j=1,2$ (c.f. \eqref{eq.L2 inner product}).
 On the other hand, for the velocity and the energy equations, we cannot separate equations as in \eqref{Euler.nonLTE density} because there is exchange of momentum and energy by the elastic collisions $\frac{1}{\alpha}\mathcal{K}_{el}[F^{(1)},F^{(2)}]$  between the two species. Therefore, we take the inner product \eqref{eq.L2 inner product} against $\begin{pmatrix}v-u\\v-u\end{pmatrix}$ for the velocity and $\frac{4}{3}\begin{pmatrix}\frac{|v-u |^2}{2}\\\frac{|v-u |^2}{2}+\epsilon_0\end{pmatrix}$ for the energy, and obtain the velocity and energy equations for both molecules as
\begin{equation}\label{Euler.nonLTE velocity energy}
    \begin{split}
          &\partial_t ((\rho_1+\rho_2)u)+\nabla_y\cdot  ((\rho_1+\rho_2)u\otimes u)+\nabla_y\cdot (S^{(1)}+S^{(2)})\\&\qquad\qquad\qquad\qquad\qquad\qquad=\sigma(J^{(1)}_m+J^{(2)}_m)+(\Sigma^{(1)}+\Sigma^{(2)}),\\
        &\partial_t ((\rho_1+\rho_2)T)+\nabla_y\cdot  ((\rho_1+\rho_2)uT+(J_q^{(1)}+J_q^{(2)}))\\&\qquad\qquad\qquad\qquad\qquad\qquad=\sigma(J^{(1)}_e+J^{(2)}_e)+J^{(1)}_r+J^{(2)}_r,
    \end{split}
\end{equation}where\begin{equation}
    S^{(j)}\eqdef \int (v-u )\otimes (v-u )M^{(j)} dv,
\end{equation}
\begin{equation}
    J_m^{(j)}\eqdef \int v\Knj[M,M]dv,
\end{equation}
\begin{equation}
    \Sigma^{(j)}\eqdef \int v\Kradj[M,G]dv,
\end{equation}
\begin{equation}
    J_q^{(j)}\eqdef \int \frac{4}{3}\left(\frac{|v-u |^2}{2}+\epsilon_0\delta_{j,2}\right)(v-u )M^{(j)} dv=0,
\end{equation}since the integrand is odd in $v$, and
\begin{equation}
    J^{(j)}_e=\int \frac{4}{3}\left(\frac{|v-u |^2}{2}+\epsilon_0\delta_{j,2}\right) \Knj[M,M]dv,
\end{equation}
and
\begin{equation}\label{J2r def}
    J^{(j)}_r\eqdef \int \frac{4}{3}\left(\frac{|v-u |^2}{2}+\epsilon_0\delta_{j,2}\right)\mathcal{R}_p^{(j)}[M,G]dv,
\end{equation}
for $j=1,2$ where we use the notation
$$\mathcal{K}^{(1)}_{non.el}[M,M]\eqdef 2 \mathcal{K}_{(1,1)}[M,M]+\mathcal{K}^{(1)}_{1,2}[M,M],$$ with all other notations given in \eqref{eq.Rp}-\eqref{Kntw}.

 \subsection{Reduction of the operators}\label{sec.representations}
 Here we reduce each operator of  \eqref{Euler.nonLTE} more explicitly below. 
Since the total mass and momentum of molecules is conserved in the non-elastic collisions and the transfer of mass and momentum due to the radiation is also conserved, we have
$$ H^{(1)}=-H^{(2)},\ Q^{(1)}=-Q^{(2)},\ J^{(1)}_m=-J^{(2)}_m,\ \text{and}\ \Sigma^{(1)}=-\Sigma^{(2)}.$$ Also, by definition, we note that 
$J^{(1)}_r+J^{(2)}_r=\frac{4\epsilon_0}{3}Q^{(2)}.$ 
 In order to see the relationship between $J^{(1)}_e$ and $J^{(2)}_e$, we first recall the conservation of total energy \eqref{EnergyInel} during each  non-elastic collision. The conservation of energy implies that for a general distribution $F$, we obtain
$$\int_\rth \frac{1}{2}|v|^2\mathcal{K}^{(1)}_{non.el}[F,F]dv + \int_\rth \left(\frac{1}{2}|v|^2+\epsilon_0\right)\mathcal{K}^{(2)}_{non.el}[F,F]dv =0,$$ for \eqref{eq.knonel1} and \eqref{eq.knonel2}.  Then we deduce that
\begin{multline*}
   0= \int_\rth\frac{1}{2}|v- u+ u|^2\mathcal{K}^{(1)}_{non.el}[M,M]dv + \int_\rth \left(\frac{1}{2}|v- u+ u|^2+\epsilon_0\right)\mathcal{K}^{(2)}_{non.el}[M,M]dv \\
   =\frac{3}{4}J^{(1)}_e + \frac{3}{4}J^{(2)}_e- u\cdot J^{(2)}_m +  u\cdot J^{(2)}_m
   +\frac{1}{2}| u|^2H^{(2)} -\frac{1}{2}| u|^2H^{(2)}.
\end{multline*}
Therefore we have
\begin{equation}\notag
    J^{(1)}_e + J^{(2)}_e=0.
\end{equation}
Consequently, we finally have the Euler-like system coupled with radiation for the non-LTE situation as
\begin{equation}\label{Euler.nonLTE}
    \begin{split}
              &\partial_t \rho_1+\nabla_y\cdot  (\rho_1u)=\sigma H^{(1)}+Q^{(1)},\\
         &\partial_t \rho_2+\nabla_y\cdot  (\rho_2u)= \sigma H^{(2)}+Q^{(2)}\\ &\partial_t ((\rho_1+\rho_2)u)+\nabla_y\cdot  ((\rho_1+\rho_2)u\otimes u)+\nabla_y\cdot (S^{(1)}+S^{(2)})=0,\\
        &\partial_t ((\rho_1+\rho_2)T)+\nabla_y\cdot  ((\rho_1+\rho_2)uT)=\frac{4\epsilon_0}{3}Q^{(2)}.
    \end{split}
\end{equation}

Regarding the equations for mass \eqref{Euler.nonLTE}$_1$ and \eqref{Euler.nonLTE}$_2$, we first observe that we can further write the operators as
\begin{multline}\label{eq.H.reduction.initial}
    H^{(2)}= -H^{(1)} =\int_{\mathbb{R}^3}dv\ \Kntw[M,M](v)\\
    =\rho_1^2 e^{-\frac{2\epsilon_0}{T}}\int_{\mathbb{R}^3}dv \int_\rth dv_4 \int_\stw d\omega\ W_+ (v,v_4;v_1,v_2)\mathcal{Z}(v,u,T)\mathcal{Z}(v_4,u,T)\\
    -\rho_1\rho_2\int_{\mathbb{R}^3}dv \int_\rth dv_4 \int_\stw d\omega\ W_+ (v,v_4;v_1,v_2)\mathcal{Z}(v,u,T)\mathcal{Z}(v_4,u,T),
\end{multline}
where $W_+$ is defined in \eqref{W+def} and
$$\mathcal{Z}(v,u ,T )\eqdef \frac{M^{(j)}(v)}{\rho_j}.$$
Define 
\begin{equation}\notag\mathcal{P}(T;u)\eqdef \int_{\mathbb{R}^3}dv \int_\rth dv_4 \int_\stw d\omega\ W_+ (v,v_4;v_1,v_2)\mathcal{Z}(v,u,T)\mathcal{Z}(v_4,u,T).\end{equation}
Since $W_+$ is invariant under the following Galilean transformation $$W_+(v,v_4;v_1,v_2)=W_+(v+U,v_4+U;v_1+U,v_2+U)\text{ for any }U\in \rth,$$ we obtain
\begin{equation}\label{eq.H.reduction.final}
    H^{(2)}=-H^{(1)}
    =\rho_1^2 e^{-\frac{2\epsilon_0}{T}}\mathcal{P}(T;u)
    -\rho_1\rho_2\mathcal{P}(T;u)
    =(\rho_1^2 e^{-\frac{2\epsilon_0}{T}}
    -\rho_1\rho_2)\mathcal{P}(T;0).
\end{equation}
Regarding the radiation terms, we have
$$
    Q^{(1)}=-Q^{(2)}=\rho_2\int_{\mathbb{S}^2}dn
\ (1+G)-\rho_1\int_{\mathbb{S}^2}dn
\ G,
$$ by definition of $\mathcal{R}_p$.

In the velocity equations \eqref{Euler.nonLTE}$_3$, we first note that $\frac{1}{3}$ of the trace of $S^{(j)}$ is equal to the pressure $p^{(j)}=\frac{1}{2}\rho_jT $. Furthermore, we have
$$
  S^{(j)}= \int (v-u )\otimes (v-u )M^{(j)} dv
    =\frac{1}{3}I\int |v-u |^2M^{(j)} dv    =\frac{1}{2}\rho_jT  I=p^{(j)}I,
$$where $I$ is the 3 by 3 identity matrix.

This completes the derivation of the Euler-radiation limit in a scaling limit yielding a non-LTE. In the next subsection, we will study the existence of a stationary solution to the Euler system in the non-LTE situation.
\subsection{Stationary solutions}\label{sec.nonLTEexist}
In this subsection, we will study the stationary solution of zero macroscopic velocities to the Euler-like system coupled with radiation in the non-LTE situation that we have obtained in Section \ref{sec.Euler nonLTE}. Here the situation with the zero macroscopic velocity physically means that there is no convection and all the heat transfer takes place by means of radiation. Throughout this subsection, we will see that if $u=0$ then we recover LTE. 
We first recall the system \eqref{Euler.nonLTE} with the further reduction of each operator introduced in Section \ref{sec.representations}. Then we obtain the following time-independent stationary system with $u=0$:
\begin{equation}\notag
    \begin{split}
        &\sigma H^{(1)}+Q^{(1)}=-\sigma H^{(2)}-Q^{(2)}=0,\\
          &\nabla_y\cdot S^{(1)}=-\nabla_y\cdot S^{(2)}=0,\\
            &  \frac{4\epsilon_0}{3}Q^{(2)}=0,
    \end{split}
\end{equation}where the operators are represented as$$
   H^{(2)}=-H^{(1)}
    =(\rho_1^2 e^{-\frac{2\epsilon_0}{T}}
    -\rho_1\rho_2)\mathcal{P}(T;0),$$
$$
    Q^{(1)}=-Q^{(2)}=\rho_2\int_{\mathbb{S}^2}dn
\ (1+G)-\rho_1\int_{\mathbb{S}^2}dn
\ G,
$$ 
$$ S^{(j)}=p^{(j)} I=\frac{1}{2}\rho_jT I.$$
We also recall \eqref{eq.rescaled kinetic equations}$_2$ and \eqref{eq.Rr} and obtain the stationary equation for the radiation
\begin{equation}\notag
    n\cdot\nabla_y G =\epsilon_0\int_{\rth}[M^{(2)}(1+G)-M^{(1)} G]dv=\epsilon_0(\rho_2 (1+G)-\rho_1 G).
\end{equation}
Finally, we have the total mass conservation that 
\begin{equation}\notag
\int_{\Omega_y} dy\  (\rho_1+\rho_2)=m_0,    
\end{equation}for some given constant $m_0>0$ where $\Omega_y$ is the rescaled domain of $\Omega$ in the variable $y$.

Then since $Q^{(2)}=0,$ we have $H^{(1)}=H^{(2)}=0.$ Hence, we have
$$\rho_2=\rho_1e^{-\frac{2\epsilon_0}{T}},$$ and this yields back the LTE situation in Section \ref{sec.LTEEulerlimit} and Section \ref{sec.LTEexist}.  In particular, we have shown that the system is well-posed under the linearization and under the exponential-dependence limit. This completes the discussion of the stationary solutions to the non-LTE system with zero macroscopic velocities. A more complicated system yielding non-LTE steady states will be discussed in Section \ref{sec.nonLTE.3level}.

\section{A non-LTE system with different temperatures for the two-molecule states: Non existence of stationary solutions }\label{sec.nonLTEnonexist}
In this section, we will prove the non-existence of a stationary solution of zero macroscopic velocities to the Euler-like system coupled with radiation in the non-LTE situation under some additional conditions on the elastic collisions. More precisely, we will assume that the elastic collisions between the ground-state molecules and the excited-state molecules $\mathcal{K}^{(1,2)}_{el}$ and $\mathcal{K}^{(2,1)}_{el}$ in \eqref{eq.rescaled kinetic equations} are much smaller in the scale than the elastic collisions between the ground-states $\mathcal{K}^{(1,1)}_{el}$ so that the elastic collisions between different species are neglected.  We emphasize that we consider in this section the situations in which the elastic collisions $A+\bar{A}\leftrightarrows A+\bar{A}$ take place less often than the elastic collisions $A+A\leftrightarrows A+A$ and $\bar{A}+\bar{A}\leftrightarrows \bar{A}+\bar{A}$. In this case, the two different species can have different temperatures and velocities, since there is not enough mixture via the elastic collisions. Then the two types of molecules have the two different local Maxwellian equilibria $M^{(j)}$ for each type of molecules $j=1,2$ as follows:
\begin{equation}
    \label{different maxwellain equilibria}
    M^{(j)}= M^{(j)}(\rho_j,u_j,T_j)\eqdef \frac{c_0\rho_j}{T_j^{\frac{3}{2}}}\exp\left(-\frac{|v-u_j|^2}{T_j}\right),\ j=1,2.
\end{equation} Here the mass densities, macroscopic velocities, and temperatures are defined as
\begin{equation}\notag
    \begin{split}
        &\rho_j\eqdef \int_\rth F^{(j)}dv,\\
         &u_j\eqdef \frac{1}{\rho_j}\int_\rth v F^{(j)}dv,\text{ for }i=1,2,3,\\
         &T_j \eqdef \frac{2}{3\rho_j}\int_\rth |v-u_j |^2F^{(j)}dv,
    \end{split}\end{equation} with the Boltzmann constant equal to $\frac{1}{2}$.
  \subsection{Chapman-Enskog expansion and the Euler-radiation limit yielding non-LTE}
Then we will consider the Chapman-Enskog expansion \eqref{Chapman} on the equation \eqref{eq.rescaled kinetic equations} with $\sigma\eqdef \frac{\eta}{\alpha}$ as follows:
\begin{equation}
    \label{eq.chapman.euler.nonlte2}
    \Fj=M^{(j)}(1+\alpha f^{(j)}_1+\cdots ).
\end{equation}
As before, we will use the inner product \eqref{new L2 inner product}. 
Now we postulate the Chapman-Enskog expansion \eqref{eq.chapman.euler.nonlte2} in \eqref{eq.rescaled kinetic equations} in the limit $\alpha \to 0^+$ and $\sigma\approx 1$. Then the leading order equation in $\alpha$ is given by
\begin{multline*}
   \partial_t M+ v\cdot \nabla_y M 
   = L_{el}[M;f_1] +\sigma L_{non.el}[M]+\mathcal{R}_p[M,G]\\=\begin{pmatrix}L^{(1,1)}_{el}[M;f_1]+L^{(1,2)}_{el}[M;f_1]\\L^{(2,1)}_{el}[M;f_1]+
L^{(2,2)}_{el}[M;f_1]\end{pmatrix}+\sigma \begin{pmatrix}2 L_{1,1}[M]+ L_{1,2}^{(1)}[M]\\ L^{(2)}_{non.el}[M]\end{pmatrix}+\mathcal{R}_p[M,G],
\end{multline*}
where $M=(M^{(1)},M^{(2)})^\top,$ $\rho=(\rho_1,\rho_2)^\top,$ $u_i=(u_{1,i},u_{2,i})^\top\text{ for }i=1,2,3,$ and $T=(T_1,T_2)^\top,$ denote the equilibria, mass densities, macroscopic velocities, and temperatures, respectively, and the linear operators are defined as
\begin{equation}\notag
\begin{split}
L^{(i,j)}_{el}[M;f_1]&\eqdef\mathcal{K}^{(i,j)}_{el}[M^{(i)},M^{(j)}f_1^{(j)}]+\mathcal{K}^{(i,j)}_{el}[M^{(i)}f_1^{(i)},M^{(j)}],\\
L_{(1,1)}[M]&\eqdef\mathcal{K}_{(1,1)}[M,M],\\
L^{(1)}_{1,2}[M]&\eqdef\mathcal{K}^{(1)}_{1,2}[M,M],\text{ and }\\
L^{(2)}_{non.el}[M]&\eqdef\mathcal{K}^{(2)}_{non.el}[M,M].
\end{split}
\end{equation}
 By taking the inner product with respect to $v$ variable against $1, (v-u_j),$ and $\frac{4}{3}\left(\frac{|v-u_j|^2}{2}+\epsilon_0\delta_{j,2}\right)$ for each molecule $j=1,2$ separately, we obtain the following system of hydrodynamic equations for each species $j=1,2$:
\begin{equation}\label{Euler.nonLTE2}
    \begin{split}
        &\partial_t \rho_1+\nabla_y\cdot  (\rho_1u_1)=\sigma H^{(1)}+Q^{(1)},\\
         &\partial_t \rho_2+\nabla_y\cdot  (\rho_2u_2)= \sigma H^{(2)}+Q^{(2)},\\
          &\partial_t (\rho_1u_1)+\nabla_y\cdot  (\rho_1u_1\otimes u_1)+\nabla_y\cdot S^{(1)}=\sigma J_m^{(1)}+\Sigma^{(1)},\\
            &\partial_t (\rho_2u_2)+\nabla_y\cdot  (\rho_2u_2\otimes u_2)+\nabla_y\cdot S^{(2)}=\sigma J_m^{(2)}+\Sigma^{(2)},\\
        &\partial_t (\rho_1T_1)+\nabla_y\cdot  (\rho_1u_1T_1+J^{(1)}_q)= \sigma J^{(1)}_e+J^{(1)}_r,\\
            &\partial_t \left(\rho_2T_2+\frac{4}{3}\epsilon_0 \rho_2\right)+\nabla_y\cdot  \left(\rho_2u_2T_2+\frac{4}{3}\epsilon_0u_2\rho_2+J^{(2)}_q\right)=  \sigma J^{(2)}_e+J^{(2)}_r,
    \end{split}
\end{equation}where\begin{equation}\notag H^{(j)}\eqdef \int \Knj [M,M]dv,\end{equation}
\begin{equation}\notag Q^{(j)}\eqdef \int \Kradj [M,G]dv,\end{equation}
\begin{equation}\notag
    J_m^{(j)}\eqdef \int v\Knj[M,M]dv,
\end{equation}
\begin{equation}\notag
    \Sigma^{(j)}\eqdef \int v\Kradj[M,G]dv,
\end{equation}
\begin{equation}\notag
    S^{(j)}\eqdef \int (v-u_j)\otimes (v-u_j)M^{(j)} dv,
\end{equation}
\begin{equation}\notag
    J_q^{(j)}\eqdef \int \frac{4}{3}\left(\frac{|v-u_j|^2}{2}+\epsilon_0\delta_{j,2}\right)(v-u_j)M^{(j)} dv=0,
\end{equation}since the integrand is odd in $v$, and
\begin{equation}\notag
    J^{(j)}_e=\int \frac{4}{3}\left(\frac{|v-u_j|^2}{2}+\epsilon_0\delta_{j,2}\right) \Knj[M,M]dv,
\end{equation}
and
\begin{equation}\notag
    J^{(j)}_r\eqdef \int \frac{4}{3}\left(\frac{|v-u_j|^2}{2}+\epsilon_0\delta_{j,2}\right)\mathcal{R}_p^{(j)}[M,G]dv,
\end{equation}
for $j=1,2$ where we defined
$$\mathcal{K}^{(1)}_{non.el}[M,M]\eqdef 2 \mathcal{K}_{(1,1)}[M,M]+\mathcal{K}^{(1)}_{1,2}[M,M],$$ with all other notations given in \eqref{eq.Rp}-\eqref{Kntw}. 
Here note that the terms $H^{(1)}$, $H^{(2)}$, $Q^{(1)}$, $Q^{(2)}$, $\Sigma^{(1)}$, and $\Sigma^{(2)}$ do not vanish via the $L^2$ inner product as in Section \ref{sec.Euler limit}, since we take the $L^2$ inner product separately for each species $j=1,2$ (c.f. \eqref{eq.L2 inner product}).

In this situation, we will prove that the stationary solutions to the system \eqref{Euler.nonLTE2} with $u_1=u_2=0$ can exist only in a very limited situation with a specific incoming radiation. Otherwise, in a general situation with generic boundary conditions, we claim that there exists no stationary solutions with $u_1=u_2=0$.

\begin{remark}
We remark here that the assumptions that the elastic collisions between the ground-state and the excited-state $\mathcal{K}^{(1,2)}_{el}$ and $\mathcal{K}^{(2,1)}_{el}$ are much smaller in the scale than the elastic collisions between the ground-states $\mathcal{K}^{(1,1)}_{el}$ can be probably a bit artificial in actual physical situations. However, we introduce this special situation because the dynamics and the aftermath here is mathematically interesting.  
\end{remark}

 \subsection{Reduction of the operators} As we did in Section \ref{sec.representations}, we now reduce each operator in a general case when there are two different velocities $u_1$ and $u_2$ and two different temperatures $T_1$ and $T_2$. As before, since the total mass and momentum of molecules is conserved in the non-elastic collisions and the transfer of mass and momentum due to the radiation is also conserved, we have
$$ H^{(1)}=-H^{(2)},\ Q^{(1)}=-Q^{(2)},\ J^{(1)}_m=-J^{(2)}_m,\ \text{and}\ \Sigma^{(1)}=-\Sigma^{(2)},$$ and hence it is more convenient to use $-\mathcal{K}^{(2)}_{non.el}$ when we compute $ H^{(1)}$ and $J^{(1)}_m$. Here we reduce each operator of  \eqref{Euler.nonLTE} more explicitly below.

\subsubsection{Terms in the mass equations}\label{sec.reduction.H}
Regarding the equations for mass \eqref{Euler.nonLTE}$_1$ and \eqref{Euler.nonLTE}$_2$, we first observe that we can further write the operators as
\begin{multline*}
    H^{(2)}= -H^{(1)} =\int_{\mathbb{R}^3}dv\ \Kntw[M,M](v)\\
    =\rho_1^2 e^{-\frac{2\epsilon_0}{T_1}}\int_{\mathbb{R}^3}dv \int_\rth dv_4 \int_\stw d\omega\ W_+ (v,v_4;v_1,v_2)\mathcal{Z}(v,u_1,T_1)\mathcal{Z}(v_4,u_1,T_1)\\
    -\rho_1\rho_2\int_{\mathbb{R}^3}dv \int_\rth dv_4 \int_\stw d\omega\ W_+ (v,v_4;v_1,v_2)\mathcal{Z}(v,u_2,T_2)\mathcal{Z}(v_4,u_1,T_1),
\end{multline*}
where 
$$\mathcal{Z}(v,u_j,T_j)\eqdef \frac{M^{(j)}(v)}{\rho_j}=\frac{c_0}{T_j^{\frac{3}{2}}}\exp\left(-\frac{|v-u_j|^2}{T_j}\right),\ j=1,2.$$
Define 
\begin{equation}\label{eq.P}\mathcal{P}(T_k,T_l,u_k,u_l)\eqdef \int_{\mathbb{R}^3}dv \int_\rth dv_4 \int_\stw d\omega\ W_+ (v,v_4;v_1,v_2)\mathcal{Z}(v,u_k,T_k)\mathcal{Z}(v_4,u_l,T_l).\end{equation} Throughout the paper, we will simply denote
$$\mathcal{P}(T_k,T_l)\eqdef \mathcal{P}(T_k,T_l,0,0)\text{ and }\mathcal{P}(T_j)\eqdef \mathcal{P}(T_j,T_j). $$
Since $W_+$ is invariant under the following Galilean transformation $$W_+(v,v_4;v_1,v_2)=W_+(v+U,v_4+U;v_1+U,v_2+U)\text{ for any }U\in \rth,$$ we have
\begin{multline*}
    H^{(2)}=-H^{(1)}
    =\rho_1^2 e^{-\frac{2\epsilon_0}{T_1}}\mathcal{P}(T_1;u_1)
    -\rho_1\rho_2\mathcal{P}(T_2,T_1,u_2,u_1)
    =\rho_1^2 e^{-\frac{2\epsilon_0}{T_1}}\mathcal{P}(T_1)
    -\rho_1\rho_2\mathcal{P}(T_2,T_1,u_2-u_1,0).
\end{multline*}
Regarding the radiation terms, we have
$$
    Q^{(1)}=-Q^{(2)}=\rho_2\int_{\mathbb{S}^2}dn
\ (1+G)-\rho_1\int_{\mathbb{S}^2}dn
\ G,
$$ by definition of $\mathcal{R}_p$.

\subsubsection{Terms in the velocity equations}
In the velocity equations \eqref{Euler.nonLTE}$_3$ and \eqref{Euler.nonLTE}$_4$, we first note that $\frac{1}{3}$ of the trace of $S^{(j)}$ is equal to the pressure $p^{(j)}=\frac{1}{2}\rho_jT_j$. Furthermore, we have
$$
  S^{(j)}= \int (v-u_j)\otimes (v-u_j)M^{(j)} dv
    =\frac{1}{3}I\int |v-u_j|^2M^{(j)} dv    =\frac{1}{2}\rho_jT_j I=p^{(j)}I,
$$where $I$ is the 3 by 3 identity matrix.
In addition, note that
\begin{multline*}
    J^{(2)}_m=-J^{(1)}_m=\int_\rth v\mathcal{K}^{(2)}_{non.el}[M,M]dv\\
    =\rho_1^2 e^{-\frac{2\epsilon_0}{T_1}}\int_{\mathbb{R}^3}vdv \int_\rth dv_4 \int_\stw d\omega\ W_+ (v,v_4;v_1,v_2)\mathcal{Z}(v,u_1,T_1)\mathcal{Z}(v_4,u_1,T_1)\\
    -\rho_1\rho_2\int_{\mathbb{R}^3}vdv \int_\rth dv_4 \int_\stw d\omega\ W_+ (v,v_4;v_1,v_2)\mathcal{Z}(v,u_2,T_2)\mathcal{Z}(v_4,u_1,T_1)\\
   =(\rho_1^2 e^{-\frac{2\epsilon_0}{T_1}}
    -\rho_1\rho_2)\mathcal{I}
    +u_1\rho_1^2 e^{-\frac{2\epsilon_0}{T_1}}\mathcal{P}(T_1;u_1)
    -u_2\rho_1\rho_2\mathcal{P}(T_2,T_1,u_2,u_1) ,
\end{multline*}due to the Galilean invariance of the collision kernel
where \begin{equation}\label{eq.I}\mathcal{I}\eqdef \int_{\mathbb{R}^3}vdv \int_\rth dv_4 \int_\stw d\omega\ W_+ (v,v_4;v_1,v_2)\mathcal{Z}(v,0,T_1)\mathcal{Z}(v_4,0,T_1).\end{equation} We note that $\mathcal{I}=0$. This is because $\mathcal{I}$ is a vector and $\mathcal{Z}(v,0,T_1)$ and $\mathcal{Z}(v_4,0,T_1)$ are invariant under rotations $v\mapsto Rv$ and $v_4\mapsto Rv_4$ with $R\in SO(3)$ and hence $R\mathcal{I}=\mathcal{I}$ for any $R\in O(3).$ Therefore, $\mathcal{I}=0.$

Regarding the radiation terms, we have
\begin{multline*}
    \Sigma^{(1)}=-\Sigma^{(2)}=\int_\rth v\mathcal{R}_p^{(1)}[M,G]dv=\int_\rth vdv\int_{\mathbb{S}^2}dn
\ [M^{(2)}(1+G)-M^{(1)}G]\\
=\rho_2u_2\int_{\mathbb{S}^2}dn
\ (1+G)-\rho_1u_1\int_{\mathbb{S}^2}dn
\ G.\end{multline*}This expression physically means that a molecule changes its state not modifying its momentum because the momentum of a photon is assumed to be zero. 
\subsubsection{Terms in the temperature equations}
We now consider the equations for the temperatures \eqref{Euler.nonLTE}$_5$ and  \eqref{Euler.nonLTE}$_6$. In order to see the relationship between $J^{(1)}_e$ and $J^{(2)}_e$, we first recall the conservation of total energy \eqref{EnergyInel} during each  non-elastic collision. The conservation of energy implies that for a general distribution $F$, we have
$$\int_\rth \frac{1}{2}|v|^2\mathcal{K}^{(1)}_{non.el}[F,F]dv + \int_\rth \left(\frac{1}{2}|v|^2+\epsilon_0\right)\mathcal{K}^{(2)}_{non.el}[F,F]dv =0,$$ for \eqref{eq.knonel1} and \eqref{eq.knonel2}.  Then we have
\begin{multline*}
   0= \int_\rth\frac{1}{2}|v-u_1+u_1|^2\mathcal{K}^{(1)}_{non.el}[M,M]dv + \int_\rth \left(\frac{1}{2}|v-u_2+u_2|^2+\epsilon_0\right)\mathcal{K}^{(2)}_{non.el}[M,M]dv \\
   =\frac{3}{4}J^{(1)}_e + \frac{3}{4}J^{(2)}_e-u_1\cdot J^{(2)}_m + u_2\cdot J^{(2)}_m
   +\frac{1}{2}|u_1|^2H^{(2)} -\frac{1}{2}|u_2|^2H^{(2)}.
\end{multline*}
Therefore we have
\begin{equation}\label{eq.relationJ1eJ2e}
    \frac{3}{4}(J^{(1)}_e + J^{(2)}_e)=J^{(2)}_m\cdot (u_1-u_2)-\frac{1}{2}H^{(2)}(|u_1|^2-|u_2|^2).
\end{equation}
This also implies that if $u_1=u_2,$ $J^{(1)}_e=-J^{(2)}_e.$
For the representation of the terms $J^{(1)}_e,\ J^{(2)}_e$, we now observe that
\begin{multline*}
  \int_\rth |v-u_2|^2\mathcal{K}^{(2)}_{non.el}[M,M]dv
    =|u_2-u_1|^2\rho_1^2 e^{-\frac{2\epsilon_0}{T_1}}\mathcal{P}(T_1)
   \\ +\rho_1^2 e^{-\frac{2\epsilon_0}{T_1}}\int_{\mathbb{R}^3}|v|^2dv \int_\rth dv_4 \int_\stw d\omega\ W_+ (v,v_4;v_1,v_2)\mathcal{Z}(v,0,T_1)\mathcal{Z}(v_4,0,T_1)\\
    -\rho_1\rho_2\int_{\mathbb{R}^3}|v|^2dv \int_\rth dv_4 \int_\stw d\omega\ W_+ (v,v_4;v_1,v_2)\mathcal{Z}(v,0,T_2)\mathcal{Z}(v_4,u_1-u_2,T_1).
\end{multline*}
Therefore, we have
\begin{multline}
  J^{(2)}_e = \frac{2}{3}\bigg[  |u_2-u_1|^2\rho_1^2 e^{-\frac{2\epsilon_0}{T_1}}\mathcal{P}(T_1)
   \\ +\rho_1^2 e^{-\frac{2\epsilon_0}{T_1}}\int_{\mathbb{R}^3}|v|^2dv \int_\rth dv_4 \int_\stw d\omega\ W_+ (v,v_4;v_1,v_2)\mathcal{Z}(v,0,T_1)\mathcal{Z}(v_4,0,T_1)\\
    -\rho_1\rho_2\int_{\mathbb{R}^3}|v|^2dv \int_\rth dv_4 \int_\stw d\omega\ W_+ (v,v_4;v_1,v_2)\mathcal{Z}(v,0,T_2)\mathcal{Z}(v_4,u_1-u_2,T_1)\bigg]
    + \frac{4}{3}\epsilon_0 H^{(2)}.
\end{multline}
Then using \eqref{eq.relationJ1eJ2e}, we can also recover the representation for $J^{(1)}_e$.

Then in the stationary situation when $u_1=u_2=0,$ we have
\begin{multline}
  J^{(2)}_e = \frac{2}{3}\bigg[  \rho_1^2 e^{-\frac{2\epsilon_0}{T_1}}\int_{\mathbb{R}^3}|v|^2dv \int_\rth dv_4 \int_\stw d\omega\ W_+ (v,v_4;v_1,v_2)\mathcal{Z}(v,0,T_1)\mathcal{Z}(v_4,0,T_1)\\
    -\rho_1\rho_2\int_{\mathbb{R}^3}|v|^2dv \int_\rth dv_4 \int_\stw d\omega\ W_+ (v,v_4;v_1,v_2)\mathcal{Z}(v,0,T_2)\mathcal{Z}(v_4,0,T_1)\bigg]
    + \frac{4}{3}\epsilon_0 H^{(2)}\\
    = \frac{4}{3}\bigg[(\rho_1^2 e^{-\frac{2\epsilon_0}{T_1}}-\rho_2\rho_1)\mathcal{A}(T_1;\epsilon_0)+\rho_1\rho_2\mathcal{B}(T_1,T_2)\bigg],
\end{multline}
where we defined and used
\begin{equation}\label{eq.A}
    \mathcal{A}(T_1;\epsilon_0)\eqdef \int_{\mathbb{R}^3}\left(\frac{|v|^2}{2}+\epsilon_0\right)dv\int_\rth dv_4 \int_\stw d\omega\ W_+ (v,v_4;v_1,v_2)\mathcal{Z}(v,0,T_1)\mathcal{Z}(v_4,0,T_1)
\end{equation}and
\begin{equation}\label{eq.B}
    \mathcal{B}(T_1,T_2)\eqdef \int_{\mathbb{R}^3}\frac{|v|^2}{2}dv \int_\rth dv_4 \int_\stw d\omega\ W_+ (v,v_4;v_1,v_2)\mathcal{Z}(v_4,0,T_1) (\mathcal{Z}(v,0,T_1)-\mathcal{Z}(v,0,T_2)).
\end{equation}
We can further simplify $\mathcal{B}$ as
\begin{multline*}
    \mathcal{B}(T_1,T_2)= \int_{\mathbb{R}^3}dv\frac{|v|^2}{2}(\mathcal{Z}(v,0,T_1)-\mathcal{Z}(v,0,T_2)) \int_\rth dv_4 \int_\stw d\omega\ W_+ (v,v_4;v_1,v_2)\mathcal{Z}(v_4,0,T_1)\\
  =  \int_{\mathbb{R}^3}dv\frac{|v|^2}{2}(\mathcal{Z}(v,0,T_1)-\mathcal{Z}(v,0,T_2))Q,
\end{multline*}
where $$Q\eqdef \int_\rth dv_4 \int_\stw d\omega\ W_+ (0,v_4-v;v_1-v,v_2-v)\mathcal{Z}(v_4-v,-v,T_1).$$
This representation can be interpreted as follows. The transfer of energy due to the non-elastic collisions are affected by both the non-LTE difference in the ratio of densities and the difference of temperatures. These are related since there are different temperatures due to the absence of the LTE.

We now examine the radiation terms in the energy equations. We recall the definition of the radiation term $\mathcal{R}_p$ from \eqref{eq.Rp} and obtain that
$$
      J^{(j)}_r=\frac{4}{3}\int_\rth dv \left(\frac{|v-u_j|^2}{2}+\epsilon_0\delta_{j,2}\right)(-1)^{j+1}\int_{\mathbb{S}^2}dn\  [M^{(2)}(1+G)-M^{(1)}G].
$$Therefore, we have
\begin{multline*}
    J^{(1)}_r =\frac{4}{3}\int_\rth dv \frac{|v-u_1|^2}{2}\int_{\mathbb{S}^2}dn\  [M^{(2)}(1+G)-M^{(1)}G]\\
=    \frac{4}{3}\int_{\mathbb{S}^2}dn\ (1+G)\int_\rth dv \frac{|v-u_2+u_2-u_1|^2}{2}M^{(2)}+\frac{4}{3}\int_{\mathbb{S}^2}dn\ G\int_\rth dv \frac{|v-u_1|^2}{2}M^{(1)}\\
 = \rho_2 T_2 \int_{\mathbb{S}^2 } (1+G)dn+\frac{2}{3}|u_2-u_1|^2 \rho_2 \int_{\mathbb{S}^2 } (1+G)dn-\rho_1T_1 \int_{\mathbb{S}^2 } Gdn,
\end{multline*} and
\begin{multline*}\notag
    J^{(2)}_r 
    =\frac{4}{3}\int_\rth dv \frac{|v-u_2|^2}{2}\int_{\mathbb{S}^2}dn\  [M^{(1)}G-M^{(2)}(1+G)]\\
    +\frac{4\epsilon_0}{3}\int_\rth dv\int_{\mathbb{S}^2}dn\  [M^{(1)}G-M^{(2)}(1+G)]\\
    = \rho_1T_1 \int_{\mathbb{S}^2 } Gdn+\frac{2}{3}|u_1-u_2|^2 \rho_1 \int_{\mathbb{S}^2 } Gdn-\rho_2 T_2 \int_{\mathbb{S}^2 } (1+G)dn
    +\frac{4\epsilon_0}{3}Q^{(2)}.
\end{multline*}

\begin{remark}
The new functions $H^{(j)}$, $J^{(j)}_m,$ and $J^{(j)}_e$  which describe the transfer of mass, momentum and energy between the two types of molecules $A$ and $\bar{A}$, respectively, depend on the specific form of the collision kernel. In particular, their dependence cannot be obtained using only equilibrium properties of the system.
\end{remark}

\subsection{Stationary equations}
We recall the system \eqref{Euler.nonLTE2} with the further reduction of each operator introduced in Section \ref{sec.representations}. Then we obtain the following time-independent stationary system with $u_1=u_2=0$:
\begin{equation}\label{Euler.nonLTE.stationary}
    \begin{split}
        &\sigma H^{(1)}+Q^{(1)}=-\sigma H^{(2)}-Q^{(2)}=0,\\
          &\nabla_y\cdot S^{(1)}=-\nabla_y\cdot S^{(2)}=0,\\
        & \sigma J^{(1)}_e+J^{(1)}_r=0,\\
            &  \sigma J^{(2)}_e+J^{(2)}_r=0,
    \end{split}
\end{equation}where the operators are represented as$$
   H^{(2)}=-H^{(1)}
    =\rho_1^2 e^{-\frac{2\epsilon_0}{T_1}}\mathcal{P}(T_1)
    -\rho_1\rho_2\mathcal{P}(T_2,T_1),$$
$$
    Q^{(1)}=-Q^{(2)}=\rho_2\int_{\mathbb{S}^2}dn
\ (1+G)-\rho_1\int_{\mathbb{S}^2}dn
\ G,
$$ 
$$ S^{(j)}=p^{(j)} I=\frac{1}{2}\rho_jT_jI,$$
$$J^{(1)}_e=-J^{(2)}_e=- \frac{4}{3}\bigg[(\rho_1^2 e^{-\frac{2\epsilon_0}{T_1}}-\rho_2\rho_1)\mathcal{A}(T_1;\epsilon_0)+\rho_1\rho_2\mathcal{B}(T_1,T_2)\bigg],$$
$$J^{(1)}_r= -\rho_1T_1 \int_{\mathbb{S}^2 } Gdn+\rho_2 T_2 \int_{\mathbb{S}^2 } (1+G)dn,$$
$$J^{(2)}_r= \rho_1T_1 \int_{\mathbb{S}^2 } Gdn-\rho_2 T_2 \int_{\mathbb{S}^2 } (1+G)dn
    +\frac{4\epsilon_0}{3}Q^{(2)},$$
 where $\mathcal{P}$, $\mathcal{A}$, and $\mathcal{B}$ are defined in \eqref{eq.P},  \eqref{eq.A}, and \eqref{eq.B}, respectively, and we recall $\mathcal{P}(T_i,T_j)\eqdef \mathcal{P}(T_i,T_j,0,0)$. We recall that the functions $\mathcal{P}$, $\mathcal{A}$, and $\mathcal{B}$ depend on the form of the collision kernel. 
We also recall \eqref{eq.rescaled kinetic equations}$_2$ and \eqref{eq.Rr} and obtain the equation for the radiation
\begin{equation}\label{eq.radiation.nonLTE.stationary}
    n\cdot\nabla_y G =\epsilon_0\int_{\rth}[M^{(2)}(1+G)-M^{(1)} G]dv=\epsilon_0(\rho_2 (1+G)-\rho_1 G).
\end{equation}
Finally, we have the total mass conservation that 
\begin{equation}\label{masscon}
\int_{\Omega_y} dy\  (\rho_1+\rho_2)=m_0,    
\end{equation}for some given constant $m_0>0,$ where $\Omega_y$ is the rescaled domain of $\Omega$ in the variable $y$.

 \subsection{Non-existence for generic boundary values}
In this subsection, we prove the non-existence of a solution to the system \eqref{Euler.nonLTE.stationary}, \eqref{eq.radiation.nonLTE.stationary}, and \eqref{masscon} with the boundary condition \eqref{eq.bc for 3d new}. Before we state our theorem, we first define several auxiliary quantities $H=H(T_1,T_2)$, $L=L(T_1,T_2)$ and $S=S(T_1,T_2)$ as follows:
\begin{equation}\label{Hdef}
    H(T_1,T_2)\eqdef \frac{T_2}{T_1}e^{-\frac{2\epsilon_0}{T_1}}\frac{\mathcal{P}(T_1)}{\mathcal{P}(T_2,T_1)},
 \end{equation}
 \begin{equation}
     \label{Sdef} S(T_1,T_2)\eqdef \frac{\frac{4\pi\frac{H(T_1,T_2)}{T_1T_2}(T_2-T_1)}{\frac{1}{T_1}-\frac{H(T_1,T_2)}{T_2}}}{ \frac{4\sigma}{3}\frac{1}{T_1}\bigg[(\frac{e^{-\frac{2\epsilon_0}{T_1}}}{T_1} -\frac{H(T_1,T_2)}{T_2})\mathcal{A}(T_1;\epsilon_0)+\frac{H(T_1,T_2)}{T_2}\mathcal{B}(T_1,T_2)\bigg]},
 \end{equation}
\begin{equation}
    \label{Ldef}
    L(T_1,T_2)\eqdef S(T_1,T_2)\left(\frac{1}{T_1}+\frac{H(T_1,T_2)}{T_2}\right),
\end{equation}where $\mathcal{P}$, $\mathcal{A}$, and $\mathcal{B}$ are defined in \eqref{eq.P}, \eqref{eq.A}, and \eqref{eq.B}.

Now we state the non-existence theorem:
\begin{theorem}\label{thm.nonexist}Let $m_0>0$ given and let $\Omega$ be a bounded convex domain with $\partial \Omega \in C^1.$  Assume that $L(T_1,T_2)=\frac{m_0}{|\Omega|}$ defines a smooth curve in the plane $(T_1,T_2)\in \mathbb{R}^2_+$ for given $m_0$ and $\Omega,$ where $L$ is defined as in \eqref{Ldef}. For each boundary point $y_0\in \partial \Omega$, define the outward unit normal vector $\nu_{y_0}$. Then define the incoming boundary condition for $G(y_0,n)$ for incoming velocities $n\in \mathbb{S}^2$ with $n\cdot \nu_{y_0}<0$ as 
\begin{equation}\label{eq.bc for 3d new}
    G(y_0,n)= f(n),
\end{equation}for a given boundary profile $f.$
Then the system of the Euler-like system coupled with radiation \eqref{Euler.nonLTE.stationary}, \eqref{eq.radiation.nonLTE.stationary}, and  \eqref{masscon} in the non-LTE case with the boundary condition \eqref{eq.bc for 3d new} does not have a solution unless the given boundary profile $f$ is chosen specifically so that it satisfies
\begin{equation}\label{nonex condition}div\left(\int_{\mathbb{S}^2}nf(n)e^{-A_2s(y,n)}dn\right)=0,\\
\text{ for any }y\in \Omega\text{ and for some }A_2>0,\end{equation}
where for each $y\in \Omega$ and $n\in\mathbb{S}^2$,  $y_0=y_0(y,n)\in \partial \Omega$ and $s=s(y,n)$ are determined uniquely such that
    $$y=y_0(y,n)+s(y,n)n.$$
\end{theorem} 
\begin{remark}
The condition \eqref{nonex condition} for the function $f(n)$ is very restrictive. Therefore there are no solutions for generic boundary conditions with $u=0.$ It is still possible that the system could have stationary solutions with $u\ne 0,$ but we will not consider these solutions in this paper.
\end{remark}
\begin{proof}We suppose on the contrary that a solution $(\rho_1,\rho_2,T_1,T_2,G)$ for the system \eqref{Euler.nonLTE.stationary}, \eqref{eq.radiation.nonLTE.stationary}, and \eqref{masscon} exists. We will then consider the equations that the solution must satisfy. 
We start with adding \eqref{Euler.nonLTE.stationary}$_3$ and \eqref{Euler.nonLTE.stationary}$_4$ and obtaining that
\begin{equation}\label{Q2zero}\frac{4\epsilon_0}{3}Q^{(2)}=0.\end{equation} Thus we have
\begin{equation}\label{eq.Gint}\int_{\mathbb{S}^2}G\ dn=\frac{4\pi\rho_2}{\rho_1-\rho_2}.\end{equation}
Then using \eqref{Q2zero} and \eqref{eq.radiation.nonLTE.stationary}, we have
\begin{equation}\label{integral of Gflow zero}
    \nabla_y \cdot \int_{\mathbb{S}^2}nGdn=0.
\end{equation}
 Since $\sigma H^{(2)}=0$, we have
 $$\rho_1e^{-\frac{2\epsilon_0}{T_1}}\mathcal{P}(T_1)=\rho_2\mathcal{P}(T_2,T_1).$$
 By \eqref{Euler.nonLTE.stationary}$_2$, we also have
 $\nabla_y p^{(1)}=0$ and $\nabla_y(p^{(1)}+p^{(2)})=0.$ Finally, \eqref{Euler.nonLTE.stationary}$_3$ implies that
 $$\frac{4\sigma}{3}\bigg[(\rho_1^2 e^{-\frac{2\epsilon_0}{T_1}}-\rho_2\rho_1)\mathcal{A}(T_1;\epsilon_0)+\rho_1\rho_2\mathcal{B}(T_1,T_2)\bigg]=\frac{4\pi\rho_1\rho_2(T_2-T_1)}{\rho_1-\rho_2}.$$
Therefore we have reduced the system and obtained the following system below:
 \begin{equation}\label{reduced nonLTE stationary}
     \begin{split}
     &\rho_1e^{-\frac{2\epsilon_0}{T_1}}\mathcal{P}(T_1)=\rho_2\mathcal{P}(T_2,T_1),\\
         &\nabla_y p^{(1)}=\nabla_y p^{(2)}=0,\\ &\frac{4\sigma}{3}\bigg[(\rho_1^2 e^{-\frac{2\epsilon_0}{T_1}}-\rho_2\rho_1)\mathcal{A}(T_1;\epsilon_0)+\rho_1\rho_2\mathcal{B}(T_1,T_2)\bigg]=\frac{4\pi\rho_1\rho_2(T_2-T_1)}{\rho_1-\rho_2},\\
         &n\cdot\nabla_y G =\epsilon_0\int_{\rth}[M^{(2)}(1+G)-M^{(1)} G]dv=\epsilon_0(\rho_2 (1+G)-\rho_1 G).
     \end{split}
 \end{equation}
 Here the functions $\mathcal{P},\mathcal{A},\mathcal{B}$ depend on the non-elastic collision kernel, and the absence of detailed balance plays a role. Moreover, we have $\mathcal{B}(T_1,T_1;\epsilon_0)=0$ and this implies that the term $\mathcal{B}$ is related to the transfer of energy due to the difference of temperatures.

 Then by \eqref{reduced nonLTE stationary}$_2$, we have
 \begin{equation}\label{rhoinT}\rho_1=\frac{C_1}{T_1}\text { and }\rho_2=\frac{C_2}{T_2},\end{equation} for some constant $C_1,C_2>0.$ Then by plugging this into \eqref{reduced nonLTE stationary}$_3$, we have
 \begin{equation}
     \label{prephi}
     \frac{4\sigma}{3}\frac{C_1}{T_1}\bigg[(\frac{C_1}{T_1} e^{-\frac{2\epsilon_0}{T_1}}-\frac{C_2}{T_2})\mathcal{A}(T_1;\epsilon_0)+\frac{C_2}{T_2}\mathcal{B}(T_1,T_2)\bigg]=\frac{4\pi\frac{C_1C_2}{T_1T_2}(T_2-T_1)}{\frac{C_1}{T_1}-\frac{C_2}{T_2}}.
 \end{equation}
 Also, by \eqref{masscon}, we have 
 \begin{equation}
     \label{masscon2}
     \int_\Omega dy\left(\frac{C_1}{T_1}+\frac{C_2}{T_2}\right)=m_0.
 \end{equation}
 In addition, by \eqref{reduced nonLTE stationary}$_1$, we have
 \begin{equation}
     \label{third eq}
   \mathcal{C}\eqdef \frac{C_2}{C_1}=  \frac{T_2}{T_1}e^{-\frac{2\epsilon_0}{T_1}}\frac{\mathcal{P}(T_1)}{\mathcal{P}(T_2,T_1)}=H(T_1,T_2).
 \end{equation}
 Then by plugging \eqref{third eq} into \eqref{prephi} we have
 \begin{equation}
     \label{C1eq}
     C_1= \frac{\frac{4\pi\frac{\mathcal{C}}{T_1T_2}(T_2-T_1)}{\frac{1}{T_1}-\frac{\mathcal{C}}{T_2}}}{ \frac{4\sigma}{3}\frac{1}{T_1}\bigg[(\frac{e^{-\frac{2\epsilon_0}{T_1}}}{T_1} -\frac{\mathcal{C}}{T_2})\mathcal{A}(T_1;\epsilon_0)+\frac{\mathcal{C}}{T_2}\mathcal{B}(T_1,T_2)\bigg]}=S(T_1,T_2).
 \end{equation} Then for each $C_1$ and $C_2$, $S$ and $H$ determine a discrete set of values $T_1$ and $T_2$, and they are independent of $y$. Then using \eqref{masscon2}, we consider
 \begin{equation}\notag L(T_1,T_2)=S(T_1,T_2)\left(\frac{1}{T_1}+\frac{H(T_1,T_2)}{T_2}\right)=\frac{m_0}{|\Omega|},\end{equation} for given $m_0$ and $\Omega$. Since $L=\frac{m_0}{|\Omega|}$ defines a smooth curve by the assumption, we can parametrize the values as
 \begin{equation}\label{TfunctionofC1}T_1=T_1(\tau),\ T_2=T_2(\tau),\ \tau\in I_L\subset \mathbb{R},\end{equation} which also implies 
  $$C_1=C_1(\tau),\ C_2=C_2(\tau),\ \tau\in I_L\subset \mathbb{R}.$$
 Also, $\rho_1=\frac{C_1}{T_1}(\tau)$ and $\rho_2=\frac{C_2}{T_2}(\tau)$ must live on $\rho_1>\rho_2$, since \eqref{eq.Gint} and $G\ge 0$ imply $\rho_1>\rho_2$.
 
 Now we consider the equation \eqref{reduced nonLTE stationary}$_4$. By \eqref{rhoinT}, \eqref{third eq}, and  \eqref{TfunctionofC1}, we have
 $$n\cdot\nabla_y G =\left(A_1-A_2 G\right),$$ where $A_1$ and $A_2$ are constant that depends only on $\tau$ and $\epsilon_0$ as  $A_1=A_1(\tau;\epsilon_0)>0$ and $ A_2=A_2(\tau;\epsilon_0)>0.$ Then we solve this with the boundary condition \eqref{eq.bc for 3d new}. If we are given $y\in \Omega$ and $n\in\mathbb{S}^2$, there exist unique $y_0=y_0(y,n)\in \partial \Omega$ and $s=s(y,n)$ such that
    $$y=y_0(y,n)+s(y,n)n,$$ and 
    $$G(y,n)=f(n)e^{-A_2s(y,n)}+A_1\int_0^{s(y,n)} e^{-A_2(s(y,n)-\xi)}d\xi.$$
    Define the constant flux $J$ as 
    $$J\eqdef \int_{\mathbb{S}^2}dn\ nG.$$ Recall that $J$ is constant due to \eqref{integral of Gflow zero}. 
    Then we further have
    \begin{equation*}
    J= \vec{R}(y)+A_1\int_{\mathbb{S}^2}n\left(\int_0^{s(y,n)} e^{-A_2(s(y,n)-\xi)}d\xi\right)dn,
    \end{equation*}where 
    $$\vec{R}(y)\eqdef \int_{\mathbb{S}^2}nf(n)e^{-A_2s(y,n)}dn.$$
    Thus, we have
    \begin{equation}\label{eq.div.JR2}
    0= div(\vec{R})+div\int_{\mathbb{S}^2}nA_1\int_0^{s(y,n)} e^{-A_2(s(y,n)-\xi)}dn.
    \end{equation}
    Here we further observe that
  $$
        \int_{\mathbb{S}^2}nA_1\left(\int_0^{s(y,n)} e^{-A_2(s(y,n)-\xi)}d\xi\right)dn
        =\int_{\partial\Omega}dz\int_{\mathbb{S}^2}n\left(\int_0^{s(y,n)} e^{-A_2(s(y,n)-\xi)}A_1\delta(z-y_0(y,n))d\xi\right)dn.$$
    We make a change of variables $\xi\mapsto \hat{\xi}=s(y,n)-\xi.$ Then we have$$
        \int_{\mathbb{S}^2}nA_1\left(\int_0^{s(y,n)} e^{-A_2(s(y,n)-\xi)}d\xi\right)dn
        =\int_{\partial\Omega}dz\int_{\mathbb{S}^2}nA_1\left(\int_0^{s(y,n)} e^{-A_2\hat{\xi}}\delta(z-y_0(y,n))d\hat{\xi}\right)dn.$$
    Then we make another change of variables $(
    \hat{\xi},n)\mapsto \eta\eqdef y-\hat{\xi} n \in \Omega$. Since $y$ is independent of $\hat{\xi}$ and $n$, we obtain the Jacobian of the change of variables $J(\eta,n)\eqdef \left|\frac{\partial(\hat{\xi},n)}{\partial \eta}\right|=\frac{1}{|y-\eta|^2}$. Thus, for $n=n(y-\eta)=\frac{y-\eta}{|y-\eta|}$, we have
$$
     \int_{\partial\Omega}dz\int_{\mathbb{S}^2}nA_1\left(\int_0^{s(y,n)} e^{-A_2\hat{\xi}}\delta(z-y_0(y,n))d\hat{\xi}\right)dn
       =\int_{\rth}\frac{1}{|y-\eta|^2}\frac{y-\eta}{|y-\eta|} e^{-A_2|y-\eta|}A_1d\eta,$$since $\hat{\xi}=(y-\eta)\cdot n= |y-\eta|.$
    Then, going back to \eqref{eq.div.JR2}, we have
  $$ 0= div(\vec{R})+\int_{\rth}div\left(\frac{1}{|y-\eta|^2}\frac{y-\eta}{|y-\eta|} e^{-A_2|y-\eta|}\right)A_1d\eta.$$
  Note that 
  \begin{equation*}
      div\left(\frac{1}{|y-\eta|^2}\frac{y-\eta}{|y-\eta|} e^{-A_2|y-\eta|}\right)=div \left(\frac{e^{-A_2r}}{r^2}\hat{r}\right)=-A_2\frac{e^{-A_2r}}{r^2}+4\pi\delta(r).
  \end{equation*}  Therefore, we finally have
\begin{equation}\label{final eq nonexist}\frac{1}{4\pi}div\int_{\mathbb{S}^2}nf(n)e^{-A_2s(y,n)}dn=A_1A_2\int_{\rth}\frac{e^{-A_2|y-\eta|}}{4\pi|y-\eta|^2} d\eta- A_1=0.\end{equation}
 Note that $A_1>0$ and $A_2>0$ are constant and \eqref{final eq nonexist} must hold for any $y\in \Omega$ and $A_2(\tau)$ for $\tau \in I_L\subset\mathbb{R}$. In general, \eqref{final eq nonexist} does not hold for a general incoming boundary profile $f$. For instance, in a 1-dimensional model, we have $$div\int_{\mathbb{S}^2}nf(n)e^{-A_2s(y,n)}dn=(\text{constant independent of }y).$$  
 Therefore, we conclude that there is no stationary solution unless \eqref{final eq nonexist} holds. 
 \end{proof}
 
 \subsection{Smooth level curves $L$}
 In this section, we will check numerically if the assumption of Theorem \ref{thm.nonexist} that $L$ in \eqref{Ldef} defines a smooth curve holds (at least for some particular choices of kernels). We will assume the hard-sphere kernel assumption in the rest of the section.  In this case, we will assume for the sake of simplicity that 
 $$B_{non.el}(|v_3-v_4|,\omega \cdot (v_3-v_4))=2|v_3-v_4|,$$ such that 
 $W_+$ in \eqref{W+def}, \eqref{eq.P}, \eqref{eq.A}, and \eqref{eq.B} is given by
 $$W_+(v_3,v_4;v_1,v_2)=\sqrt{|v_3-v_4|^2+4\epsilon_0}.$$ Then the kernel does not depend on the angular variable $\omega$ and we have $\int_{\mathbb{S}^2} W_+d\omega = 4\pi W_+,$ and we have
 \begin{equation}\label{new.P}\mathcal{P}(T_k,T_l)= 4\pi \int_{\mathbb{R}^3}dv_3 \int_\rth dv_4  \sqrt{|v_3-v_4|^2+4\epsilon_0}\mathcal{Z}(v_3,0,T_k)\mathcal{Z}(v_4,0,T_l).\end{equation}
Also, by \eqref{Hdef}, \eqref{Sdef}, and \eqref{eq.B}, we have\begin{equation}
   \notag S(T_1,T_2)\eqdef \frac{\frac{4\pi H(T_1,T_2)(T_2-T_1)}{T_2-T_1H(T_1,T_2)}}{ \frac{4\sigma}{3}\frac{1}{T_1}\bigg[( 1-\frac{\mathcal{P}(T_1)}{\mathcal{P}(T_2,T_1)})\frac{e^{-\frac{2\epsilon_0}{T_1}}}{T_1}\mathcal{A}(T_1;\epsilon_0)+\frac{H(T_1,T_2)}{T_2}\mathcal{B}(T_1,T_2)\bigg]}.
 \end{equation}
 
\subsubsection{Level curves of $L$ via numerical simulations}\label{sec.numerics}
 In order to proceed to the numerical simulation in Section \ref{sec.numerics}, it is convenient to remove the removable singularity of $T_1-T_2=0$ in both numerator and denominator of $S(T_1,T_2)$ above. More precisely, we will consider the following representation
 \begin{equation}
   \notag S(T_1,T_2)\eqdef \frac{\frac{4\pi H(T_1,T_2)}{T_2-T_1H(T_1,T_2)}}{ \frac{4\sigma}{3}\frac{1}{T_1}\bigg[ \frac{\mathcal{P}(T_2,T_1)-\mathcal{P}(T_1)}{T_2-T_1}\frac{e^{-\frac{2\epsilon_0}{T_1}}}{T_1\mathcal{P}(T_2,T_1)}\mathcal{A}(T_1;\epsilon_0)+\frac{H(T_1,T_2)}{T_2}\frac{\mathcal{B}(T_1,T_2)}{T_2-T_1}\bigg]},
 \end{equation} and reduce the representations of $\frac{\mathcal{P}(T_2,T_1)-\mathcal{P}(T_1)}{T_2-T_1}$ and $\frac{\mathcal{B}(T_1,T_2)}{T_2-T_1}$. The reduction of the integral operators $\mathcal{P}$, $\mathcal{A}$, $\mathcal{B}$, $\frac{\mathcal{P}(T_2,T_1)-\mathcal{P}(T_1)}{T_2-T_1}$, and $\frac{\mathcal{B}(T_1,T_2)}{T_2-T_1}$ will be explicitly given in Appendix \ref{sec.appendix.reduction}. 

For the numerical simulations, we consider the domain $[10,12]\times [10,12]$ for $(T_1,T_2)$. Then Figure \ref{levelcurve2} below provides the level contour curves for $L(T_1,T_2)$ in the range $(T_1,T_2)\in [10,12]^2$ with $\epsilon_0=c_0=\sigma=1.$ In Figure \ref{levelcurve2}, we choose the values of $T_1$ and $T_2$ from 10 to 12 varying in steps of 0.1 and plot the level curves of $L$ using the values of  21 by 21 matrix. \begin{figure}[ht]
   
    \begin{subfigure}[b]{0.45\textwidth}
         \centering
         \includegraphics[width=\textwidth]{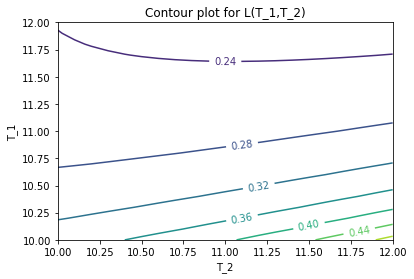}
         \caption{Contour}
     \end{subfigure}
 \hfill
     \begin{subfigure}[b]{0.45\textwidth}
         \centering
         \includegraphics[width=\textwidth]{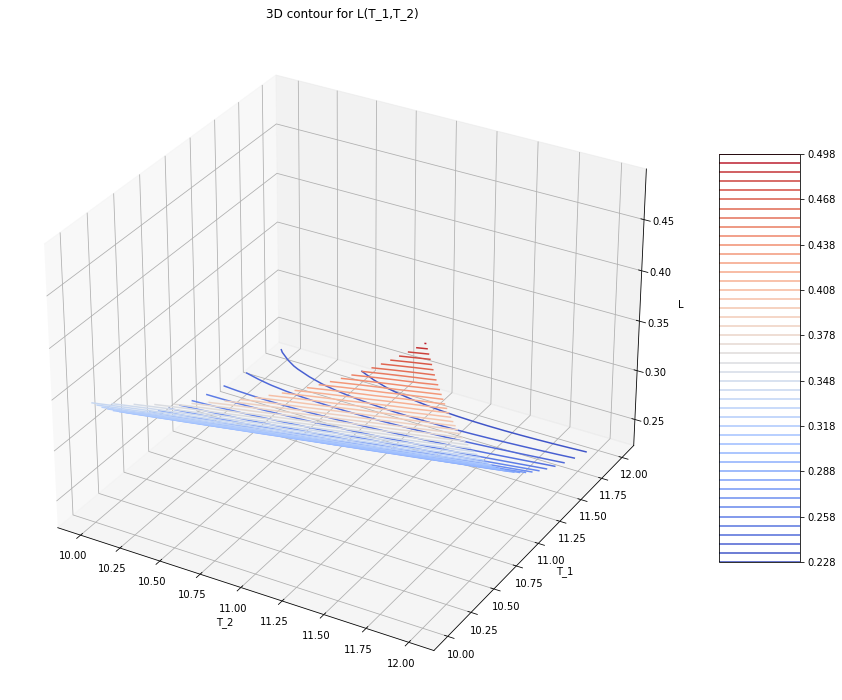}
         \caption{3D Contour}
     \end{subfigure}
    \caption{Contour level curves for $L(T_1,T_2)$ for $(T_1,T_2)\in [10,12]^2$}
    \label{levelcurve2}
\end{figure}
 Figure \ref{levelcurve2} shows that $L(T_1,T_2)$ defines a smooth curve in the plane $(T_1,T_2)\in [10,12]\times [10,12]$.
 
This completes the discussion of the non-existence of a stationary solution in the two-molecule system of Euler-radiation equations with two different temperatures. In the next section, we will introduce a system of three-level system in which we have non-LTE without assuming any artificial assumptions on the scale differences in the elastic collisions.

 \section{Non-LTE via a model for three-level molecules}\label{sec.nonLTE.3level}
 In this section, we will study the system of Euler-radiation equations for three-level molecules. The goal is to obtain situations in which the stationary solutions yield non-LTE with $u=0$. Previously, in Section \ref{sec.nonLTEnonexist}, we have obtained a non-LTE regime in which the elastic collisions between molecules $A$ and $\bar{A}$ take place much less often than the collisions between $A$ and $A$ or $\bar{A}$ and $\bar{A}.$ We will assume now that the rate of elastic collisions $A+\bar{A}\leftrightarrows A+\bar{A}$ is of the same order as that of the collisions $A+A\leftrightarrows A+A$  and $\bar{A}+\bar{A}\leftrightarrows \bar{A}+\bar{A}$. We have seen in Section \ref{sec.LTEexist} that in the case of a system with two levels $A$ and $\bar{A}$, the stationary solutions with $u=0$ yield LTE.  This implies that in stationary regime the system is necessarily in LTE. In order to obtain non-LTE steady states, we need a system with at least three levels.
 
 Therefore, in this section, we will consider a system with three states: the ground-state $A_1$, the first-level excited-state $A_2$, and the second-level excited-state $A_3,$ where the differences in the energy levels are given by $\epsilon_1$ and $\epsilon_2$.  In the rest of this section, we assume that the energy levels differ by the same amount $\epsilon_1=\epsilon_2=\epsilon$ such that the ground-state $A_1$ and the second-level excited-state $A_3$ have the energy difference of $2\epsilon.$
 The relevant difference is that in these limits we do not have in stationary configurations with zero macroscopic velocities that the transfer of molecules induced by the radiation are necessarily zero. Indeed, the energy equation gives just $Q^{(1)}\epsilon_1+Q^{(2)}\epsilon_2=0$, instead of $Q^{(1)}=-Q^{(2)}=0$. This allows in systems exchanging radiation with their surroundings, non-LTE stationary solutions (with $u=0$). 
Our assumptions on the elastic collisions imply that the temperatures of all the molecules are the same $T_1=T_2=T_3.$  We will assume that all the elastic collisions between molecules at the states $A_1$, $A_2,$ and $A_3$ take place at a comparable rate. 
 
 In this case, we consider the following interactions:
 \begin{enumerate}\item Elastic collisions between molecules: \begin{equation}\label{eq.elastic reactions.3p}
	A_i+A_j \rightleftarrows A_i+A_j,\ \text{ where }1\le i,j\le 3.
\end{equation}
These are collisions between molecules in which the total kinetic energy and the total momentum are conserved. 
\item Nonelastic collisions:
\begin{equation}\label{eq.nonelastic reactions.3p}
	\begin{split}
	    &A_1+A_1\rightleftarrows A_1+A_2,\\
	   &A_2+A_2\rightleftarrows A_2+A_3.
	\end{split}
\end{equation}These collisions are the collisions between two lower-state molecules in which one lower-state molecule is being excited. The reverse reaction can also take place.  Here we neglect the nonelastic collisions between $A_1$ and $A_3$. 
\item Collisions between a molecule and a photon:
\begin{equation}\label{eq.reactions2.3p}
	A_i+\phi\rightleftarrows A_{i+1},\text{ for }i=1,2.
\end{equation}Here a molecule in the lower-state absorbs a photon via this reaction and becomes a molecule in the upper excited state.  Since the energy differences between $A_1$ and $A_2$ and between $A_2$ and $A_3$ are the same, we expect that the same type of photon is being absorbed or is being emitted via the collision. 
The reverse reaction can also take place.
\end{enumerate}
 We could also include non-elastic collisions with the form,
$$A_{i+1}+A_i\rightleftarrows A_{i+1}+A_{i+1},\text{ for }i=1,2,$$
which, however, will be ignored in this section.

\subsection{Kinetic equations}
Denote $F^{(i)}$ as the distribution function of the molecule $A_i$ for $i=1,2,3$. Then as in Section \ref{sec.rescaled.kinetic} each distribution function $F^{(i)}$ in the three-level system satisfies the following kinetic equations 
\begin{equation}\label{eq.rescaled kinetic equations.3p}
    \begin{split}
      &  \frac{\partial}{\partial t}F^{(1)}+ v\cdot \nabla_y F^{(1)}=\frac{1}{\alpha}\left( \mathcal{K}^{(1,1)}_{el}[F^{(1)},F^{(1)}]+\mathcal{K}^{(1,2)}_{el}[F^{(1)},F^{(2)}]+\mathcal{K}^{(1,3)}_{el}[F^{(1)},F^{(3)}]\right)\\&\qquad\qquad+\sigma \left(\mathcal{K}^{(1,2)}_{non.el}[F^{(1)},F^{(2)}]\right)+\gamma_1 \int_{\mathbb{S}^2}(F^{(2)}(1+G)-F^{(1)}G)dn,\\
        &\frac{\partial}{\partial t}F^{(2)}+ v\cdot \nabla_y F^{(2)}=\frac{1}{\alpha}\left( \mathcal{K}^{(2,1)}_{el}[F^{(2)},F^{(1)}]+\mathcal{K}^{(2,2)}_{el}[F^{(2)},F^{(2)}]+\mathcal{K}^{(2,3)}_{el}[F^{(2)},F^{(3)}]\right)\\&\qquad\qquad+\sigma \left(\mathcal{K}^{(2,1)}_{non.el}[F^{(2)},F^{(1)}]+\mathcal{K}^{(2,3)}_{non.el}[F^{(2)},F^{(3)}]\right)\\&\qquad\qquad-\gamma_1 \int_{\mathbb{S}^2}(F^{(2)}(1+G)-F^{(1)}G)dn+\gamma_2 \int_{\mathbb{S}^2}(F^{(3)}(1+G)-F^{(2)}G)dn,\\
        &\frac{\partial}{\partial t}F^{(3)}+ v\cdot \nabla_y F^{(3)}=\frac{1}{\alpha}\left( \mathcal{K}^{(3,1)}_{el}[F^{(3)},F^{(1)}]+\mathcal{K}^{(3,2)}_{el}[F^{(3)},F^{(2)}]+\mathcal{K}^{(3,3)}_{el}[F^{(3)},F^{(3)}]\right)\\&\qquad\qquad+\sigma \left(\mathcal{K}^{(3,2)}_{non.el}[F^{(3)},F^{(2)}]\right)-\gamma_2 \int_{\mathbb{S}^2}(F^{(3)}(1+G)-F^{(2)}G)dn,\\
        &n\cdot \nabla_y G=\gamma_1\epsilon \int_{\mathbb{R}^3}(F^{(2)}(1+G)-F^{(1)}G)dv+\gamma_2\epsilon \int_{\mathbb{R}^3}(F^{(3)}(1+G)-F^{(2)}G)dv,
    \end{split}
\end{equation}
where $F=\begin{pmatrix}F^{(1)}\\F^{(2)}\\F^{(3)} \end{pmatrix}$ and we rescale the variables and have $\gamma_1+\gamma_2=1$ and $\sigma=\eta/\alpha$. Here $\mathcal{K}^{(i,j)}_{el}$ is defined as in \eqref{eq.kel} for $i,j=1,2,3$. Analogously to \eqref{eq.knonel1}-\eqref{eq.knonel2}, the non-elastic operators $\mathcal{K}^{(i,j)}_{non.el}$ in the three-level system are defined as follows:
\begin{multline*}
    \mathcal{K}^{(1,2)}_{non.el}[F,F](\bar{v})\eqdef 2\int_\rth d\bar{v}_2 \int_\stw d\omega \frac{\sqrt{|\bar{v}-\bar{v}_2|^2-4\epsilon_0}}{2|\bar{v}-\bar{v}_2|} B_{non.el}(|\bar{v}-\bar{v}_2|,\omega\cdot (\bar{v}-\bar{v}_2)) (\bar{F}_3^{(2)}\bar{F}_4^{(1)}-\bar{F}^{(1)}(\bar{v}) \bar{F}_2^{(1)})\\
    +\int_\rth d\bar{v}_3 \int_\stw d\omega \frac{\sqrt{|\bar{v}_3-\bar{v}|^2+4\epsilon_0}}{2|\bar{v}_3-\bar{v}|} B_{non.el}(|\bar{v}_3-\bar{v}|,\omega\cdot (\bar{v}_3-\bar{v}))  (\bar{F}^{(1)}_1\bar{F}^{(1)}_2-\bar{F}^{(2)}_3\bar{F}^{(1)}(\bar{v})),
\end{multline*}

\begin{multline*}
    \mathcal{K}^{(2,1)}_{non.el}[F,F](\bar{v})\eqdef \int_\rth d\bar{v}_4 \int_\stw d\omega \frac{\sqrt{|\bar{v}-\bar{v}_4|^2+4\epsilon_0}}{2|\bar{v}-\bar{v}_4|} B_{non.el}(|\bar{v}-\bar{v}_4|,\omega\cdot (\bar{v}-\bar{v}_4))   (\bar{F}^{(1)}_1\bar{F}^{(1)}_2-\bar{F}^{(2)}(\bar{v})\bar{F}^{(1)}_4),
\end{multline*}

\begin{multline*}
    \mathcal{K}^{(2,3)}_{non.el}[F,F](\bar{v})\eqdef 2\int_\rth d\bar{v}_2 \int_\stw d\omega \frac{\sqrt{|\bar{v}-\bar{v}_2|^2-4\epsilon_0}}{2|\bar{v}-\bar{v}_2|} B_{non.el}(|\bar{v}-\bar{v}_2|,\omega\cdot (\bar{v}-\bar{v}_2)) (\bar{F}_3^{(3)}\bar{F}_4^{(2)}-\bar{F}^{(2)}(\bar{v}) \bar{F}_2^{(2)})\\
    +\int_\rth d\bar{v}_3 \int_\stw d\omega \frac{\sqrt{|\bar{v}_3-\bar{v}|^2+4\epsilon_0}}{2|\bar{v}_3-\bar{v}|} B_{non.el}(|\bar{v}_3-\bar{v}|,\omega\cdot (\bar{v}_3-\bar{v}))  (\bar{F}^{(2)}_1\bar{F}^{(2)}_2-\bar{F}^{(3)}_3\bar{F}^{(2)}(\bar{v})),
\end{multline*}
\begin{multline*}
    \mathcal{K}^{(3,2)}_{non.el}[F,F](\bar{v})\eqdef \int_\rth d\bar{v}_4 \int_\stw d\omega \frac{\sqrt{|\bar{v}-\bar{v}_4|^2+4\epsilon_0}}{2|\bar{v}-\bar{v}_4|} B_{non.el}(|\bar{v}-\bar{v}_4|,\omega\cdot (\bar{v}-\bar{v}_4))   (\bar{F}^{(2)}_1\bar{F}^{(2)}_2-\bar{F}^{(3)}(\bar{v})\bar{F}^{(2)}_4).
\end{multline*}
In principle, the collision kernels $B_{non.el}$ could be different for each type of collision, but we will take them equal for the sake of simplicity.

As \eqref{nonLTE equilibrium}, the local Maxwellian equilibria $M^{(i)}$ for the molecules $A_i$ for $i=1,2,3$ are defined as
\begin{equation}
    \label{Maxwellian.3p}
   M^{(i)}= M^{(i)}(t,x,v;T,u)\eqdef \frac{c_0\rho_i}{T^{\frac{3}{2}}}\exp\left(-\frac{|v-u|^2}{T}\right),\text{ for }i=1,2,3,
\end{equation}
where the molecule density $\rho_i$ is defined as 
$$\rho_i(t,x)\eqdef \int_{\mathbb{R}^3} F^{(i)}(t,x,v)dv.
$$ 
 \subsection{Hydrodynamic equations via conservation laws}Then as in Section \ref{sec.chapman.nonlte} we use the same Chapman-Enskog expansion \eqref{eq.chapman.euler.nonlte} in \eqref{eq.rescaled kinetic equations.3p} and obtain the leading order equations in $\alpha.$ Then we take the inner product \eqref{new L2 inner product} on the leading order equations in $\alpha$ against 1, and also take the inner product \eqref{eq.L2 inner product} on the leading order equation against  $\begin{pmatrix}v-u\\v-u\\v-u\end{pmatrix}$, and $\frac{4}{3}\begin{pmatrix}\frac{|v-u |^2}{2}\\\frac{|v-u |^2}{2}+\epsilon\\\frac{|v-u |^2}{2}+2\epsilon\end{pmatrix}$. Then finally we follow the same principles for the reduction via the conservation laws as in the beginning of Section \ref{sec.representations} and obtain the following set of hydrodynamic equations analogous to \eqref{Euler.nonLTE}:
 \begin{equation}\label{eq.nonlte.3p1}
     \partial_t \rho_1 +\nabla_y \cdot (\rho_1u)=\sigma H_{1,2}(\rho_1,\rho_2,T)+Q_{1,2}(\rho_1,\rho_2,T;G),
 \end{equation}
  \begin{equation}
     \partial_t \rho_2 +\nabla_y \cdot (\rho_2u)=\sigma H_{2,1}(\rho_2,\rho_1,T)+\sigma H_{2,3}(\rho_2,\rho_3,T)+Q_{2,1}(\rho_2,\rho_1,T;G)+Q_{2,3}(\rho_2,\rho_3,T;G),
 \end{equation}
  \begin{equation}
     \partial_t \rho_3 +\nabla_y \cdot (\rho_3u)=\sigma H_{3,2}(\rho_3,\rho_2,T)+Q_{3,2}(\rho_3,\rho_2,T;G),
 \end{equation}
   \begin{equation}
     \partial_t ((\rho_1+\rho_2+\rho_3)u) +\nabla_y\cdot ((\rho_1+\rho_2+\rho_3)(u\otimes u))+\nabla_y (S^{(1)}+S^{(2)}+S^{(3)})=0,
 \end{equation}
 \begin{multline}\label{eq.nonlte.3p5}
     \partial_t \left(\left(\rho_1+\rho_2+\rho_3\right)T+\frac{4}{3}\epsilon \rho_2+\frac{8}{3}\epsilon \rho_3\right) +\nabla_y\cdot \left(\left(\left(\rho_1+\rho_2+\rho_3\right)T+\frac{4}{3}\epsilon \rho_2+\frac{8}{3}\epsilon \rho_3\right)u\right)\\=\frac{4}{3}\epsilon( Q_{2,1}(\rho_2,\rho_1,T;G)+ Q_{3,2}(\rho_3,\rho_2,T;G)),
 \end{multline}
 where the operators are defined as
 \begin{equation*}
    H_{i,j}(\rho_i,\rho_j,T)\eqdef \int_{\mathbb{R}^3}\mathcal{K}^{(i,j)}_{non.el}[M,M]dv,
 \end{equation*}
 \begin{equation*}
 Q_{1,2}(\rho_1,\rho_2,T;G)=-Q_{2,1}(\rho_2,\rho_1,T;G)\eqdef \gamma_1 \int_{\mathbb{S}^2}(\rho_2(1+G)-\rho_1 G)dn,
 \end{equation*}
  \begin{equation*}
 Q_{2,3}(\rho_2,\rho_3,T;G)=-Q_{3,2}(\rho_3,\rho_2,T;G)\eqdef \gamma_2 \int_{\mathbb{S}^2}(\rho_3(1+G)-\rho_2 G)dn,
 \end{equation*}
 \begin{equation}\notag
    S^{(j)}\eqdef \int (v-u )\otimes (v-u )M^{(j)} dv=\frac{1}{2}\rho_j T,
\end{equation}
 where $M=\begin{pmatrix}M^{(1)}\\M^{(2)}\\M^{(3)} \end{pmatrix}$. By the same argument as in \eqref{eq.H.reduction.initial} and \eqref{eq.H.reduction.final}, we have
 \begin{equation}\label{H12.eq}
     H_{2,1}=-H_{1,2}=(\rho_1^2 e^{-\frac{2\epsilon}{T}}
    -\rho_1\rho_2)\mathcal{P}_{2,1}(T),
 \end{equation}
 and
 
 \begin{equation}
     H_{3,2}=-H_{2,3}=(\rho_2^2 e^{-\frac{2\epsilon}{T}}
    -\rho_2\rho_3)\mathcal{P}_{3,2}(T),
 \end{equation} where $\mathcal{P}_{i+1,i}$ for $i=1,2$ is defined as 
\begin{equation}\label{eq.P.3p}\mathcal{P}_{i+1,i}(T;u)\eqdef \int_{\mathbb{R}^3}dv \int_\rth dv_4 \int_\stw d\omega\ W^{(i+1,i)}_+ (v,v_4;v_1,v_2)\mathcal{Z}(v,u,T)\mathcal{Z}(v_4,u,T),\end{equation} with   $$\mathcal{Z}(v,u,T)\eqdef\frac{c_0}{T^{\frac{3}{2}}}\exp\left(-\frac{|v-u|^2}{T}\right),$$ 
 where $W^{(i+1,i)}_+$ for $i=1,2$ is the collision cross-section for the nonelastic operator $\mathcal{K}^{(i+1,i)}_{non.el}$ which satisfies
 $$W^{(i+1,i)}_+(v_3,v_4;v_1,v_2)\eqdef W_+(v_3,v_4;v_1,v_2)\text{ for }i=1,2.$$ Here we simplified the notation and used $\mathcal{P}_{i,j}(T)\eqdef \mathcal{P}_{i,j}(T;0)$.  Similarly, we define $\mathcal{P}_{i,i+1}$ for $i=1,2$ such that it satisfies
 $$H_{i,i+1}=(\rho_i\rho_{i+1}-\rho_i^2e^{-\frac{2\epsilon}{T}})\mathcal{P}_{i,i+1}(T)\text{ for }i=1,2.$$ 
 
 \subsection{Stationary equations with $u=0$}
 In this subsection and the forthcoming subsection, we examine the existence of the steady states to the hydrodynamic equations \eqref{eq.nonlte.3p1}-\eqref{eq.nonlte.3p5} with $u=0$ for a linearized version of the problem. They satisfy the following system of equations:
 \begin{equation}\notag\begin{split}
     &  \sigma H_{1,2}(\rho_1,\rho_2,T)+Q_{1,2}(\rho_1,\rho_2,T;G)=0,\\
     &  \sigma( H_{2,1}(\rho_2,\rho_1,T)+ H_{2,3}(\rho_2,\rho_3,T))+Q_{2,1}(\rho_2,\rho_1,T;G)+Q_{2,3}(\rho_2,\rho_3,T;G)=0,\\
     &  \sigma H_{3,2}(\rho_3,\rho_2,T)+Q_{3,2}(\rho_3,\rho_2,T;G)=0,\\
     &  \nabla_y (S^{(1)}+S^{(2)}+S^{(3)})=0,\\
     &\frac{4}{3}\epsilon( Q_{2,1}(\rho_2,\rho_1,T;G)+ Q_{3,2}(\rho_3,\rho_2,T;G))=0.
 \end{split}
 \end{equation}
 Using $Q_{2,1}=-Q_{1,2}$ and $Q_{3,2}=-Q_{2,3}$, we obtain the following set of equations for the steady states 
 \begin{equation}\label{final eqs. 3p}\begin{split}
     &  \sigma H_{1,2}+Q_{1,2}=0,\\
     &  \sigma H_{2,3}+Q_{2,3}=0,\\
     & \frac{1}{2}(\rho_1+\rho_2+\rho_3)T = \text{(constant)},\\
     &Q_{1,2}+Q_{2,3}=0.\end{split}
 \end{equation}
 For the radiation intensity $G$, we recall \eqref{eq.rescaled kinetic equations.3p} and have
 \begin{equation}\label{rad.3p}
        n\cdot \nabla_y G=\gamma_1\epsilon (\rho_2(1+G)-\rho_1G)+\gamma_2\epsilon (\rho_3(1+G)-\rho_2G).
 \end{equation}Finally, the set of equations above is coupled with the total mass conservation that
 \begin{equation}\label{mass.3p}
     \int_{\Omega_y} (\rho_1+\rho_2+\rho_3)dy = m_0,
 \end{equation}where $\Omega_y$ is the rescaled domain of $\Omega$ in the variable $y$.
 By \eqref{final eqs. 3p}$_1$, \eqref{final eqs. 3p}$_2$, and \eqref{final eqs. 3p}$_4$, we have 
 $H_{2,1}+H_{3,2}=0.$ 
 In the simplest situation where we assume $\mathcal{P}_{2,1}=\mathcal{P}_{3,2}$ in \eqref{eq.P.3p}, we have
 $$\rho_1^2e^{-\frac{2\epsilon}{T}}-\rho_1\rho_2+\rho_2^2e^{-\frac{2\epsilon}{T}}-\rho_2\rho_3=0.$$ Note that this equation does not necessarily yield the LTE situation where the density distributions are related by the Boltzmann ratio as $\rho_3=\rho_2e^{-\frac{2\epsilon}{T}}$ and $\rho_2=\rho_1e^{-\frac{2\epsilon}{T}}$. This contrasts with the situations for the two-level system in Section \ref{sec.nonLTEexist} where the stationary system with $u=0$ yields $H^{(1)}=H^{(2)}=0$, which yields back the LTE situation $\rho_2=\rho_1^{-\frac{2\epsilon_0}{T}}$. 
 \subsection{Linearized problem for the stationary equations}In general, it is difficult to study the full solutions to the stationary equations \eqref{final eqs. 3p}-\eqref{mass.3p}. In this subsection, we study the linearized problem of the three-level stationary system  \eqref{final eqs. 3p} -\eqref{mass.3p}. The idea is to take radiation close to equilibrium with small perturbation parameter. It will be then possible to use the methods used in Section \ref{sec.LTEexist} to obtain the temperature for each given incoming radiation flux.   We study under the scaling where $\frac{\epsilon}{T}$ is of order one and the radiation is close to the equilibrium of the pseudo-Planck distribution \eqref{pseudo planck} as in Section \ref{sec.linearized}. 
 
 By following similar arguments of Section \ref{sec.constant solutions}, we can assume that under some proper boundary conditions as in Proposition \ref{prop.constant.sol} the system \eqref{final eqs. 3p} -\eqref{mass.3p} admits a constant stationary solution.  We have a stationary solution with constant temperature and gas density given by
 \begin{equation}\label{constant.stationary.3p}(\rho_1,\rho_2,\rho_3,u,T,G)=(\rho_0, \rho_0e^{-\frac{2\epsilon}{T_0}},\rho_0e^{-\frac{4\epsilon}{T_0}},0,T_0,G_p),\end{equation} where $\rho_0$ and $T_0$ are positive constants and $G_p$ is the pseudo-Planck equilibrium defined as 
 $$G_p\eqdef \frac{e^{-\frac{2\epsilon}{T_0}}}{1-e^{-\frac{2\epsilon}{T_0}}}.$$ 
 We will consider the perturbations nearby the constant stationary solution \eqref{constant.stationary.3p}. First of all we write for $j=1,2,3,$
 $$\rho_j=\rho_0e^{-(j-1)\frac{2\epsilon}{T}}(1+\sigma_j),\ T=T_0(1+\xi),\ G=G_p(1+h),$$ for some small perturbations $|\sigma_j|\ll \rho_0$, $|\xi|\ll T_0$, and $|h|\ll 1,$ with the scaling $\frac{2\epsilon}{T_0}\simeq 1.$
  Then we now linearize the system of equations \eqref{final eqs. 3p} -\eqref{mass.3p}. Then we obtain
  \begin{multline*}
      H_{2,3}=-H_{3,2}=\mathcal{P}_{2,3} (T) (\rho_2\rho_3-\rho_2^2e^{-\frac{2\epsilon}{T}})\\\simeq \mathcal{P}_{2,3} (T_0)\rho_0^2 (e^{-\frac{6\epsilon}{T_0}}(1+\sigma_2)(1+\sigma_3)-e^{-\frac{4\epsilon}{T_0}}e^{-\frac{2\epsilon}{T_0}(1-\xi)}(1+\sigma_2)^2)\\
       \simeq \mathcal{P}_{2,3} (T_0)\rho_0^2e^{-\frac{6\epsilon}{T_0}} \left(\sigma_3-\sigma_2-\frac{2\epsilon}{T_0}\xi\right).
  \end{multline*}
  Similarly, we use \eqref{H12.eq} and have
$$
      H_{1,2}=-H_{2,1}
       \simeq \mathcal{P}_{1,2} (T_0)\rho_0^2e^{-\frac{2\epsilon}{T_0}} \left(\sigma_2-\sigma_1-\frac{2\epsilon}{T_0}\xi\right).
$$
 In addition, we have
 \begin{multline*}
     Q_{1,2}=-Q_{2,1}=\gamma_1 \int_{\mathbb{S}^2}(\rho_2(1+G)-\rho_1G)dn\\
     \simeq \gamma_1\rho_0 \int_{\mathbb{S}^2}(e^{-\frac{2\epsilon}{T_0}}(1+\sigma_2)(1+G_p+G_ph)-(1+\sigma_1)(G_p+G_ph))dn\\
     \simeq \gamma_1\rho_0 \left(4\pi G_p\left( \sigma_2-
     \sigma_1\right)-e^{-\frac{2\epsilon}{T_0}}\int_{\mathbb{S}^2}hdn\right),
 \end{multline*} since $e^{-\frac{2\epsilon}{T_0}}(1+G_p)=G_p.$
 Similarly, we also deduce that
$$
    Q_{2,3}=-Q_{3,2}
     \simeq \gamma_2\rho_0 e^{-\frac{2\epsilon}{T_0}}\left(4\pi G_p\left( \sigma_3-
     \sigma_2\right)-e^{-\frac{2\epsilon}{T_0}}\int_{\mathbb{S}^2}hdn\right).
$$
 These operators are the energy fluxes due to the radiation.
 From the radiation equation \eqref{rad.3p},
 we also have$$
      n\cdot \nabla_y (G_ph)=\gamma_1\epsilon \rho_0G_p (\sigma_2-\sigma_1-h(1-e^{-\frac{2\epsilon}{T_0}}))+\gamma_2\epsilon \rho_0 G_p e^{-\frac{2\epsilon}{T_0}}(\sigma_3-\sigma_2-h(1-e^{-\frac{2\epsilon}{T_0}})).$$
 Therefore, we obtain
  \begin{equation}\label{eq.linearized.radiation}
      n\cdot \nabla_y h=\gamma_1\epsilon \rho_0 (\sigma_2-\sigma_1-h(1-e^{-\frac{2\epsilon}{T_0}}))+\gamma_2\epsilon \rho_0  e^{-\frac{2\epsilon}{T_0}}(\sigma_3-\sigma_2-h(1-e^{-\frac{2\epsilon}{T_0}})).
 \end{equation}
 From the pressure equation \eqref{final eqs. 3p}$_3$, we deduce that
 \begin{equation}\label{eq.20}(1+e^{-\frac{2\epsilon}{T_0}}+e^{-\frac{4\epsilon}{T_0}})\xi + (\sigma_1+e^{-\frac{2\epsilon}{T_0}}\sigma_2+e^{-\frac{4\epsilon}{T_0}}\sigma_3)=C_0,\end{equation} for some constant $C_0>0.$ Then together with \eqref{mass.3p}, we further have
 \begin{equation}
     \label{mass.3p.final}
    (1+e^{-\frac{2\epsilon}{T_0}}+e^{-\frac{4\epsilon}{T_0}})\int_{\Omega_y}\xi(y) \ dy =C_0|\Omega_y|-\left(m_0-\rho_0|\Omega_y|(1+e^{-\frac{2\epsilon}{T_0}}+e^{-\frac{4\epsilon}{T_0}})\right),
 \end{equation}where $C_0$ and $m_0$ are given. 
 The constant $C_0$ depends on the total amount of gas in the container. The pressure of the gas is affected by the value of $C_0.$
 
 Then the solvability of the system is as follows. Firstly, we can represent $h$ linearly in terms of $\sigma_1,\sigma_2,\sigma_3$ by solving the radiation equation \eqref{eq.linearized.radiation} as in Section \ref{sec.linearized existence}.  We then have from $Q_{1,2}+Q_{2,3}=0$ that
 \begin{equation}\label{eq.1}
      4\pi \gamma_1 G_p\left( \sigma_2-
     \sigma_1\right)+ e^{-\frac{2\epsilon}{T_0}}4\pi\gamma_2 G_p\left( \sigma_3-
     \sigma_2\right)=e^{-\frac{2\epsilon}{T_0}}\left(\gamma_1+\gamma_2e^{-\frac{2\epsilon}{T_0}}\right)\int_{\mathbb{S}^2}hdn,
 \end{equation} where $\int_{\mathbb{S}^2}hdn$ is further represented linearly in terms of $\sigma_1,\sigma_2,\sigma_3$ via the solution $h$ of the radiation equation \eqref{eq.linearized.radiation}.  Also, from \eqref{eq.20}, we obtain
 \begin{equation}\label{eq.2} \sigma_1+e^{-\frac{2\epsilon}{T_0}}\sigma_2+e^{-\frac{4\epsilon}{T_0}}\sigma_3=C_0-(1+e^{-\frac{2\epsilon}{T_0}}+e^{-\frac{4\epsilon}{T_0}})\xi(y) ,\end{equation} where $C_0$ is the constant that depends on $\int_{\Omega_y}\xi dy$ and the given constants $m_0$, $\rho_0$, $T_0$, and $\Omega_y$ as in \eqref{mass.3p.final}. 
 Finally, from $H_{1,2}+H_{2,3}=0,$ we have
 \begin{equation}
     \label{eq.3}
     \mathcal{P}_{1,2}(T_0)(\sigma_2-\sigma_1)+\mathcal{P}_{2,3}(T_0)e^{-\frac{4\epsilon}{T_0}}(\sigma_3-\sigma_2)=\frac{2\epsilon}{T_0}\left(\mathcal{P}_{1,2}+\mathcal{P}_{2,3}(T_0)e^{-\frac{4\epsilon}{T_0}}\right)\xi(y).
 \end{equation} We then deduce a linear system of equations \eqref{eq.1}-\eqref{eq.3} for $\sigma_1,\sigma_2,\sigma_3$ and generically they are independent. We then obtain a solution $\sigma_1,\sigma_2,\sigma_3$ and for generic choices of $\gamma_1,$ $\gamma_2,$ $\mathcal{P}_{1,2}$ and $\mathcal{P}_{2,3}$ we have 
 $\sigma_1\neq\sigma_2\neq \sigma_3.$ Thus, this yields a general non-LTE stationary solution to the linearized problem.

This completes the study of the Euler-like system coupled with radiation yielding non-LTE steady states in a gas with three-level molecules.

\section*{Acknowledgement} The authors gratefully acknowledge the support of the grant CRC
1060 ``The Mathematics of Emergent Effects" of the University of Bonn funded
through the Deutsche Forschungsgemeinschaft (DFG, German Research Foundation). J. W. Jang is supported by Basic Science Research Institute Fund of Korea, whose NRF grant number is 2021R1A6A1A10042944. Juan J. L. Vel\'azquez is also funded by DFG under Germany's Excellence Strategy-EXC-2047/1-390685813.

\appendix

\section{Derivation of the nonelastic collision operators via weak formulation}
\label{sec.weakfor derivation}
In this section, we will introduce how to derive the explicit forms of the nonelastic Boltzmann operator of \eqref{eq.knonel1}-\eqref{eq.knonel12} for the reaction \eqref{eq.nonelastic reactions}.  
  To this end, we will first write the kinetic system of radiation from Section \ref{sec.kinetic system} in the weak form as follows.  

\subsection{Weak formulation}
For the weak formulation of the system, we also introduce a relevant physical boundary condition for $x\in\partial\Omega$. We assume that we consider the molecules in a bounded domain $\Omega\subset \mathbb{R}^3$ where we impose the specular boundary condition for the gas molecules at the boundary $\partial \Omega$. Then we introduce the weak formulation of the system of the gas molecules.

\begin{definition}We say that a vector of non-negative Radon measures  $(F,Q)=(\Fon,\Ftw,Q)^\top $ is a weak solution to the kinetic system for radiation if for any $\varphi^{(i)}\in C^\infty_c( \rth)$ for $i=1,2,3$, it satisfies  
\begin{multline}\label{weak.formulation}
    \partial_t \int_{\Omega}dx\bigg(\int_\rth dv\ \left( \Fon\varphi^{(1)}(v)+\Ftw\varphi^{(2)}(v)\bigg)+\int_{\mathbb{S}^2}dn\ Q\varphi^{(3)}(n)\right) \\
    =\frac{1}{2}\sum_{i=1}^2\sum_{j=1}^2\int_{\Omega}dx\int_\rth dv_1\int_\rth dv_2 \int_\stw d\omega \ W_{el}^{(i,j)}F_1^{(i)}F_2^{(j)}\\\times \left(\varphi^{(i)}(v_3)+\varphi^{(i)}(v_4)-\varphi^{(i)}(v_1)-\varphi^{(i)}(v_2)\right) \\
    +\int_{\Omega}dx\iiiint_{\mathbb{R}^{12}} d\bar{v}_1d\bar{v}_2d\bar{v}_3d\bar{v}_4\ \delta(\bar{v}_1+\bar{v}_2-\bar{v}_3-\bar{v}_4)\delta(|\bar{v}_1|^2+|\bar{v}_2|^2-|\bar{v}_3|^2-|\bar{v}_4|^2-2\epsilon_0)\\\times W_{non.el}(\bar{v}_1,\bar{v}_2;\bar{v}_3,\bar{v}_4)\bigg[\bar{F}^{(1)}_1\bar{F}^{(1)}_2(\varphi^{(1)}(\bar{v}_4)+\varphi^{(2)}(\bar{v}_3)-\varphi^{(1)}(\bar{v}_1)-\varphi^{(1)}(\bar{v}_2))\\
    +\bar{F}^{(2)}_3\bar{F}^{(1)}_4(\varphi^{(1)}(\bar{v}_1)+\varphi^{(1)}(\bar{v}_2)-\varphi^{(1)}(\bar{v}_4)-\varphi^{(2)}(\bar{v}_3))\bigg]\\
    +\int_{\Omega}dx\int_\rth dv\int_{\mathbb{S}^2} dn\ h_{rad}[\Fon ,\Ftw ,Q](\varphi^{(1)}(v)-\varphi^{(2)}(v)+\varphi^{(3)}(n)),
\end{multline}
where for both $j=1,2$, $\varphi^{(j)}(t,x,v)=\varphi^{(j)}(t,x,v-2(v\cdot n_x)n_x)$ with the outward normal vector $n_x$ at the boundary point $x\in \partial\Omega$ and
$$ h_{rad}[\Fon ,\Ftw ,Q]=\frac{B_{12}}{4\pi} \left[\frac{2h\nu_0^3}{c^2}\Ftw (v)\left(1+\frac{c^3}{2\nu_0^2}Q\right)-\Fon (v) c\epsilon_0 Q \right].$$
\end{definition}

As a consequence, we obtain the conservation of the number of gas molecules that 
$$
\frac{d}{dt}\bigg(	\int_\Omega\int_\rth \Fon dvdx+ \int_\Omega\int_\rth \Ftw dvdx\bigg)=0,
$$ with $\varphi^{(1)}=\varphi^{(2)}=1$ and $\varphi^{(3)}=0$.
In addition, we obtain the conservation of total energy in a closed system
 $$
\frac{d}{dt}\bigg(	\int_\Omega\bigg(\int_\rth \left(\frac{|v|^2}{2}\Fon + \left(\frac{|v|^2}{2}+\epsilon_0\right)\Ftw\right) dv + \epsilon_0\int_{\mathbb{S}^2}Qdn \bigg)dx\bigg)=0,
$$with $\varphi^{(1)}=\frac{|v|^2}{2}$, $\varphi^{(2)}=\frac{|v|^2}{2}+\epsilon_0$, and $\varphi^{(3)}=\epsilon_0.$   The examples of the boundary-value conditions that yield the energy conservation include the case of specularly-reflected molecules and radiation at the boundary in a bounded domain or in a torus.

\subsection{Derivation of the nonelastic operators}
Now we first let 
$$\bar{v}_3=\frac{\bar{v}_1+\bar{v}_2}{2}+k\omega,\ \bar{v}_4=\frac{\bar{v}_1+\bar{v}_2}{2}-k\omega,$$ for some $k=k(\bar{v}_1,\bar{v}_2,\omega)\ge 0$. 
Now, we plug this into the nonelastic terms in the weak formulation \eqref{weak.formulation}. Then we observe that 
\begin{multline*}
   I_{non.el}\eqdef \iiiint_{\mathbb{R}^{12}} d\bar{v}_1d\bar{v}_2d\bar{v}_3d\bar{v}_4\ \delta(\bar{v}_1+\bar{v}_2-\bar{v}_3-\bar{v}_4)\delta(|\bar{v}_1|^2+|\bar{v}_2|^2-|\bar{v}_3|^2-|\bar{v}_4|^2-2\epsilon_0)\\\times W_{non.el}(\bar{v}_1,\bar{v}_2;\bar{v}_3,\bar{v}_4)\bigg[\bar{F}^{(1)}_1\bar{F}^{(1)}_2(\varphi^{(1)}(\bar{v}_4)+\varphi^{(2)}(\bar{v}_3)-\varphi^{(1)}(\bar{v}_1)-\varphi^{(1)}(\bar{v}_2))\\
    +\bar{F}^{(2)}_3\bar{F}^{(1)}_4(\varphi^{(1)}(\bar{v}_1)+\varphi^{(1)}(\bar{v}_2)-\varphi^{(1)}(\bar{v}_4)-\varphi^{(2)}(\bar{v}_3))\bigg]\\
     =\iint_{\mathbb{R}^{6}} d\bar{v}_1d\bar{v}_2 \int_0^\infty dk \ k^2 \int_{\mathbb{S}^2}d\omega \  \delta\bigg(|\bar{v}_1|^2+|\bar{v}_2|^2-\bigg|\frac{\bar{v}_1+\bar{v}_2}{2}+k\omega\bigg|^2-\bigg|\frac{\bar{v}_1+\bar{v}_2}{2}-k\omega\bigg|^2-2\epsilon_0\bigg)\\\times W_{non.el}\bigg(\bar{v}_1,\bar{v}_2;\frac{\bar{v}_1+\bar{v}_2}{2}+k\omega,\frac{\bar{v}_1+\bar{v}_2}{2}-k\omega\bigg)\\\times \bigg[\bar{F}^{(1)}_1\bar{F}^{(1)}_2\bigg(\varphi^{(1)}\bigg(\frac{\bar{v}_1+\bar{v}_2}{2}-k\omega\bigg)+\varphi^{(2)}\bigg(\frac{\bar{v}_1+\bar{v}_2}{2}+k\omega\bigg)-\varphi^{(1)}(\bar{v}_1)-\varphi^{(1)}(\bar{v}_2)\bigg)\\
    +\bar{F}^{(2)}\bigg(\frac{\bar{v}_1+\bar{v}_2}{2}+k\omega\bigg)\bar{F}^{(1)}\bigg(\frac{\bar{v}_1+\bar{v}_2}{2}-k\omega\bigg)\\\times \bigg(\varphi^{(1)}(\bar{v}_1)+\varphi^{(1)}(\bar{v}_2)-\varphi^{(1)}\bigg(\frac{\bar{v}_1+\bar{v}_2}{2}-k\omega\bigg)-\varphi^{(2)}\bigg(\frac{\bar{v}_1+\bar{v}_2}{2}+k\omega\bigg)\bigg)\bigg].
\end{multline*}
Then note that we also have
\begin{multline*}
    \delta\bigg(|\bar{v}_1|^2+|\bar{v}_2|^2-\bigg|\frac{\bar{v}_1+\bar{v}_2}{2}+k\omega\bigg|^2-\bigg|\frac{\bar{v}_1+\bar{v}_2}{2}-k\omega\bigg|^2-2\epsilon_0\bigg)\\
    =\delta\bigg(2\left(k-\sqrt{\frac{|\bar{v}_1-\bar{v}_2|^2}{4}-\epsilon_0}\right)\left(k+\sqrt{\frac{|\bar{v}_1-\bar{v}_2|^2}{4}-\epsilon_0}\right)\bigg).
\end{multline*}
By evaluating this delta function, we obtain that 
\begin{multline*}
    I_{non.el} 
    =\iint_{\mathbb{R}^{6}} d\bar{v}_1d\bar{v}_2 \frac{1}{4}\sqrt{\frac{|\bar{v}_1-\bar{v}_2|^2}{4}-\epsilon_0} \int_{\mathbb{S}^2}d\omega \   W_{non.el}\bigg(\bar{v}_1,\bar{v}_2;\frac{\bar{v}_1+\bar{v}_2}{2}+k\omega,\frac{\bar{v}_1+\bar{v}_2}{2}-k\omega\bigg)\\\times \bigg[
    (\varphi^{(1)}(\bar{v}_1)+\varphi^{(1)}(\bar{v}_2))(\bar{F}^{(2)}_3\bar{F}^{(1)}_4-\bar{F}^{(1)}_1\bar{F}^{(1)}_2)
    +\varphi^{(1)}\bigg(\frac{\bar{v}_1+\bar{v}_2}{2}-k\omega\bigg)(\bar{F}^{(1)}_1\bar{F}^{(1)}_2-\bar{F}^{(2)}_3\bar{F}^{(1)}_4)\\
    +\varphi^{(2)}\bigg(\frac{\bar{v}_1+\bar{v}_2}{2}+k\omega\bigg)(\bar{F}^{(1)}_1\bar{F}^{(1)}_2-\bar{F}^{(2)}_3\bar{F}^{(1)}_4)\bigg]
    \eqdef I_1+I_2+I_3,
\end{multline*}
where $k=\sqrt{\frac{|\bar{v}_1-\bar{v}_2|^2}{4}-\epsilon_0}$ and $\bar{v}_3$ and $\bar{v}_4$ are now defined as in \eqref{VelocInelCase}. 
Assume that $W_{non.el}$ also satisfies $W_{non.el}(a,b;c,d)=W_{non.el}(c,d;a,b).$ In general we write $$W_{non.el}\bigg(\bar{v}_1,\bar{v}_2;\frac{\bar{v}_1+\bar{v}_2}{2}+k\omega,\frac{\bar{v}_1+\bar{v}_2}{2}-k\omega\bigg)=W_{non.el}\bigg(|\bar{v}_1-\bar{v}_2|,(\bar{v}_1-\bar{v}_2)\cdot \omega\bigg).$$One example can be the one given in \eqref{non.el.cross.section}. 

For $I_1$ note that we can interchange the variables $\bar{v}_1 \leftrightarrows \bar{v}_2$ for the terms multiplied by $\varphi^{(1)}(\bar{v}_2)$ and see that
\begin{multline*}
    I_1=2\iint_{\mathbb{R}^{6}} d\bar{v}_1d\bar{v}_2 \frac{1}{4}\sqrt{\frac{|\bar{v}_1-\bar{v}_2|^2}{4}-\epsilon_0} \int_{\mathbb{S}^2}d\omega \  \\\times W_{non.el}\bigg(|\bar{v}_1-\bar{v}_2|,(\bar{v}_1-\bar{v}_2)\cdot \omega\bigg)
    \varphi^{(1)}(\bar{v}_1)(\bar{F}^{(2)}_3\bar{F}^{(1)}_4-\bar{F}^{(1)}_1\bar{F}^{(1)}_2).
\end{multline*}
Then by the weak formulation \eqref{weak.formulation} this corresponds to the term $2\mathcal{K}_{1,1}[F,F]$ in \eqref{eq.knonel1} and we obtain the explicit form \eqref{eq.knonel11} by the additional choice of the cross section that
\begin{equation}\label{non.el.cross.section}
   W_{non.el}(|\bar{v}_1-\bar{v}_2|,(\bar{v}_1-\bar{v}_2)\cdot \omega)\eqdef \frac{B_{non.el}(|\bar{v}_1-\bar{v}_2|,(\bar{v}_1-\bar{v}_2)\cdot \omega)}{\frac{1}{4}|\bar{v}_1-\bar{v}_2|}= \frac{C_0|\omega \cdot (\bar{v}_1-\bar{v}_2)|}{|\bar{v}_1-\bar{v}_2|},
\end{equation}which leads the operators to the ones for the well-known hard-sphere collisions.

For $I_2+I_3$, we recover the delta functions of energy and momentum conservation laws and write again as
\begin{multline*}
   I_2+I_3\eqdef \iiiint_{\mathbb{R}^{12}} d\bar{v}_1d\bar{v}_2d\bar{v}_3d\bar{v}_4\ \delta(\bar{v}_1+\bar{v}_2-\bar{v}_3-\bar{v}_4)\delta(|\bar{v}_3|^2+|\bar{v}_4|^2-|\bar{v}_1|^2-|\bar{v}_2|^2+2\epsilon_0)\\\times W_{non.el}(\bar{v}_1,\bar{v}_2;\bar{v}_3,\bar{v}_4)(\varphi^{(1)}(\bar{v}_4)+\varphi^{(2)}(\bar{v}_3))(\bar{F}^{(1)}_1\bar{F}^{(1)}_2-\bar{F}^{(2)}_3\bar{F}^{(1)}_4).
\end{multline*}
This time, we evaluate the delta function by choosing $\bar{v}_1=\bar{v}_3+\bar{v}_4-\bar{v}_2.$ Then by making further change of variables $\bar{v}_2 \mapsto (k',\omega')$ such that
$$\bar{v}_1=\frac{\bar{v}_3+\bar{v}_4}{2}+k'\omega',\ \bar{v}_2=\frac{\bar{v}_3+\bar{v}_4}{2}-k'\omega',$$ we obtain that the energy delta function is now equal to 
\begin{multline*}
    \delta\bigg(|\bar{v}_3|^2+|\bar{v}_4|^2-\bigg|\frac{\bar{v}_3+\bar{v}_4}{2}+k'\omega'\bigg|^2-\bigg|\frac{\bar{v}_3+\bar{v}_4}{2}-k'\omega'\bigg|^2+2\epsilon_0\bigg)\\
    =\delta\bigg(2\left(k'-\sqrt{\frac{|\bar{v}_3-\bar{v}_4|^2}{4}+\epsilon_0}\right)\left(k'+\sqrt{\frac{|\bar{v}_3-\bar{v}_4|^2}{4}+\epsilon_0}\right)\bigg).
\end{multline*}
Then similarly, we obtain that
\begin{multline*}
    I_2+I_3=\iint_{\mathbb{R}^{6}} d\bar{v}_3d\bar{v}_4 \frac{1}{4}\sqrt{\frac{|\bar{v}_3-\bar{v}_4|^2}{4}+\epsilon_0} \int_{\mathbb{S}^2}d\omega' \  \\\times W_{non.el}\bigg(\bar{v}_3,\bar{v}_4;\frac{\bar{v}_3+\bar{v}_4}{2}+k'\omega',\frac{\bar{v}_3+\bar{v}_4}{2}-k'\omega'\bigg) (\varphi^{(1)}(\bar{v}_4)+\varphi^{(2)}(\bar{v}_3))(\bar{F}^{(1)}_1\bar{F}^{(1)}_2-\bar{F}^{(2)}_3\bar{F}^{(1)}_4).
\end{multline*}
Using \eqref{non.el.cross.section}, we also write
$$W_{non.el}\bigg(\bar{v}_3,\bar{v}_4;\frac{\bar{v}_3+\bar{v}_4}{2}+k'\omega',\frac{\bar{v}_3+\bar{v}_4}{2}-k'\omega'\bigg)=\frac{C_0|\omega'\cdot (\bar{v}_3-\bar{v}_4)|}{|\bar{v}_3-\bar{v}_4|},$$ in the hard-sphere case.
Then we obtain that the term with $\varphi^{(1)}(\bar{v}_4)$ exactly corresponds to $\mathcal{K}^{(1)}_{1,2}$ in \eqref{eq.knonel12} by the weak formulation \eqref{weak.formulation}. Similarly, term with $\varphi^{(2)}(\bar{v}_3)$ corresponds to $\Kntw$ in \eqref{eq.knonel2}.
 This completes the derivation of the nonelastic collision operators.

\section{Detailed balance property for systems with radiation}
In this section, we discuss a general situation where we derive the detailed balance via the well-known Planck distribution at equilibrium. This has been already introduced in \cite[Section 1.4.4]{oxenius} and we summarize it here for the readers' convenience. We start with a more generalized radiation operator as follows (cf. \eqref{eq.radiation}):
\begin{equation}\label{eq.radiation.general}
	\frac{1}{c}\frac{\partial I_{\nu}}{\partial t} + n\cdot \nabla_x I_{\nu} =\frac{h\nu} {4\pi}\int_\rth dv\ \left[ A_{21}\Ftw (v) \left(1+\frac{B_{21}}{A_{21}}I_{\nu}\right)\eta_{21}(\xi,n)-B_{12}\Fon(v) I_{\nu}(n)\alpha_{12}(\xi)\right],
\end{equation}
where $\xi=\nu-\frac{\nu_0}{c}n\cdot v$ describes the Doppler effect and $\alpha$ and $\eta$ are the atomic absorption and emission profiles transformed from the lab frame to the rest frame of the atom moving with the velocity $v$.  In this section, we can assume $\alpha_{1,2}=\eta_{2,1}$ for the sake of simplicity.  
Differently from the construction in Section \ref{sec.eq.rad}, we will also take the number of photon states into account. First of all, we define the number density of photons $Q(\nu)$ as
$Q(\nu)\eqdef \frac{I_\nu(n)}{h\nu c}$ such that the number of photons with in the frequency range $[\nu,\nu+d\nu]$ is equal to $Q(\nu)d\nu$. We also define $G(\nu) d\nu$ as the number of photon states in the frequency range $[\nu,\nu+d\nu]$ is equal to $Q(\nu)d\nu$. Then we understand that the number of possible photon states is proportional to the number of changes in the molecule states. Similarly, we denote the corresdponding quantities for the ground-state molecule as $\bar{Q}$ and $\bar{G}$ in the reaction \eqref{eq.reactions2}. Then we denote the mean occupation number of a photon state of frequency $\nu$ as $z(\nu)=\frac{Q(\nu)}{G(\nu)}.$
Then the detailed balance of the reaction \eqref{eq.reactions2} can be obtained if
$$\bar{Q}(E/h)\bar{G}(E/h-\nu) G(\nu)[1+z(\nu)]=\bar{Q}(E/h-\nu)\bar{G}(E/h) N(\nu),$$ holds.
This holds when
$$\frac{\bar{Q}(E/h)}{\bar{G}(E/h)}[1+z(\nu)]=\frac{\bar{Q}(E/h-\nu)}{\bar{G}(E/h-\nu)}z(\nu),$$
and hence
$$\frac{1+z(\nu)}{z(\nu)}=\frac{\frac{\bar{Q}(E/h-\nu)}{\bar{G}(E/h-\nu)}}{\frac{\bar{Q}(E/h)}{\bar{G}(E/h)}}=e^{\beta h\nu}.$$ Thus the Planck distribution
$$Q(\nu)=\frac{G(\nu)}{e^{\beta h\nu}-1},$$ implies the detailed balance.

\section{Reduction of the integral operators for numerical simulations}\label{sec.appendix.reduction}This section is devoted to the reduction of the representations of the functions $S$ and $L$ for the numerical simulations in Section \ref{sec.numerics}.   To this end, we will consider the reduction of
 $$\frac{\mathcal{P}(T_2,T_1)-\mathcal{P}(T_1)}{T_2-T_1} \text { and } \frac{\mathcal{B}(T_1,T_2)}{T_2-T_1}.$$ 
 Using \eqref{new.P} we first have$$
     \mathcal{P}(T_2,T_1)=\mathcal{P}(T_1,T_2)
     = 4\pi \int_{\mathbb{R}^3}dv_3 \int_\rth dv_4  \sqrt{|v_3-v_4|^2+4\epsilon_0}\mathcal{Z}(v_3,0,T_1)\mathcal{Z}(v_4,0,T_2).
$$
 We make the change of variables $(v_3,v_4)\mapsto (w_3,w_4)\eqdef \left(\frac{v_3}{\sqrt{T_1}},\frac{v_4}{\sqrt{T_2}}\right)$ for $\mathcal{P}(T_2,T_1)$ and $(v_3,v_4)\mapsto (w_3,w_4)\eqdef \left(\frac{v_3}{\sqrt{T_1}},\frac{v_4}{\sqrt{T_1}}\right)$ for $\mathcal{P}(T_1)$. 
 Then we have
 \begin{multline}\label{P.diff}
   \frac{  \mathcal{P}(T_2,T_1)
    -     \mathcal{P}(T_1)}{T_2-T_1}
     = \frac{4\pi}{T_2-T_1} \int_{\mathbb{R}^3}dw_3 \int_\rth dw_4  \mathcal{Z}(w_3,0,1)\mathcal{Z}(w_4,0,1)\\\times \left(\sqrt{|\sqrt{T_1}w_3-\sqrt{T_2}w_4|^2+4\epsilon_0}-\sqrt{|\sqrt{T_1}w_3-\sqrt{T_1}w_4|^2+4\epsilon_0}\right).
 \end{multline}We then observe
 \begin{multline}\label{def.R1R2}
     \left(\sqrt{|\sqrt{T_1}w_3-\sqrt{T_2}w_4|^2+4\epsilon_0}-\sqrt{|\sqrt{T_1}w_3-\sqrt{T_1}w_4|^2+4\epsilon_0}\right)\\=\sqrt{T_1}\bigg(\sqrt{|w_4-w_3|^2+2\delta(w_4-w_3)\cdot w_4 +\delta^2|w_4|^2+\frac{4\epsilon_0}{T_1}}-\sqrt{|w_4-w_3|^2+\frac{4\epsilon_0}{T_1}}\bigg)\\
     \eqdef \sqrt{T_1}(R_1-R_2),
 \end{multline}where we define
 $$\delta=\sqrt{\frac{T_2}{T_1}}-1.$$ Then we further have$$
     \sqrt{T_1}(R_1-R_2)=\frac{\sqrt{T_1}}{R_1+R_2}(2\delta(w_4-w_3)\cdot w_4+\delta^2|w_4|^2).$$
 Note that
 $$\frac{\delta}{T_2-T_1}=\frac{1}{T_1(\sqrt{\frac{T_2}{T_1}}+1)},$$
 and $$\frac{\delta^2}{T_2-T_1}=\frac{\sqrt{\frac{T_2}{T_1}}-1}{T_1(\sqrt{\frac{T_2}{T_1}}+1)}.$$
 Now we consider the reduction of the dimensionality of the integral. We note that the term in \eqref{def.R1R2} only depend on 
 $|w_3-w_4|^2$, $(w_3-w_4)\cdot w_4$, and $|w_4|^2$. In addition we observe 
 $$e^{-|w_3|^2-|w_4|^2} = e^{-|w_3-w_4|^2-2(w_3-w_4)\cdot w_4-2|w_4|^2}.$$ Furthermore, we can also write $|w_3-w_4|^2$, $(w_3-w_4)\cdot w_4$ in terms of $|w_3-w_4|^2$, $(w_3-w_4)\cdot (w_4+w_3)$, and $|w_3+w_4|^2$ only.
 Therefore, removing the singularity of $\frac{1}{T_2-T_1}$ in \eqref{P.diff}, we can write
 \begin{multline*}\frac{\mathcal{P}(T_2,T_1)
    -     \mathcal{P}(T_1)}{T_2-T_1}\\
    =\iint dw_3dw_4\  F_\delta(|w_3-w_4|^2, (w_3-w_4)\cdot (w_4+w_3),|w_3+w_4|^2)e^{-\frac{1}{2}|w_3-w_4|^2-\frac{1}{2}|w_3+w_4|^2},\end{multline*} where $F_\delta$ is defined as
    \begin{multline*}
        F_\delta(a,b,c)\eqdef \frac{4\pi}{\sqrt{a+\delta(a-b)+\delta^2(\frac{a}{4}+\frac{c}{4}-\frac{b}{2})+\frac{4\epsilon_0}{T_1}}+\sqrt{a+\frac{4\epsilon_0}{T_1}}}\\\times \left(\frac{a-b}{\sqrt{T_1}+\sqrt{T_2}}+\frac{a-2b+c}{4(\sqrt{T_1}+\sqrt{T_2})}\left(\sqrt{\frac{T_2}{T_1}}-1\right)\right).
    \end{multline*}By making the change of variables $(w_3,w_4)\mapsto (\xi,\eta)\eqdef(w_3+w_4,w_3-w_4)$, we have $$\frac{\mathcal{P}(T_2,T_1)
    -     \mathcal{P}(T_1)}{T_2-T_1}
    =2^{-3}\iint_{\mathbb{R}^6} d\xi d\eta \  F_\delta(|\eta|^2, \eta\cdot \xi,|\xi|^2)e^{-\frac{1}{2}|\eta|^2-\frac{1}{2}|\xi|^2}.$$
   By writing $\eta $ in the spherical coordinates $(\rho, \phi, \theta) $
with $\cos\theta \eqdef \frac{\xi\cdot \eta}{|\xi||\eta|}$, we further have\begin{multline*}\frac{\mathcal{P}(T_2,T_1)
    -     \mathcal{P}(T_1)}{T_2-T_1}
    =2^{-3}\int_{\mathbb{R}^3} d\xi \int_0^\infty d\rho \int_0^{2\pi}d\phi\int_0^\pi d\theta\ \rho^2\sin\theta  F_\delta(\rho^2, \rho|\xi|\cos\theta,|\xi|^2)e^{-\frac{1}{2}\rho^2-\frac{1}{2}|\xi|^2}\\
 =\pi^2\int_0^\infty dr\int_0^\infty d\rho \int_0^\pi d\theta\ r^2\rho^2\sin\theta  F_\delta(\rho^2, \rho r\cos\theta,r^2)e^{-\frac{1}{2}\rho^2-\frac{1}{2}r^2}. \end{multline*}
Similarly, we have
\begin{multline*}
     \mathcal{P}(T_2,T_1)
     =  4\pi\int_{\mathbb{R}^3}dw_3 \int_\rth dw_4 \sqrt{T_1}R_1 \mathcal{Z}(w_3,0,1)\mathcal{Z}(w_4,0,1)\\
     =\pi^2\int_0^\infty dr\int_0^\infty d\rho \int_0^\pi d\theta\ r^2\rho^2\sin\theta  G_\delta(\rho^2, \rho r\cos\theta,r^2)e^{-\frac{1}{2}\rho^2-\frac{1}{2}r^2},
\end{multline*}and
\begin{multline*}
     \mathcal{P}(T_1)
     =  4\pi\int_{\mathbb{R}^3}dw_3 \int_\rth dw_4 \sqrt{T_1}R_2 \mathcal{Z}(w_3,0,1)\mathcal{Z}(w_4,0,1)\\
     =\pi^2\int_0^\infty dr\int_0^\infty d\rho \int_0^\pi d\theta\ r^2\rho^2\sin\theta  G_0(\rho^2, \rho r\cos\theta,r^2)e^{-\frac{1}{2}\rho^2-\frac{1}{2}r^2},
\end{multline*}
where $$G_\delta(a,b,c)\eqdef 4\pi \sqrt{a+\delta(a-b)+\delta^2(\frac{a}{4}+\frac{c}{4}-\frac{b}{2})+\frac{4\epsilon_0}{T_1}}.$$

Now we also reduce the integrals $\mathcal{A}(T_1;\epsilon_0)$ and $\frac{\mathcal{B}(T_1,T_2)}{T_2-T_1}$.
Recall that they are defined as$$
    \mathcal{A}(T_1;\epsilon_0)\eqdef \int_{\mathbb{R}^3}\left(\frac{|v_3|^2}{2}+\epsilon_0\right)dv_3\int_\rth dv_4 \int_\stw d\omega\ W_+ (v_3,v_4;v_1,v_2)\mathcal{Z}(v_3,0,T_1) \mathcal{Z}(v_4,0,T_1),
$$and
$$ \frac{\mathcal{B}(T_1,T_2)}{T_2-T_1}\eqdef \int_{\mathbb{R}^3}\frac{|v_3|^2}{2}dv_3 \int_\rth dv_4 \int_\stw d\omega\ W_+ (v_3,v_4;v_1,v_2)\mathcal{Z}(v_4,0,T_1) \frac{\mathcal{Z}(v_3,0,T_1)-\mathcal{Z}(v_3,0,T_2)}{T_2-T_1}.$$
Regarding $\mathcal{A}$, we make the change of variables $(v_3,v_4)\mapsto (w_3,w_4)\eqdef \left(\frac{v_3}{\sqrt{T_1}},\frac{v_4}{\sqrt{T_1}}\right)$. Then we note that $\mathcal{A}(T_1;\epsilon_0)$ is the same as $\mathcal{P}(T_1)$ except for the additional weight of $\left(\frac{|v_3|^2}{2}+\epsilon_0\right)$ and obtain
\begin{multline*}
    \mathcal{A}(T_1;\epsilon_0)=4\pi\int_{\mathbb{R}^3}dw_3 \int_\rth dw_4 \left(\frac{T_1|w_3|^2}{2}+\epsilon_0\right)\sqrt{T_1|w_3-w_4|^2+4\epsilon_0}  \mathcal{Z}(w_3,0,1)\mathcal{Z}(w_4,0,1)\\
     =\pi^2\int_0^\infty dr\int_0^\infty d\rho \int_0^\pi d\theta\ r^2\rho^2\sin\theta  A(\rho^2, \rho r\cos\theta,r^2)e^{-\frac{1}{2}\rho^2-\frac{1}{2}r^2},
\end{multline*}
where $A$ is defined as
$$A(a,b,c)\eqdef 4\pi\left(\frac{T_1}{8}(a+2b+c)+\epsilon_0\right) \sqrt{T_1a+4\epsilon_0}.$$
On the other hand, regarding $\mathcal{B}$, we note that 
\begin{multline}\notag    \frac{\mathcal{B}(T_1,T_2)}{T_2-T_1}
=\frac{2\pi}{T_2-T_1} \int_{\mathbb{R}^3}|w_3|^2dw_3 \int_\rth dw_4  \mathcal{Z}(w_4,0,1)\mathcal{Z}(w_3,0,1)\\
\times \left(T_1\sqrt{|\sqrt{T_1}w_3-\sqrt{T_1}w_4|^2+4\epsilon_0}-T_2\sqrt{|\sqrt{T_2}w_3-\sqrt{T_1}w_4|^2+4\epsilon_0}\right)
,
\end{multline}where we made the change of variables 
$(v_3,v_4)\mapsto (w_3,w_4)\eqdef \left(\frac{v_3}{\sqrt{T_1}},\frac{v_4}{\sqrt{T_1}}\right)$ for the first integral and $(v_3,v_4)\mapsto (w_3,w_4)\eqdef \left(\frac{v_3}{\sqrt{T_2}},\frac{v_4}{\sqrt{T_1}}\right)$ for the second integral. By relabelling variables $(w_3,w_4)\leftrightarrows (w_4,w_3), $ we have
\begin{multline}\notag    \frac{\mathcal{B}(T_1,T_2)}{T_2-T_1}
=\frac{2\pi}{T_2-T_1} \int_{\mathbb{R}^3}|w_4|^2dw_4 \int_\rth dw_3  \mathcal{Z}(w_3,0,1)\mathcal{Z}(w_4,0,1)\left(T_1\sqrt{T_1}R_2-T_2\sqrt{T_1}R_1\right)\\
=2\pi \int_{\mathbb{R}^3}|w_4|^2dw_4 \int_\rth dw_3  \mathcal{Z}(w_3,0,1)\mathcal{Z}(w_4,0,1) \left[-\sqrt{T_1}R_2-\frac{T_2}{T_2-T_1}\sqrt{T_1}(R_1-R_2)\right]
\eqdef -\mathcal{B}_1-\mathcal{B}_2
,
\end{multline}where $R_1$ and $R_2$ are defined in \eqref{def.R1R2}.
Now note that the difference between $\mathcal{P}(T_2,T_1)$ and $\mathcal{B}_1$ is on the additional weight of $|w_4|^2$ and the change of the kernel from $R_1$ to $R_2$ only. Therefore, we obtain
\begin{equation*}
  \mathcal{B}_1
     =\pi^2\int_0^\infty dr\int_0^\infty d\rho \int_0^\pi d\theta\ r^2\rho^2\sin\theta  B_1(\rho^2, \rho r\cos\theta,r^2)e^{-\frac{1}{2}\rho^2-\frac{1}{2}r^2},
\end{equation*}
where $$B_1(a,b,c)\eqdef 2\pi \left(\frac{a}{4}+\frac{c}{4}-\frac{b}{2}\right)\sqrt{a+\frac{4\epsilon_0}{T_1}}.$$ For $\mathcal{B}_2$, we also note that the difference between $\frac{\mathcal{P}(T_2,T_1)-\mathcal{P}(T_1)}{T_2-T_1}$ and $\mathcal{B}_2$ is on the additional weight of $\frac{T_2|w_4|^2}{2}$ only. Therefore, we have
$$
    \mathcal{B}_2=\pi^2\int_0^\infty dr\int_0^\infty d\rho \int_0^\pi d\theta\ r^2\rho^2\sin\theta  B_{2,\delta}(\rho^2, \rho r\cos\theta,r^2)e^{-\frac{1}{2}\rho^2-\frac{1}{2}r^2},$$
where the kernel $B_{2,\delta}$ is defined as
\begin{multline*}
    B_{2,\delta}\eqdef \frac{2\pi T_2(\frac{a}{4}+\frac{c}{4}-\frac{b}{2})}{\sqrt{a+\delta(a-b)+\delta^2(\frac{a}{4}+\frac{c}{4}-\frac{b}{2})+\frac{4\epsilon_0}{T_1}}+\sqrt{a+\frac{4\epsilon_0}{T_1}}} \left(\frac{a-b}{\sqrt{T_1}+\sqrt{T_2}}+\frac{a-2b+c}{4(\sqrt{T_1}+\sqrt{T_2})}\left(\sqrt{\frac{T_2}{T_1}}-1\right)\right). 
\end{multline*}
Thus, we have$$
    \frac{\mathcal{B}(T_1,T_2)}{T_2-T_1}=-\pi^2\int_0^\infty dr\int_0^\infty d\rho \int_0^\pi d\theta\ r^2\rho^2\sin\theta  (B_1+B_{2,\delta})(\rho^2, \rho r\cos\theta,r^2)e^{-\frac{1}{2}\rho^2-\frac{1}{2}r^2}.$$
 This completes the reduction of 6-fold integrals to 3-fold integrals.

\bibliographystyle{amsplain3links}
\bibliography{bibliography.bib}{}

 \end{document}